\RequirePackage{tikz}
\documentclass[sn-mathphys]{sn-jnl}
\pdfoutput=1
\newif\iflong
\longtrue % for version on the Web
\usepackage{xcolor}
\usepackage{epsfig}
\usepackage[final]{changes}

\title{Persistent homology of \replaced{partially ordered spaces}{directed spaces}}
\author[1]{Cameron Calk}\email{cameron.calk@lis-lab.fr}
\affil[1]{Laboratoire d'Informatique et des Syst\`emes, Universit\'e Aix-Marseille, \orgaddress{\city{Marseille}, \postcode{13397 Cedex}, \country{France}}}
\author[2]{Eric Goubault}\email{eric.goubault@polytechnique.edu}
\affil[2]{\orgdiv{LIX}, \orgname{CNRS, Ecole Polytechnique, Institut Polytechnique de Paris}, \orgaddress{\city{Palaiseau}, \postcode{91128 Cedex}, \country{France}}}
\author[3]{Philippe Malbos}\email{malbos@math.univ-lyon1.fr}
\affil[3]{\orgname{Universit\'e Claude Bernard Lyon 1, CNRS, Institut Camille Jordan, UMR5208},
\orgaddress{\city{Villeurbanne}, \postcode{69622 Cedex}, \country{France}}}
\date{}
\usepackage{fancyhdr,titlesec,url}
\usepackage{amsmath,amssymb,amsthm,bbm}
\usepackage[all,knot,poly]{xy}
\usepackage{hyperref,enumerate}
\usepackage{tikz,tikz-3dplot}
\tdplotsetmaincoords{70}{115}
\usepackage{epsfig}
\usepackage{graphicx}
\usetikzlibrary{matrix, shapes, snakes}
\usepackage{mathbbol}
\newcommand{\auteur}[3]{
\noindent
\begin{minipage}[t]{.45\textwidth}
\begin{flushright}
\textsc{#1} \\
{\footnotesize\textsf{#2}}
\end{flushright} 
\end{minipage}
\qquad
\begin{minipage}[t]{.45\textwidth}
#3
\end{minipage}
}

\newcount\hh
\newcount\mm
\mm=\time
\hh=\time
\divide\hh by 60
\divide\mm by 60
\multiply\mm by 60
\mm=-\mm
\advance\mm by \time
\def\hhmm{\number\hh:\ifnum\mm<10{}0\fi\number\mm}
\newcommand\map[3]{#1 : #2 \rightarrow #3}
\newtheorem{lemma}{Lemma}
\newtheorem{theorem}{Theorem}
\newtheorem{proposition}{Proposition}

\newtheorem{example}{Example}

\newtheorem{question}{Question}
\def\Re{\mathbb{R}}
\def\K{\mathbb{K}}
\def\k{\mathbb{k}}

\newcommand\ForAuthors[1]%          %  temporary remark for the
 {\par\smallskip                     %  authors:
  \begin{center}%                    %
   \fbox%                            %    --------
   {\parbox{0.9\linewidth}%          %    |  #1  |
    {\raggedright--- #1}%         %    --------
   }%                                %
  \end{center}%                      %
  \par\smallskip                     %
 }

\input{macrosPersAPCT}

\definecolor{pmcolor}{rgb}{1,0.49,0}
\definecolor{egcolor}{cmyk}{1,0,0,0}
\definecolor{cccolor}{cmyk}{0,1,0,0}

\definecolor{vert}{rgb}{0,0.45,0}
\definechangesauthor[name=PM,color=pmcolor]{PM}
\definechangesauthor[name=EG,color=egcolor]{EG}
\definechangesauthor[name=CC,color=cccolor]{CC}

\begin{document}

%\begin{center}
%\begin{small}\begin{minipage}{14cm}
%\noindent\textbf{Abstract --}
\abstract{In this work, we explore links between natural homology and persistent homology for the classification of directed spaces. The former is an algebraic invariant of directed spaces, a semantic model of concurrent programs. The latter was developed in the context of topological data analysis, in which topological properties of point-cloud data sets are extracted while eliminating noise. In both approaches, the evolution \added{of} homological properties \replaced{is}{are} tracked through a sequence of inclusions of usual topological spaces. Exploiting this similarity, we show that natural homology may be considered a persistence object, and may be calculated as a colimit of uni-dimensional persistent homologies along traces. \deleted{Finally, we suggest further links and avenues of future work in this direction.}}
%\medskip

%\smallskip\noindent

%\smallskip\noindent\textbf{M.S.C. 2020 --} 55N31, 68Q55, 68Q85.
%\end{minipage}
%\end{small}
%\end{center}

%\vskip+20pt

%\tableofcontents

\maketitle

%\end{document}

\section{Introduction}

%\new{mettre en forme la partie en vert + à relire + rédiger la partie annonce des résultas + organisation of the paper}

%\pmin{New title? : Persistent homology of directed spaces and concurrency / Persistent homology of directed spaces in concurrency / Persistent homology of concurrent systems}

\replaced
{
Geometry, algebraic topology and homological algebra have become part of the computer science landscape, thanks to a number of significant applications in recent years. Topological data analysis aims to analyze data sets using techniques from topology \cite{Carlsson09}. One of its main tools is persistent homology, a homological theory that provides efficient algorithms for the analysis of point cloud data which emerged simultaneously in several works in the late 1990s \cite{Barannikov1994,FrosiniLandi99,EdelsbrunnerLetscherZomorodian02}.
}
{Geometry, algebraic topology, and homological algebra have now been in the computer science landscape through many significant applications in recent years.
 In topological data analysis, the shape of point-cloud data can be hinted at through suitable homological invariants, persistence homology (see e.g. \cite{EdelsbrunnerLetscherZomorodian02}, and \cite{Carlsson09} for a survey of the earlier days of persistence).}
\deleted{These invariants capture the essential topological features of the point-cloud data, in that these are independent of the metrics used, are robust to noise and compact in their presentation.}
In the same period, similar ideas appeared in the realm of semantics of programming languages, in particular in concurrency theory \cite{GeomConc,gouthe} and distributed computing~\cite{herlihy1999topological}, see e.g. \cite{marco2013homological,fajstrup16,herlihy} for surveys.

This article makes a formal bridge between these two approaches and illustrates the interest of applying \replaced{methods from persistent homology}{persistent homology methods} to problems such as the classification of potential behaviors in concurrency theory and distributed computing through motivational examples. 
%HYP

The semantic models that describe possible executions of concurrent programs are based on the notion of \emph{directed space}, which is a topological space $X$ equipped with a \replaced{set}{topological space} $dX$ of \emph{directed paths}, \emph{i.e.} maps from the unit interval to $X$ which enjoy a number of properties, see Section \ref{SS:dspaces} and \cite{grandisbook,fajstrup16}. For the purpose of this article, we will be exemplifying our results on directed spaces which are generated by \emph{compact partially-ordered spaces} \added{(\emph{pospaces})}, which will furthermore be (directed) geometric realizations of finite precubical sets, see e.g. \cite{fajstrup16}. A compact partially-ordered space is a compact topological space $X$, together with a partial ordering $\leq$ which is a closed subset of $X\times X$ \replaced{for}{with} the product topology. The directed paths in $X$ are continuous and increasing maps from the unit interval, with the standard total order, to $X$. \replaced{The}{A} particular class of compact partially-ordered spaces \added{we will be considering here}{} is given by the geometric realization of \replaced{loop-free, non-self linked}{(loop-free)}{} finite precubical sets, in which all $n$-cubes are realized by the unit $n$-cube in $\Re^n$, with the componentwise ordering \added{\cite{Faj06}}{}. 

\replaced{
Let us now consider the problem of classifying directed spaces, i.e. determining when two directed spaces are dihomeomorphic, as defined in Section~\ref{SS:dspaces}, that is homeomorphic as topological spaces with additional conditions concerning the directed structure.}{
Consider now the problem of classifying such directed spaces, meaning determining when two directed spaces are \replaced{``}{"}the same", that is, dihomeomorphic, as defined in Section \ref{SS:dspaces}, meaning homeomorphic as topological spaces plus some extra condition implying preservation of the directed structure.}
From the concurrency theory point of view, having dihomeomorphic directed spaces means essentially having the same coordination between concurrent processes~\cite{fajstrup16}.
%, whereas non-dihomeomorphic directed spaces are the semantics of different concurrent programs. 

As for the classification problem of topological spaces, it is far too hard a problem as such, and we need tractable (directed) topological invariants to give witnesses of the non-existence of dihomeomorphisms. One such invariant, called \emph{natural homology}, is based on homology, as generally done in classical algebraic topology, and has been developed in \cite{naturalhomology}{, see Section~\ref{SS:NaturalHomology}}. The idea is to observe that a dihomeomorphism {$\map{f}{X}{Y}$} induces, for all points $\alpha,\beta$ in $X$, an isomorphism $f^*$ between sub-spaces {$dX(\alpha,\beta)$} of dipaths in $X$ from $\alpha$ to $\beta$ (see Section \ref{SS:dspaces}) \replaced{and}{in with} sub-spaces $dY(f(\alpha),f(\beta))$ of dipaths in $Y$ from $f(\alpha)$ to $f(\beta)$. The natural homology of directed space $X$ is a diagram (or a \replaced{``}{"}natural system" of groups, \cite{BauesWirsching}) combining all (classical) homology groups of {$dX(\alpha,\beta)$} for $\alpha$ and $\beta$ varying over $X$, with induced \replaced{``}{"}extension maps" between them. {When $f$ is a dihomeomorphism, the graph of the homology of $f^*$ is a bisimulation between the {natural homology} of $X$ and the natural homology of $Y$ (Lemma \ref{lem:dihomeobisim})}.

Consider for instance the two compact partially-ordered spaces within $\mathbb{R}^2$ with the componentwise ordering pictured in Figure \ref{fig:0}. They give semantics, for the left one $X$, to program $U.S \mid \mid U.S$, and for the one on the right $Y$, to the program $U.S \mid \mid S.U$, where $S$ stands for a scan operation and $U$ stands for an update operation in shared memory, see e.g. \cite{geomsem}. As topological spaces, $X$ and $Y$ are homeomorphic, homotopically equivalent to a wedge of two circles. 
Still, they are not dihomeomorphic directed spaces. 
Indeed, for $\alpha,\beta$ in $X$, $dX(\alpha,\beta)$ can only be homotopic to a point, two points\added{,} or three points, whereas for $\gamma,\delta$ in $Y$, $dY(\gamma,\delta)$ can only be homotopic to a point, two points\added{,} or four points. 
All these points do indeed correspond to traces of executions of the respective concurrent programs. These traces describe distinct coordination properties between two processes, and may give different outcomes. 
Figure \ref{fig:0} shows the three maximal directed paths up to directed homotopy for the space on the left, and the four maximal ones in the space on the right.

\begin{figure}
%\centering
\begin{center}
\begin{tikzpicture}[auto,scale = 1]
%\draw (0,0) rectangle (1.5,1.5);
\draw (1.5,1.5) rectangle (4,4);
%\draw [fill = gray!50,draw = gray!50] (0.5,0.5) rectangle (1,1);
\draw [fill = gray!50,draw = gray!50] (2,3) rectangle (2.5,3.5);
\draw [fill = gray!50,draw = gray!50] (3,2) rectangle (3.5,2.5);
%\draw [dotted] (1,0) to (1,1.5);
%\draw [dotted] (0.5,0) to (0.5,1.5);
%\draw [dotted] (0,1) to (1.5,1);
%\draw [dotted] (0,0.5) to (1.5,0.5);
\draw [dotted] (2,1.5) to (2,4);
\draw [dotted] (2.5,1.5) to (2.5,4);
\draw [dotted] (3,1.5) to (3,4);
\draw [dotted] (3.5,1.5) to (3.5,4);
\draw [dotted] (1.5,2) to (4,2);
\draw [dotted] (1.5,2.5) to (4,2.5);
\draw [dotted] (1.5,3) to (4,3);
\draw [dotted] (1.5,3.5) to (4,3.5);
%\draw (0.75,-0.1) to (0.75,0.1);
%\draw (-0.1,0.75) to (0.1,0.75);
\draw (2.25,1.4) to (2.25,1.6);
\draw (3.25,1.4) to (3.25,1.6);
\draw (1.4,2.25) to (1.6,2.25);
\draw (1.4,3.25) to (1.6,3.25);
%\node (S11) at (0.75,-0.3) {\footnotesize{S}};
%\node (U21) at (-0.3,0.75) {\footnotesize{U}};
\node (U11) at (2.25,1.2) {\footnotesize{U}};
\node (S12) at (3.25,1.2) {\footnotesize{S}};
\node (U22) at (1.2,2.25) {\footnotesize{U}};
\node (S21) at (1.2,3.25) {\footnotesize{S}};
%\draw [thick] (0,0) .. controls (0.3,1.5)  .. (1.5,1.5);
%\draw [thick] (0,0) .. controls (1.5,0.3)  .. (1.5,1.5);
\draw [thick] (1.5,1.5) to (4,4);
\draw [thick] (1.5,1.5) .. controls (1.6,4)  .. (4,4);
\draw [thick] (1.5,1.5) .. controls (4,1.6)  .. (4,4);
\end{tikzpicture}
\begin{tikzpicture}
\draw[color = white] (0,-1) rectangle (0.5,-1.1);
\end{tikzpicture}
\qquad
\begin{tikzpicture}[auto,scale = 1]
%\draw (0,0) rectangle (1.5,1.5);
\draw (1.5,1.5) rectangle (4,4);
%\draw [fill = gray!50,draw = gray!50] (0.5,0.5) rectangle (1,1);
\draw [fill = gray!50,draw = gray!50] (3,3) rectangle (3.5,3.5);
\draw [fill = gray!50,draw = gray!50] (2,2) rectangle (2.5,2.5);
%\draw [dotted] (1,0) to (1,1.5);
%\draw [dotted] (0.5,0) to (0.5,1.5);
%\draw [dotted] (0,1) to (1.5,1);
%\draw [dotted] (0,0.5) to (1.5,0.5);
\draw [dotted] (2,1.5) to (2,4);
\draw [dotted] (2.5,1.5) to (2.5,4);
\draw [dotted] (3,1.5) to (3,4);
\draw [dotted] (3.5,1.5) to (3.5,4);
\draw [dotted] (1.5,2) to (4,2);
\draw [dotted] (1.5,2.5) to (4,2.5);
\draw [dotted] (1.5,3) to (4,3);
\draw [dotted] (1.5,3.5) to (4,3.5);
%\draw (0.75,-0.1) to (0.75,0.1);
%\draw (-0.1,0.75) to (0.1,0.75);
\draw (2.25,1.4) to (2.25,1.6);
\draw (3.25,1.4) to (3.25,1.6);
\draw (1.4,2.25) to (1.6,2.25);
\draw (1.4,3.25) to (1.6,3.25);
%\node (S11) at (0.75,-0.3) {\footnotesize{S}};
%\node (U21) at (-0.3,0.75) {\footnotesize{U}};
\node (U11) at (2.25,1.2) {\footnotesize{U}};
\node (S12) at (3.25,1.2) {\footnotesize{S}};
\node (U22) at (1.2,2.25) {\footnotesize{S}};
\node (S21) at (1.2,3.25) {\footnotesize{U}};
%\draw [thick] (0,0) .. controls (0.3,1.5)  .. (1.5,1.5);
%\draw [thick] (0,0) .. controls (1.5,0.3)  .. (1.5,1.5);
%\draw [thick] (1.5,1.5) to (4,4);
\draw [thick] (1.5,1.5) .. controls (1.7,2.75)  .. (2.75,2.75);
\draw [thick] (1.5,1.5) .. controls (2.75,1.7)  .. (2.75,2.75);
\draw [thick] (4,4) .. controls (2.75,3.8)  .. (2.75,2.75);
\draw [thick] (4,4) .. controls (3.8,2.75)  .. (2.75,2.75);
%\draw [thick] (0,0) .. controls (0.2,1.25)  .. (1.25,1.25);
%\draw [thick] (0,0) .. controls (1.25,0.2)  .. (1.25,1.25);
%\draw [thick] (2.5,2.5) .. controls (1.25,2.3)  .. (1.25,1.25);
%\draw [thick] (2.5,2.5) .. controls (2.3,1.25)  .. (1.25,1.25);
\draw [thick] (1.5,1.5) .. controls (1.6,4)  .. (4,4);
\draw [thick] (1.5,1.5) .. controls (4,1.6)  .. (4,4);
\end{tikzpicture}
\begin{tikzpicture}
\draw[color = white] (0,-1) rectangle (0.5,-1.1);
\end{tikzpicture}
\caption{Two essentially different concurrent programs with homeomorphic state spaces.}
\label{fig:0}       % Give a unique label
\end{center}
\end{figure}

Even more importantly, for any directed path $u$ from $\alpha'$ to $\alpha$ in $X$ and any path $v$ from $\beta$ to $\beta'$ in $X$, we have a continuous map $dX\langle u,v\rangle$ from $dX(\alpha,\beta)$ to $dX(\alpha',\beta')$ by pre-composition by $u$ and post-composition by $v$. 
This is illustrated in a simplified form below, in which the red arrows represent elements of $dX(\alpha,\beta)$:
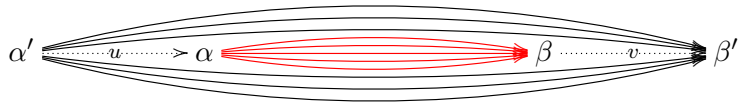
\vspace{.25cm}
\[
\xymatrix@C=5em{
\alpha'
    \ar@{..>}[r] |-{u}
    \ar@/^1.5pc/[rrrr]
    \ar@/^1.125pc/[rrrr]
    \ar@/^.75pc/[rrrr]
    \ar@/_.75pc/[rrrr]
    \ar@/_1.125pc/[rrrr]
    \ar@/_1.5pc/[rrrr]
&
\alpha
    \ar@[red]@/^.5pc/[rr]
    \ar@[red]@/^.25pc/[rr]
    \ar@[red][rr]
    \ar@[red]@/_.25pc/[rr]
    \ar@[red]@/_.5pc/[rr]
&
&
\beta
 \ar@{..>}[r] |-{v}
&
\beta'
}
\]

\medskip

In other words, each pair of paths $(u:\alpha' \fl \alpha, v : \beta \fl \beta')$ determine a different continuous inclusion from $dX(\alpha,\beta)$ to $dX(\alpha',\beta')$.
% (see Section \ref{SS:NaturalHomology} for a recap). %\todo{Proof to be given in this paper, we actually never wrote down that observation! Also, what can we say in terms of interleaving distance?}
%$f^* dX\langle u,v\rangle=dY\langle f^*(u),f*(v)\rangle$\todo{Defines "bisimulation"?}. 
We will recap this notion in Section \ref{sec:natural}. 
The study of the shape of such spaces of dipaths, when moving $\alpha$ and $\beta$, will be the essential ingredient of invariants under directed homeomorphisms.

Now, the idea of persistent homology comes from a different line, but we will see that it is based on similar intuitions. 
%\todo{rjouter une phrase avec filtrations, \dots}
In the classical``uni-dimensional'' approach to persistent homology of a point cloud in a space $\mathbb{R}^n$, we construct a filtration, i.e. a sequence of inclusions of simplicial sets. Typically, this starts on the point cloud and develops by adding relationships between the initial points resulting in a $\mathbb{R}$-persistence simplicial complex, \added{a sequence of inclusions of simplicial complexes indexed over the reals,} from which we extract a sequence of homological invariants. This sequence will witness the appearances, or births, of homological generators, as well as their disappearances, or deaths, as the filtration develops. The method of filtration depends on the nature and characteristics of the point \replaced{cloud,}{cloud ,} according to which one can consider \v{C}ech, Vietoris-Rips, witness filtrations\replaced{, etc}{...}~\cite{Carlsson09}. In general these filtrations are unidimensional but there exist multidimensional analogues \cite{carzom2}, in which multiple parameters can vary, which we will briefly discuss later on. %\pmin{mettre en cohérence les citations Carlsson 2009 : il y a trois fois la même}
%\todo{there exist multidimentional version : but it is not the question here}

%\New{
%In the classical approach of persistent homology, the filtration starts on a point cloud and develops by adding relationships between the initial points, and then the simplexes built as the evolution proceeds. \todo{on obtient ainsi, un complexe simplicial, dont on calcule l'homologie}

%In this work, we apply this philosophy (données->objet combinatoires filtrés selon une direction (poset persistance)->invariant linéaires) to concurrent programs described by directed paths. 

%dans ce contexte, les données sur lesquelles on veut étudier via persitance homologique, ne sont pas des points, mais des traces maximales dans des espaces dirigés. Ces traces maximales correspondents à des executions...

%La structure dirigée de l'espace, permet pour chaque trace de définir plusieurs filtrations unidimensionnelles de l'espace de traces associé. L'objectif est d'amalgamer toutes ces filtrations pour avoir obtenir les propriétés topologiques de l'espaces des traces.

%Contrairement, à la situation classique, ces filtrations ne débutent pas par l'espace que l'on veut étudier, mais y arrive car leur amalgamation donne l'espace des traces tout entier. De plus, dans le cas classique, la poset-persistance sous jacente au fibrations est obtenue par des méthodes externes de filtrations. ICI, la filtration est inhérente à l'objet d'étude.

%} 

Similarly to persistent homology, there is a natural notion of birth and death of homological generators within natural homology. 
\replaced{This leads us}{It is therefore natural} to study the relationship between natural homology and persistent homology. \replaced{For instance, when $X$ is a cubical complex as in \cite{Eilenberg}}{For $X$ a Euclidean cubical complex}{}, which can be embedded as a sub-directed space of $\Re^n$ for some high enough $n \in \N$, endowed with the component-wise ordering, it is natural to consider multidimensional persistence on $2n$ parameters, the $n$ coordinates of the start point $\alpha$ and the $n$ coordinates of the end point $\beta$ and study the evolution of $dX(\alpha,\beta)$. In Section~\ref{SS:MatchboxExample}, we show that this point of view does not quite work\added{,} using a simple example. 
%\todo{Idea of death and birth of homology}

\bigskip

%\todo{ICI}

The objective of this article is to cope with this difficulty, in order to give a meaning to natural homology as persistent homology. The shift of point of view we are going to make is similar to recent approaches in multidimensional persistence \cite{carzom2}. We will look at all unidimensional persistent homologies that are compatible with the structure of a directed space, and glue this information together. In many ways, this resembles ``probing" approaches in multidimensional homology, see e.g. \cite{curry2021directions}. 

%\todo{Mention somewhere already the landscapes of \cite{Bubenik}?}

In this work, we apply this idea %(données->objet combinatoires filtrés selon une direction (poset persistance)->invariant linéaires) 
to concurrent programs described by directed spaces. %, where we replace data by paths in a directed space, suitably filtered according to some flow of time. 
Here, the data which we want to study via homological persistence are not points, but traces in directed spaces, \ie images of directed paths, which correspond to all observable executions of some concurrent program. 

The directed structure of the space allows for each trace to define several one-dimensional filtrations of the associated trace space. The objective is to amalgamate all of these filtrations to obtain directed \replaced{algebro-topological}{algebraico-topological} properties of the considered directed space.
%Contrary to the classical situation, these filtrations do not start with the space we want to study, but amalgamate together to give the whole trace space. Moreover, 
Contrary to the classical case, wherein the poset-persistence underlying the filtrations is obtained by external methods, the filtrations in the directed framework are inherent to the object of study.

Apart from the theoretical bridge between seemingly different lines of work, this work is a first step toward providing more tractable ways of computing natural homology (such as rank invariants \cite{carzom2}), as well as giving succinct descriptions of the semantics of concurrent and distributed systems, as some sort of barcodes (see e.g. \cite{lesnick2015interactive}), an approach which has proven very practical in different areas of engineering, see e.g. \cite{Carlsson09}. This will be developed elsewhere.
%(from \cite{Eilenberg}, same argument as for natural homology)

%\todo{Example of utility of "barcodes" here? (or later, after the technical sections?) : 
%- another interesting example is the matchbox, which is not going to be "equivalent" to a simple mutual exclusion nor a 2-semaphore. 
%- Helix example? Papadimitriou example? 
%- Linearisability? (2 phase locking?)
%} 

\bigskip

Our first result, Proposition~\ref{P:FactisPos}, provides the basis for this work. Indeed, it states that in the case of a partially ordered space $\Xr$, \deleted{which constitute a prevalent example of directed space,} there exists an order relation on traces which is isomorphic to the factorisation category \added{of} $\Xr$, which we afterwards call the trace poset $\tpos\Xr$ of $\Xr$. The latter is the domain of the natural homology functor, so in particular, Proposition~\ref{P:FactisPos} states that natural homology is a \replaced{persistence}{peristence} vector space.

Theorem~\ref{t:maxchains} characterises maximal chains in the trace poset of a compact partially ordered space, and Proposition~\ref{prop:1} shows that parametrisations of (maximal) traces yield unidimensional persistence modules along (maximal) chains by restriction of the natural homology functor.

Finally, in order to link these unidimensional persistence modules to the full natural homology functor, we prove Propositions~\ref{prop:colimchain} and~\ref{prop:colimmaxchain}, stating that an arbitrary poset may be obtained as the colimit of the diagram of its (maximal) chains. We then apply this to functors whose domain is a poset, showing in Proposition~\ref{Proposition:Colimits} that the colimit construction along chains also applies to such functors.

All of this leads to our main result, Theorem~\ref{thm:1}, stating that natural homology is obtained as the colimit of the unidimensional persistence modules along (maximal) traces, thereby establishing a concrete theoretical link between these two homology theories.

%\todo{more in Section \ref{sec:remarks}?}

\bigskip

\noindent \textbf{Outline of the article.} %\pmin{reprendre cette partie à la fin}

The first section presents a motivational example \replaced{to introduce}{in order to present} the relationship between persistent homology and natural homology. In particular, we argue that there is no canonical (multidimensional) filtration of the trace spaces of a given directed space which is solely based on end-points of dipaths. Indeed, we show that one must consider filtrations given by inclusions of trace spaces induced by pre- or post-concatenation with dipaths.

Section~\ref{S:Sketch} deals with basic notions of persistent homology, fixing the notations and making the article self-contained. We recall the definitions of persistence objects and persistent homology, as well as an algorithm for computing persistent homology, more details on which can be found in~\cite{Carlsson09,EdelsbrunnerHarer2008,EdelsbrunnerHarer2010}.

Section \ref{sec:natural} recalls the main definitions of directed topology and homological algebra used in this work. In Sections \ref{SS:dspaces} and~\ref{sec:tracecat}, we recall the structure of directed space, while the definition of natural homology is recalled in \ref{SS:NaturalHomology}. Before moving on to the main results of the article, we recall in \replaced{Section~\ref{sec:bisimulation}}{Section\ref{sec:bisimulation}} the definition of bisimulation, which constitutes the appropriate notion of equivalence of natural homology modules.

The first results of the paper are shown in Section~\ref{sec:persdir}. First, we show by way of Proposition~\ref{P:FactisPos} that in the case of a partially ordered space $\Xr$, the domain of the associated natural homology functor is in fact a poset $\tpos \Xr$, thus allowing us to consider it as a persistence object. In Section \ref{sec:persistenttrace}, we define a unidimensional persistence module along traces in a pospace, before proving Theorem~\ref{t:maxchains} and Proposition~\ref{prop:1}, relating (maximal) chains in the poset $\tpos\Xr$ to (maximal) traces in $\Xr$.
These concepts are illustrated in Section \ref{ex:persmatchbox} on the motivational example from Section~\ref{S:Sketch}, the matchbox, for which full calculations are made. The relation to the natural homology of the matchbox is described in Section~\ref{sec:nathommatchbox}.

In Section~\ref{sec:simulpers}, we show how colimit constructions permit us to glue together the unidimensional persistence modules obtained along traces in order to recover the entire natural homology diagram of a given pospace. First, in Section~\ref{ss:colimitschains}, we recall folkloric results about posets which provide various ways of reconstructing a poset as a colimit of its (maximal) chains. In Section~\ref{sec:applidiag}, we show via Proposition~\ref{Proposition:Colimits} that these constructions carry over to functors whose domains are posets. Finally, putting this together with results from Section~\ref{sec:persdir}, we obtain the main result of this article in Section~\ref{sec:nathomupset}, namely Theorem~\ref{thm:1}, stating that the natural homology of a pospace, or certain subdiagrams thereof, can be obtained from the uni-dimensional persistence modules along traces.

Finally, Section~\ref{sec:remarks} finishes up by giving hints about extensions of these results, further linking natural homology with persistence when the considered partially ordered spaces are equipped with some metric structure in Section~\ref{ss:MetricStructure}, and briefly touching on algorithmic considerations in Section~\ref{SS:algo}. 

\section{A motivational example}
\label{SS:MatchboxExample}

To give a glimpse of the intimate relationship between multidimensional persistence
and natural homology, let us describe our construction on the following example, namely Fahrenberg's matchbox \cite{fahdihom}:

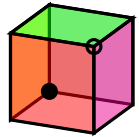
\medskip

\begin{center}
\begin{tikzpicture}[scale=1.5,tdplot_main_coords]
\begin{scope}[very thick, every node/.style={sloped,allow upside down}]
    \coordinate (O) at (0,0,0);
    \tdplotsetcoord{P}{1.414213}{54.68636}{45}

    \draw[fill=gray!50,fill opacity=0.5] (O) -- (Py) -- (Pyz) -- (Pz) -- cycle;
    \draw[-,thick,fill=gray!50,fill opacity=0.5] (O) -- (Py);
    \draw[-,thick,fill=gray!50,fill opacity=0.5] (Py) -- (Pyz);
    \draw[-,thick,fill=gray!50,fill opacity=0.5] (O) -- (Pz);
    \draw[-,thick,fill=gray!50,fill opacity=0.5] (Pz) -- (Pyz);

%    \draw[fill=blue,fill opacity=0.5] (O) -- (Px) -- (Pxy) -- (Py) -- cycle;
    \draw[fill=yellow,fill opacity=0.5] (O) -- (Px) -- (Pxz) -- (Pz) -- cycle;
    \draw[-,thick,fill=yellow,fill opacity=0.5] (O) -- (Px);
    \draw[-,thick,fill=yellow,fill opacity=0.5] (Px) -- (Pxz);
    \draw[-,thick,fill=yellow,fill opacity=0.5] (Pz) -- (Pxz);
    \draw[-,thick,fill=yellow,fill opacity=0.5] (O) -- (Pz);

    \draw[fill=green,fill opacity=0.5] (Pz) -- (Pyz) -- (P) -- (Pxz) -- cycle;
    \draw[-,thick,fill=green,fill opacity=0.5] (Pz) -- (Pyz); 
    \draw[-,thick,fill=green,fill opacity=0.5] (Pyz) -- (P); 
    \draw[-,thick,fill=green,fill opacity=0.5] (Pz) -- (Pxz);
    \draw[-,thick,fill=green,fill opacity=0.5] (Pxz) -- (P);

    \draw[fill=red,fill opacity=0.5] (Px) -- (Pxy) -- (P) -- (Pxz) -- cycle;
    \draw[-,thick,fill=red,fill opacity=0.5] (Px) -- (Pxy);
    \draw[-,thick,fill=red,fill opacity=0.5] (Pxy) -- (P);
    \draw[-,thick,fill=red,fill opacity=0.5] (Pxz) -- (P);
    \draw[-,thick,fill=red,fill opacity=0.5] (Px) -- (Pxz);

    \draw[fill=magenta,fill opacity=0.5] (Py) -- (Pxy) -- (P) -- (Pyz) -- cycle;
    \draw[-,thick,fill=magenta,fill opacity=0.5] (Py) -- (Pxy);
    \draw[-,thick,fill=magenta,fill opacity=0.5] (Pxy) -- (P);
    \draw[-,thick,fill=magenta,fill opacity=0.5] (Py) -- (Pyz);
    \draw[-,thick,fill=magenta,fill opacity=0.5] (Pyz) -- (P);

    \filldraw (O) circle [radius=1.75pt];
    \draw (P) circle [radius=1.75pt];
\end{scope}
  \end{tikzpicture}
  \end{center}
\medskip
  
The faces of the unit cube $[0,1]\times [0,1] \times [0,1]$, except the bottom face, constitute the associated cubical complex $K$, \ie there are 5 squares glued together.
Dipaths in the matchbox are given by continuous maps which are increasing in each coordinate, the order being given by considering the disk in the figure as the origin, and the empty circle as $(1,1,1)$. To simplify, in the rest of this example we will only consider so-called combinatorial dipaths, \ie those which follow the black lines in the above diagram. 
A \emph{trace} is the equivalence class $\overline p$ of a
dipath $p$ modulo reparametrisation equivalence as defined in Section~\ref{SS:dspaces}, 
%Concatenation induces an associative operation $\star$ on the quotient space of traces, 
and the set of such equivalence classes can be given the structure of a topological space $\trace K a b$ for all start (resp. end) points $a$ (resp. $b$), homotopic to a CW-complex for a large class of geometric realizations of pre-cubical sets.

Using Ziemia\'nski's construction~\cite{ziemianski}, the simplicial set corresponding
to its trace space from beginning to end is as follows 
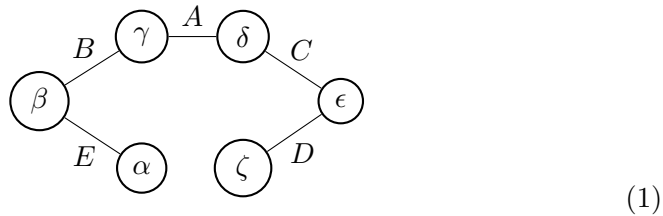
\begin{equation}
\label{E:SimplicialSet}
%\raisebox{1cm}{
\begin{aligned}
\begin{tikzpicture}[auto,line/.style={draw,thick,-latex',shorten >=2pt},cloudgrey/.style={draw=black,thick,circle,fill={rgb:black,1;white,2},minimum height=1em},cloud/.style={draw=black,thick,circle,fill=white,minimum height=1em}]
\matrix[column sep=6mm,row sep=1mm]
{
& \node [cloud] (gamma) {$\gamma$}; & \node [cloud] (delta) {$\delta$}; & \\
\node [cloud] (beta) {$\beta$}; & & & \node [cloud] (epsilon) {$\epsilon$}; \\
& \node [cloud] (alpha) {$\alpha$}; & \node [cloud] (zeta) {$\zeta$}; & \\
};
\draw (gamma) -- node[above] {$A$} (delta);
\draw (beta) -- node[above] {$B$\;\;} (gamma);
\draw (delta) -- node[above] {\;\;$C$} (epsilon);
\draw (alpha) -- node[below] {$E$\;\;} (beta);
\draw (zeta) -- node[below] {\;\;$D$} (epsilon);
\end{tikzpicture}
\end{aligned}
\end{equation}
where the edges $A$, $B$, $C$, $D$ and $E$ correspond to the following 
five 2-dimensional cubical paths:

\medskip

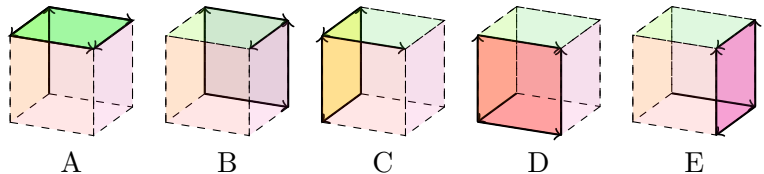
\begin{center}
\begin{tabular}{ccccc}
  \begin{tikzpicture}[scale=1.5,tdplot_main_coords]
    \coordinate (O) at (0,0,0);
    \tdplotsetcoord{P}{1.414213}{54.68636}{45}
    \draw[->,thick,fill=green,fill opacity=0.3] (Pz) -- (Pyz) -- (P) -- (Pxz) -- cycle;
\draw [->,thick] (O) -- (Pz);
    \draw[->,thick,fill=green,fill opacity=0.3] (Pz) -- (Pyz); 
    \draw[->,thick,fill=green,fill opacity=0.3] (Pyz) -- (P); 
    \draw[->,thick,fill=green,fill opacity=0.3] (Pz) -- (Pxz);
    \draw[->,thick,fill=green,fill opacity=0.3] (Pxz) -- (P);
%    \draw[fill=red,fill opacity=0.3] (Px) -- (Pxy) -- (P) -- (Pxz) -- cycle;
%    \draw[fill=magenta,fill opacity=0.3] (Py) -- (Pxy) -- (P) -- (Pyz) -- cycle;

    \draw[dashed,fill=gray!50,fill opacity=0.1] (O) -- (Py) -- (Pyz) -- (Pz) -- cycle;
    \draw[dashed,fill=yellow,fill opacity=0.1] (O) -- (Px) -- (Pxz) -- (Pz) -- cycle;
    \draw[dashed,fill=green,fill opacity=0.1] (Pz) -- (Pyz) -- (P) -- (Pxz) -- cycle;
    \draw[dashed,fill=red,fill opacity=0.1] (Px) -- (Pxy) -- (P) -- (Pxz) -- cycle;
    \draw[dashed,fill=magenta,fill opacity=0.1] (Py) -- (Pxy) -- (P) -- (Pyz) -- cycle;
  \end{tikzpicture}
&
  \begin{tikzpicture}[scale=1.5,tdplot_main_coords]
    \coordinate (O) at (0,0,0);
    \tdplotsetcoord{P}{1.414213}{54.68636}{45}

    \draw[fill=gray!50,fill opacity=0.3] (O) -- (Py) -- (Pyz) -- (Pz) -- cycle;
    \draw[->,thick,fill=gray!50,fill opacity=0.3] (O) -- (Py);
    \draw[->,thick,fill=gray!50,fill opacity=0.3] (Py) -- (Pyz);
    \draw[->,thick,fill=gray!50,fill opacity=0.3] (O) -- (Pz);
    \draw[->,thick,fill=gray!50,fill opacity=0.3] (Pz) -- (Pyz);
    \draw[->,thick,fill=gray!50,fill opacity=0.3] (Pyz) -- (P);

    \draw[dashed,fill=gray!50,fill opacity=0.1] (O) -- (Py) -- (Pyz) -- (Pz) -- cycle;
    \draw[dashed,fill=yellow,fill opacity=0.1] (O) -- (Px) -- (Pxz) -- (Pz) -- cycle;
    \draw[dashed,fill=green,fill opacity=0.1] (Pz) -- (Pyz) -- (P) -- (Pxz) -- cycle;
    \draw[dashed,fill=red,fill opacity=0.1] (Px) -- (Pxy) -- (P) -- (Pxz) -- cycle;
    \draw[dashed,fill=magenta,fill opacity=0.1] (Py) -- (Pxy) -- (P) -- (Pyz) -- cycle;
\end{tikzpicture}
&
  \begin{tikzpicture}[scale=1.5,tdplot_main_coords]
    \coordinate (O) at (0,0,0);
    \tdplotsetcoord{P}{1.414213}{54.68636}{45}

    \draw[fill=yellow,fill opacity=0.3] (O) -- (Px) -- (Pxz) -- (Pz) -- cycle;
    \draw[->,thick,fill=yellow,fill opacity=0.3] (O) -- (Px);
    \draw[->,thick,fill=yellow,fill opacity=0.3] (Px) -- (Pxz);
    \draw[->,thick,fill=yellow,fill opacity=0.3] (Pz) -- (Pxz);
    \draw[->,thick,fill=yellow,fill opacity=0.3] (O) -- (Pz);
    \draw[->,thick,fill=yellow,fill opacity=0.3] (Pxz) -- (P);

    \draw[dashed,fill=gray!50,fill opacity=0.1] (O) -- (Py) -- (Pyz) -- (Pz) -- cycle;
    \draw[dashed,fill=yellow,fill opacity=0.1] (O) -- (Px) -- (Pxz) -- (Pz) -- cycle;
    \draw[dashed,fill=green,fill opacity=0.1] (Pz) -- (Pyz) -- (P) -- (Pxz) -- cycle;
    \draw[dashed,fill=red,fill opacity=0.1] (Px) -- (Pxy) -- (P) -- (Pxz) -- cycle;
    \draw[dashed,fill=magenta,fill opacity=0.1] (Py) -- (Pxy) -- (P) -- (Pyz) -- cycle;
\end{tikzpicture}
& 
  \begin{tikzpicture}[scale=1.5,tdplot_main_coords]
    \coordinate (O) at (0,0,0);
    \tdplotsetcoord{P}{1.414213}{54.68636}{45}

    \draw[fill=red,fill opacity=0.3] (Px) -- (Pxy) -- (P) -- (Pxz) -- cycle;
    \draw[->,thick,fill=red,fill opacity=0.3] (O) -- (Px);
    \draw[->,thick,fill=red,fill opacity=0.3] (Px) -- (Pxy);
    \draw[->,thick,fill=red,fill opacity=0.3] (Pxy) -- (P);
    \draw[->,thick,fill=red,fill opacity=0.3] (Pxz) -- (P);
    \draw[->,thick,fill=red,fill opacity=0.3] (Px) -- (Pxz);

    \draw[dashed,fill=gray!50,fill opacity=0.1] (O) -- (Py) -- (Pyz) -- (Pz) -- cycle;
    \draw[dashed,fill=yellow,fill opacity=0.1] (O) -- (Px) -- (Pxz) -- (Pz) -- cycle;
    \draw[dashed,fill=green,fill opacity=0.1] (Pz) -- (Pyz) -- (P) -- (Pxz) -- cycle;
    \draw[dashed,fill=red,fill opacity=0.1] (Px) -- (Pxy) -- (P) -- (Pxz) -- cycle;
    \draw[dashed,fill=magenta,fill opacity=0.1] (Py) -- (Pxy) -- (P) -- (Pyz) -- cycle;
\end{tikzpicture}
&
  \begin{tikzpicture}[scale=1.5,tdplot_main_coords]
    \coordinate (O) at (0,0,0);
    \tdplotsetcoord{P}{1.414213}{54.68636}{45}

    \draw[fill=magenta,fill opacity=0.3] (Py) -- (Pxy) -- (P) -- (Pyz) -- cycle;
    \draw[->,thick,fill=magenta,fill opacity=0.3] (Py) -- (Pxy);
    \draw[->,thick,fill=magenta,fill opacity=0.3] (Pxy) -- (P);
    \draw[->,thick,fill=magenta,fill opacity=0.3] (Py) -- (Pyz);
    \draw[->,thick,fill=magenta,fill opacity=0.3] (Pyz) -- (P);
    \draw[->,thick,fill=magenta,fill opacity=0.3] (O) -- (Py);

    \draw[dashed,fill=gray!50,fill opacity=0.1] (O) -- (Py) -- (Pyz) -- (Pz) -- cycle;
    \draw[dashed,fill=yellow,fill opacity=0.1] (O) -- (Px) -- (Pxz) -- (Pz) -- cycle;
    \draw[dashed,fill=green,fill opacity=0.1] (Pz) -- (Pyz) -- (P) -- (Pxz) -- cycle;
    \draw[dashed,fill=red,fill opacity=0.1] (Px) -- (Pxy) -- (P) -- (Pxz) -- cycle;
    \draw[dashed,fill=magenta,fill opacity=0.1] (Py) -- (Pxy) -- (P) -- (Pyz) -- cycle;
\end{tikzpicture} 
\\
A & B & C & D & E \\
\end{tabular}
\end{center}
and the vertices $\alpha$, $\beta$, $\gamma$, $\delta$, $\epsilon$ and $\zeta$ correspond to the following six combinatorial dipaths:
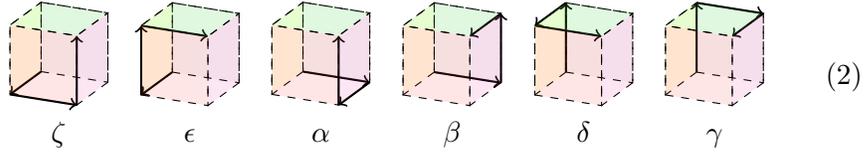
\begin{equation}
\label{E:SimplicialSet2}
\begin{tabular}{cccccc}
  \begin{tikzpicture}[scale=1.2,tdplot_main_coords]
    \coordinate (O) at (0,0,0);
    \tdplotsetcoord{P}{1.414213}{54.68636}{45}

    \draw[->,thick,fill=gray!50,fill opacity=0.3] (O) -- (Px);
    \draw[->,thick,fill=gray!50,fill opacity=0.3] (Px) -- (Pxy);
    \draw[->,thick,fill=gray!50,fill opacity=0.3] (Pxy) -- (P);

    \draw[dashed,fill=gray!50,fill opacity=0.1] (O) -- (Py) -- (Pyz) -- (Pz) -- cycle;
    \draw[dashed,fill=yellow,fill opacity=0.1] (O) -- (Px) -- (Pxz) -- (Pz) -- cycle;
    \draw[dashed,fill=green,fill opacity=0.1] (Pz) -- (Pyz) -- (P) -- (Pxz) -- cycle;
    \draw[dashed,fill=red,fill opacity=0.1] (Px) -- (Pxy) -- (P) -- (Pxz) -- cycle;
    \draw[dashed,fill=magenta,fill opacity=0.1] (Py) -- (Pxy) -- (P) -- (Pyz) -- cycle;
\end{tikzpicture}
& 
  \begin{tikzpicture}[scale=1.2,tdplot_main_coords]
    \coordinate (O) at (0,0,0);
    \tdplotsetcoord{P}{1.414213}{54.68636}{45}

    \draw[->,thick,fill=gray!50,fill opacity=0.3] (O) -- (Px);
    \draw[->,thick,fill=gray!50,fill opacity=0.3] (Px) -- (Pxz);
    \draw[->,thick,fill=gray!50,fill opacity=0.3] (Pxz) -- (P);

    \draw[dashed,fill=gray!50,fill opacity=0.1] (O) -- (Py) -- (Pyz) -- (Pz) -- cycle;
    \draw[dashed,fill=yellow,fill opacity=0.1] (O) -- (Px) -- (Pxz) -- (Pz) -- cycle;
    \draw[dashed,fill=green,fill opacity=0.1] (Pz) -- (Pyz) -- (P) -- (Pxz) -- cycle;
    \draw[dashed,fill=red,fill opacity=0.1] (Px) -- (Pxy) -- (P) -- (Pxz) -- cycle;
    \draw[dashed,fill=magenta,fill opacity=0.1] (Py) -- (Pxy) -- (P) -- (Pyz) -- cycle;
\end{tikzpicture}
&
  \begin{tikzpicture}[scale=1.2,tdplot_main_coords]
    \coordinate (O) at (0,0,0);
    \tdplotsetcoord{P}{1.414213}{54.68636}{45}

    \draw[->,thick,fill=gray!50,fill opacity=0.3] (O) -- (Py);
    \draw[->,thick,fill=gray!50,fill opacity=0.3] (Py) -- (Pxy);
    \draw[->,thick,fill=gray!50,fill opacity=0.3] (Pxy) -- (P);

    \draw[dashed,fill=gray!50,fill opacity=0.1] (O) -- (Py) -- (Pyz) -- (Pz) -- cycle;
    \draw[dashed,fill=yellow,fill opacity=0.1] (O) -- (Px) -- (Pxz) -- (Pz) -- cycle;
    \draw[dashed,fill=green,fill opacity=0.1] (Pz) -- (Pyz) -- (P) -- (Pxz) -- cycle;
    \draw[dashed,fill=red,fill opacity=0.1] (Px) -- (Pxy) -- (P) -- (Pxz) -- cycle;
    \draw[dashed,fill=magenta,fill opacity=0.1] (Py) -- (Pxy) -- (P) -- (Pyz) -- cycle;
\end{tikzpicture}
&
  \begin{tikzpicture}[scale=1.2,tdplot_main_coords]
    \coordinate (O) at (0,0,0);
    \tdplotsetcoord{P}{1.414213}{54.68636}{45}

    \draw[->,thick,fill=gray!50,fill opacity=0.3] (O) -- (Py);
    \draw[->,thick,fill=gray!50,fill opacity=0.3] (Py) -- (Pyz);
    \draw[->,thick,fill=gray!50,fill opacity=0.3] (Pyz) -- (P);

    \draw[dashed,fill=gray!50,fill opacity=0.1] (O) -- (Py) -- (Pyz) -- (Pz) -- cycle;
    \draw[dashed,fill=yellow,fill opacity=0.1] (O) -- (Px) -- (Pxz) -- (Pz) -- cycle;
    \draw[dashed,fill=green,fill opacity=0.1] (Pz) -- (Pyz) -- (P) -- (Pxz) -- cycle;
    \draw[dashed,fill=red,fill opacity=0.1] (Px) -- (Pxy) -- (P) -- (Pxz) -- cycle;
    \draw[dashed,fill=magenta,fill opacity=0.1] (Py) -- (Pxy) -- (P) -- (Pyz) -- cycle;
\end{tikzpicture}
&
  \begin{tikzpicture}[scale=1.2,tdplot_main_coords]
    \coordinate (O) at (0,0,0);
    \tdplotsetcoord{P}{1.414213}{54.68636}{45}

    \draw[->,thick,fill=gray!50,fill opacity=0.3] (O) -- (Pz);
    \draw[->,thick,fill=gray!50,fill opacity=0.3] (Pz) -- (Pxz);
    \draw[->,thick,fill=gray!50,fill opacity=0.3] (Pxz) -- (P);

    \draw[dashed,fill=gray!50,fill opacity=0.1] (O) -- (Py) -- (Pyz) -- (Pz) -- cycle;
    \draw[dashed,fill=yellow,fill opacity=0.1] (O) -- (Px) -- (Pxz) -- (Pz) -- cycle;
    \draw[dashed,fill=green,fill opacity=0.1] (Pz) -- (Pyz) -- (P) -- (Pxz) -- cycle;
    \draw[dashed,fill=red,fill opacity=0.1] (Px) -- (Pxy) -- (P) -- (Pxz) -- cycle;
    \draw[dashed,fill=magenta,fill opacity=0.1] (Py) -- (Pxy) -- (P) -- (Pyz) -- cycle;
\end{tikzpicture}
&
  \begin{tikzpicture}[scale=1.2,tdplot_main_coords]
    \coordinate (O) at (0,0,0);
    \tdplotsetcoord{P}{1.414213}{54.68636}{45}

    \draw[->,thick,fill=gray!50,fill opacity=0.3] (O) -- (Pz);
    \draw[->,thick,fill=gray!50,fill opacity=0.3] (Pz) -- (Pyz);
    \draw[->,thick,fill=gray!50,fill opacity=0.3] (Pyz) -- (P);

    \draw[dashed,fill=gray!50,fill opacity=0.1] (O) -- (Py) -- (Pyz) -- (Pz) -- cycle;
    \draw[dashed,fill=yellow,fill opacity=0.1] (O) -- (Px) -- (Pxz) -- (Pz) -- cycle;
    \draw[dashed,fill=green,fill opacity=0.1] (Pz) -- (Pyz) -- (P) -- (Pxz) -- cycle;
    \draw[dashed,fill=red,fill opacity=0.1] (Px) -- (Pxy) -- (P) -- (Pxz) -- cycle;
    \draw[dashed,fill=magenta,fill opacity=0.1] (Py) -- (Pxy) -- (P) -- (Pyz) -- cycle;
\end{tikzpicture}
\\
$\zeta$ & $\epsilon$ & $\alpha$ & $\beta$ & $\delta$ & $\gamma$ 
\end{tabular}
\end{equation}

\medskip

Here $\beta$ is the geometric intersection
of $B$ with $E$, $\gamma$ is the geometric intersection of $B$ with $A$ and so forth. This leads to the simplicial set pictured in~\eqref{E:SimplicialSet}. 

In order to apply persistent homology, we need to obtain a (multidimensional) filtration of the simplicial set pictured in (\ref{E:SimplicialSet}) from this directed space, \ie a map from some poset $P$ such as $\N^2$ to the category of simplicial sets. One could think that such a filtration could be obtained by moving the end-points $a$ and $b$ associated to a trace space $\trace K a b$. However, as we illustrate below, there is no canonical way of obtaining such a filtration in general; we must use extensions along traces to define inclusion maps.

There are maps from $\trace K a b$ to $\trace K{a'}{b'}$, for $a \leq a'$
and $b' \leq b$ that act as restriction maps~: they just ``cut'' the combinatorial
dipaths so as to only keep the parts (if any) that go from $a'$ to $b'$. Hence, we get a decreasing sequence of simplicial sets as soon
as any of the coordinates of $a$ increase or any of the coordinates of $b$ decrease. 
Below, we have represented the part of the multidimensional filtration generated, for the vertical coordinate of $b$ (the end point) and  of $a$ (the starting point); recall also that the 5 squares are here unit squares and the lower coordinates are 0, upper ones are 1. 
In this filtration, the restriction maps acting on combinatorial dipaths should correspond to inclusion maps from bottom to top, and from left to right, of simplicial sets representing the corresponding trace spaces. 

For instance, moving the end point $b$ from vertical coordinate $1$ to $0$ while keeping the vertical coordinate of $a$ at 0 (right column in the table below), the only 1-dimensional paths going through coordinate $0$ for $b$ are $\alpha$ and $\zeta$, hence all other vertices (and edges) have to disappear. This induces the upwards inclusion map from the two-point simplicial set ($\alpha$ and $\zeta$) into the connected simplicial set above~: $H_0$ of these simplicial sets goes from $\Z^2$ to $\Z$, ``killing'' one component when extending paths to reach the end point of the matchbox. This corresponds, in the natural homology diagram $\overrightarrow H_1 (K)$, to part of the diagram being a projection map from $\Z^2$ to $\Z$ when moving $b$ to the endpoint of the matchbox, while keeping the starting point fixed at the initial vertex.

\medskip

\begin{center}
\begin{tabular}{|c|c|c|}
\hline
 & a=1 & a=0 \\ 
\hline
b=1 & \begin{tikzpicture}[auto,line/.style={draw,thick,-latex',shorten >=2pt},cloudgrey/.style={draw=black,thick,circle,fill={rgb:black,1;white,2},minimum height=1em},cloud/.style={draw=black,thick,circle,fill=white,minimum height=1em},clouddashed/.style={dashed,draw=black,thick,circle,fill=white,minimum height=1em}]
\matrix[column sep=4mm,row sep=2mm]
{
& \node [cloud] (gamma) {$\gamma$}; & \node [cloud] (delta) {$\delta$}; & \\
\node [clouddashed] (beta) {}; & & & \node [clouddashed] (epsilon) {}; \\
& \node [clouddashed] (alpha) {}; & \node [clouddashed] (zeta) {}; & \\
};
\draw (gamma) -- node[above] {$A$} (delta);
\draw [dashed] (beta) -- node[above,left] {$B$} (gamma);
\draw [dashed] (delta) -- node[above,right] {$C$} (epsilon);
\draw [dashed] (alpha) -- node[below,left] {$E$} (beta);
\draw [dashed] (zeta) -- node[below,right] {$D$} (epsilon);
\end{tikzpicture}
& 
\begin{tikzpicture}[auto,line/.style={draw,thick,-latex',shorten >=2pt},cloudgrey/.style={draw=black,thick,circle,fill={rgb:black,1;white,2},minimum height=1em},cloud/.style={draw=black,thick,circle,fill=white,minimum height=1em}]
\matrix[column sep=4mm,row sep=2mm]
{
& \node [cloud] (gamma) {$\gamma$}; & \node [cloud] (delta) {$\delta$}; & \\
\node [cloud] (beta) {$\beta$}; & & & \node [cloud] (epsilon) {$\epsilon$}; \\
& \node [cloud] (alpha) {$\alpha$}; & \node [cloud] (zeta) {$\zeta$}; & \\
};
\draw (gamma) -- node[above] {$A$} (delta);
\draw (beta) -- node[above,left] {$B$} (gamma);
\draw (delta) -- node[above,right] {$C$} (epsilon);
\draw (alpha) -- node[below,left] {$E$} (beta);
\draw (zeta) -- node[below,right] {$D$} (epsilon);
\end{tikzpicture}
\\
\hline 
b=0 & $\emptyset$
& 
\begin{tikzpicture}[auto,line/.style={draw,thick,-latex',shorten >=2pt},cloudgrey/.style={draw=black,thick,circle,fill={rgb:black,1;white,2},minimum height=1em},cloud/.style={draw=black,thick,circle,fill=white,minimum height=1em},clouddashed/.style={dashed,draw=black,thick,circle,fill=white,minimum height=1em}]
\matrix[column sep=4mm,row sep=2mm]
{
& \node [clouddashed] (gamma) {$\gamma$}; & \node [clouddashed] (delta) {$\delta$}; & \\
\node [clouddashed] (beta) {$\beta$}; & & & \node [clouddashed] (epsilon) {$\epsilon$}; \\
& \node [cloud] (alpha) {$\alpha$}; & \node [cloud] (zeta) {$\zeta$}; & \\
};
\draw [dashed] (gamma) -- node[above] {$A$} (delta);
\draw [dashed] (beta) -- node[above,left] {$B$} (gamma);
\draw [dashed] (delta) -- node[above,right] {$C$} (epsilon);
\draw [dashed] (alpha) -- node[below,left] {$E$} (beta);
\draw [dashed] (zeta) -- node[below,right] {$D$} (epsilon);
\end{tikzpicture}
\\
\hline
\end{tabular}
\end{center}

\medskip

The reason that we obtain an inclusion map from bottom to top in this case is because there is a unique map from the point $(1,1,0)$ to the point $(1,1,1)$. When there is a choice between maps, we no longer obtain a canonical inclusion map. Indeed, consider the case in which $a=(0,0,0)$, the initial point, $b' = (1,1,0)$ and $b= (1,1,1)$, the terminal point. There are two extension maps from $\trace K {b'} b$ to $\trace K a b$, one by precomposition by the trace of $\zeta$, and one by precomposition by the trace of $\alpha$. These will produce different homology maps once the invariant is applied. 

Therefore there is no canonical multidimensional filtration of the trace space which depends only on start and end points. More generally, it is easily seen that such a multidimensional filtration for studying a directed space $X$ would exist only if we had a way to associate in a continuous manner, to each pair of points $\alpha$ and $\beta$, a directed path going from $\alpha$ to $\beta$. It is well-known that indeed, such a continuous map will only exist if $X$ is contractible in a directed manner, \ie is trivial, see e.g. \cite{dircompl}. 

{The objective of this article is to cope with this difficulty, in order to give a meaning to natural homology as persistence homology. The shift of point of view we are going to make is similar to recent approaches in multidimensional persistence. We are going to look at all unidimensional persistence homologies that are compatible with our directed space structure, and glue this information together.} \replaced{As exemplified above, in general there exist several extensions between the same trace spaces which induce different homology maps. For this reason, filtrations parametrised by beginning- and end-points do not suffice to capture all of the directed homological information of a given space; we therefore choose to consider filtrations generated by the extensions themselves.}{The non-canonicity of multidimensional filtrations, due to the existence in general of multiple extensions between trace spaces, we will no longer be considering filtrations parametrised by beginning- and end-points, but which are generated by extensions along a given trace.} These have a natural interpretation as uni-dimensional persistence modules and are closely related to the structure of the directed space in question.

\section{A sketch of persistence}
\label{S:Sketch}

In this section we recall the definitions of persistent homology using the notion of persistence module from~\cite{Zomorodian:2005aa}. We \replaced{also recall}{recall also} an algorithm for computing persistent homology based on a structure theorem for persistence vector spaces~\cite{Zomorodian:2005aa}.  We refer the reader to \cite{Carlsson09,EdelsbrunnerHarer2008,EdelsbrunnerHarer2010} for more detailed accounts of persistent homology.

\subsection{Persistence complexes}

Given a poset $P$, considered as a category, a \emph{$P$-persistence object} in a category $\Cr$ is a functor $\Phi : P \fl \Cr$. Explicitly, it is given by a collection \replaced{$\{C_x\}_{x\in P}$}{$\{C_x\}_x$} of objects in $\Cr$ indexed by the elements of $P$, and for all $x \leq y$ in $P$, \replaced{a choice of a map in $C$}{there exists a unique map} $\phi_{x,y}: C_x \fl C_y$ \replaced{satisfying}{such that} $\phi_{y,z} \circ \phi_{x,y} = \phi_{x,z}$ whenever $x\leq y \leq z$. We denote by $\PersCat{P}{\Cr}$ the functor category of $P$-persistence objects in $\Cr$. When $\Cr$ is the category of simplicial complexes, chain complexes, or groups, respectively,~the objects of $\PersCat{P}{\Cr}$ are called $P$-persistence simplicial complexes, chain complexes, or groups respectively.

In particular, considering the poset $\mathbb{N}$ of natural numbers with the usual order, a \emph{positive $\mathbb{N}$-persistence complex}, or \emph{persistence complex} for short, over a ground ring $R$ is a family of chain complexes $C=\{C_\ast^i\}_{i\geq 0}$ over $R$, together with chain \replaced{maps}{map} $f^i : C_\ast^i \fl C_\ast^{i+1}$, giving the following diagram in the category of $R$-modules:
\[
\xymatrix{
C_\ast^0
\ar[r] ^-{f^0}
&
C_\ast^1
\ar[r] ^-{f^1}
&
\; 
\ldots
\;
\ar[r]
&
C_\ast^i
\ar[r] ^-{f^i}
&
C_\ast^{i+1}
\ar[r]
&
\;
\ldots
}
\]
A \emph{persistence module} $M$ is a persistence complex concentrated in degree zero, \ie a family of $R$-modules $\{M^i\}_{i\geq 0}$, together with maps $f^i : M^i \fl M^{i+1}$. 

The persistence complex $C$ is said to be of \emph{finite type} if each $R$-module $C^i_k$ is finitely generated, and if there exists some $N$ such that the maps $f^i$ are isomorphisms for $i\geq N$.

\subsection{Persistent homology}

A \emph{simplical complex} is a set $K$, together with a collection $\mathcal{K}$ of subsets of $K$, satisfying the following two conditions:
\begin{enumerate}[{\bf i)}]
\item for every $v\in K$, $\{v\}\in \mathcal{K}$, and $\{v\}$ is called a \emph{vertex} of $K$;
\item $\sigma\in \mathcal{K}$ and $\sigma'\subseteq\sigma$ imply $\sigma'\in \mathcal{K}$.
\end{enumerate}
A \emph{$k$-simplex of $K$} is an element $\sigma$ of $\mathcal{K}$ whose cardinal $\mid\sigma\mid$ is equal to $k+1$. An 
% A CORRIGER
\emph{orientation} of a $k$-simplex $\sigma=\{v_0,\ldots,v_k\}$ is an equivalence \added{class} of orderings of the $v_i$ in $\sigma$ \replaced{under the equivalence relation generated by even permutations}{; two orderings are equivalent if they can be obtained from an even permutation}. A simplex with an orientation is called an oriented simplex, and we write $[v_0,\ldots,v_k]$ or $[\sigma]$ to denote the equivalence class.

Denote by $C_k(K)$ the \emph{$k$th chain module} of $K$ defined as the free $R$-module on oriented $k$-simplices of $K$. The boundary operator $\partial_k:C_k(K) \fl C_{k-1}(K)$ is the map defined on any simplex $\sigma=\{v_0,\ldots,v_k\}$ by setting 
\[
\partial_k(\sigma) = \sum_i(-1)^i[v_0,\ldots,\widehat{v_i},\ldots,v_k],
\]
where, on the right side, $\widehat{v_i}$ indicates that the vertex $v_i$ is eliminated from the simplex. Denote by $Z_k(K)=\text{ker}\,\partial_k$, $B_k(K)=\text{Im}\,\partial_{k+1}$, and $H_k(K)=Z_k(K)/B_k(K)$, the \emph{cycle}, \emph{boundary}, and \emph{homology} modules respectively. 

A \emph{subcomplex} of $K$ is a simplicial complex $L$ such that $L\subseteq K$.
A \emph{filtered complex} is a complex $K$ together with a \emph{filtration}, that is a nested sequence of subcomplexes:
\[
K^0 \subseteq K^1 \subseteq K^2 \subseteq \ldots \subseteq K^n = K.
\]

Given such a filtration, we define the persistence complex $C(K)=\{ C_\ast(K^i) \}_{i\in \N}$, in which the chain maps $f^{i, i+1} \colon C_\ast(K^i) \fl C_\ast(K^{i+1})$ are induced by the inclusions $K^i \fl K^{i+1}$. Applying the $k^{th}$ homology functor $H_k$ to each complex, we obtain $H_k(C(K)) := \{ H_k \left( C(K_i)_* \right) \}_{i \in \N}$, which has the structure of persistence module over the ground ring $R$. Denoting by $\eta^{i,i+p}_k : H_k(K^i) \fl H_k(K^{i+p})$ the map induced by the inclusion $K^i \fl K^{i+p}$, we define the \emph{$p$-persistent $k^{th}$ homology group of $K^i$} as $\mathrm{Im}(\eta^{i,i+p}_k)$, which we denote by $H_k^{i,p}(K)$, see~\cite{EdelsbrunnerLetscherZomorodian02}.

When $K$ is finite, the persistence complex $C(K)$ is of finite type, and hence the persistent homology $H_k(C(K))$ is also of finite.

\subsection{Classification of persistence module and algorithm for persistence}

Given a persistence $R$-module $(M^i, \phi_{i,i+1})_i$, we define a graded module over $R[t]$ by setting
\[
\alpha(M) = \bigoplus_{i=0}^\infty M^i,
\]
where the $R$-module structure is induced by the direct sum of its components, and the action of $t$ is defined by
\[
t \cdot (m_0, m_1, \dots, m_k, \dots) = (0, \phi_{0,1}(m_0), \phi_{1,2}(m_1), \dots, \phi_{n,n+1}(m_k), \dots).
\]
The correspondence $\alpha$ is functorial and establishes an equivalence between the category of persistence $R$-modules of finite type and the category of finitely generated graded $R[t]$-modules.
When the ground ring is a field $\k$, combining this with the structure theorem of finitely generated graded modules, we know that a persistence module has a decomposition 
\begin{equation}
\label{Eqn:GradedModChar}
\left( \bigoplus_{i=1}^n \Sigma^{k_i} \k[t] \right) \oplus 
\left( \bigoplus_{j=1}^m \Sigma^{l_j} \k[t]/(t^{h_j}) \right)
\end{equation}
for some $n,m$ and families of natural numbers $(k_i)_i$, $(l_j)_j$ and $(h_j)_j$, where $\Sigma^k$ denotes a $k$-shift in grading.

From this classification of persistent modules one derives an algorithm for computation of persistence homology over a field~\cite{Zomorodian:2005aa}. This algorithm is defined as follows. Let $\{e_j\}$ and $\{\hat{e}_i\}$ denote homogenous bases for the persistence $\k$-modules $C_k$ and $Z_{k-1}$. Denote by $M_k$ the matrix of $\partial_k$ in these bases. The usual procedure for calculating homology is to reduce the matrix to Smith normal form and read off the description of $H_k$ from the diagonal elements. We compute these bases and matrix representations by induction on $k$. For $k=1$, the standard basis of $C_0 = Z_0$ is homogenous and we may proceed as usual.

Suppose now that we have a representation $M_k$ of $\partial_k$ relative to the standard basis $\{e_j\}$ of $C_k$ and a homogeneous basis $\{\hat e_i\}$ of $Z_{k-1}$. For induction, we must compute a homogeneous basis for $Z_k$ and represent $\partial_{k+1}$ relative to the computed basis for $Z_k$ and the standard basis of $C_{k+1}$.

We begin by sorting the basis $\{\hat e_i\}$ in reverse degree order and then transforming the matrix $M_k$ into column-echelon form $\tilde M_k$, also known as lower staircase form. We call the first non-zero value in a column a pivot, and a row with a pivot is called a pivot row. 
%\[
%\begin{pmatrix}
%\fbox{$\ast$} & 0 & \cdots&  & & 0 \\
% & \fbox{$\ast$} & 0 & \cdots & & \vdots \\ 
%& {\ast} & 0 & \cdots & & \\
% & & \fbox{$\ast$} & 0 & \cdots \\
%\end{pmatrix}
%\]
The diagonal elements in Smith normal form are the same as the pivots in column-echelon form and the degrees of the corresponding basis elements are also the same in both cases~\cite{Zomorodian:2005aa}. 
Zomorodian and Carlsson prove in~\cite{Zomorodian:2005aa} that each row contributes to the persistent homology $H_{k-1}$ of $\mathcal C_*$ in the following way:
\begin{enumerate}[{\bf i)}]
\item If row $i$ is a pivot row with pivot $t^n$, then it contributes $\Sigma^{\deg \hat e_i} \k[t] / t^n$,
\item If row $i$ is not a pivot row, \added{then} it contributes $\Sigma^{\deg \hat e_i} \k[t]$,
\end{enumerate}
where these contributions correspond to factors in the characterization (\ref{Eqn:GradedModChar}), see~\cite{Zomorodian:2005aa} for additional details.

\section{Natural homology of directed spaces}
\label{sec:natural}

In this section, we recall definitions and constructions involving directed spaces and in particular, we recall natural homology, an invariant thereof. In Sections~\ref{SS:dspaces} and~\ref{sec:tracecat}, we recall the definition of directed space and associated categorical constructions, as well as the notions of natural systems and composition pairings. Finally, Section~\ref{SS:NaturalHomology} recalls the definition of natural homology from~\cite{naturalhomology}, an invariant of directed spaces encoded via natural systems. Therein we also recall the notion of bisimulation and provide, see Lemma~\ref{lem:dihomeobisim}, the link between dihomeomorphic directed spaces and bisimulation equivalence of their associated natural homology diagrams.

\subsection{Directed spaces}\label{SS:dspaces}

Recall from~\cite{grandisbook} that a \emph{directed space}, or \emph{dispace}, is a pair $\Xr=(X, dX)$, where~$X$ is a topological space and $dX$ is a set of paths in $X$, \ie continuous maps from~$[0,1]$ to~$X$, called \emph{directed paths}, or \emph{dipaths}, satisfying the three following conditions:
\begin{enumerate}[{\bf i)}]
\item Every constant path is directed,
\item \replaced{If $f \in dX$, then given any continuous, monotonic map $\rho : [0,1] \fl [0,1]$, we have $f\circ \rho \in dX$. The map $\rho$ is called a \emph{reparametrisation-restriction}, or simply a \emph{reparametrisation} if it is surjective}{$dX$ is closed under \emph{monotonic reparametrisation}, \emph{i.e.}, for any continuous map $\rho : [0,1] \fl [0,1]$ such that $\rho(0)=0$, $\rho(1)=1$ and $\rho$ is monotonic, then $f\circ \rho$ belongs to $dX$ for every $f$ in $dX$,} 
\item $dX$ is closed under concatenation.
\end{enumerate}

We will denote by $f\star g$ the concatenation of dipaths $f$ and $g$.
%, modulo monotonic reparametrisation. 
A morphism $\map{\varphi}{(X,dX)}{(Y,dY)}$ of dispaces is a continuous function $\varphi : X \fl Y$ that preserves directed paths, \emph{i.e.}, for every path $\map{f}{[0,1]}{X}$ in $dX$, the path $\varphi_*(p)=\map{\varphi\circ f}{[0,1]}{Y}$ belongs to $dY$. The category of dispaces is denoted $\dtop$. 

\added{A \emph{partially ordered space}, or \emph{pospace} $\Xr = (X,\leq_X)$ consists of a Hausdorff topological space $X$ and a partial order $\leq_X$ which is closed in the product topology $X \times X$. Pospaces are naturally interpreted as directed spaces by equipping them with the set of increasing paths $dX$ from the unit interval, with its usual ordering, to $X$. Thus, morphisms in $\dtop$ between pospaces are simply continuous, order-preserving maps.}

Given two directed spaces $X$ and $Y$, a dihomeomorphism from $X$ to $Y$ is a morphism of directed spaces $\varphi: \ X \rightarrow Y$ such that there exists a morphism of directed spaces $\psi: \ Y \rightarrow X$, with $\varphi \circ \psi=Id_Y$ and $\psi \circ \varphi=Id_X$. \replaced{In the case of compact partially ordered spaces, this is implied by $\varphi$ and $\psi$ being continuous, order-preserving, and mutually inverse.}{Equivalently, for the sub-class of compact partially-ordered spaces, $\varphi$ and $\psi$ must be continuous, {respect} the respective partial orders, and be inverses to one another.} \replaced{H}{It is clear that h}aving dihomeomorphic directed spaces is a \deleted{fine} semantic equivalence\deleted{,} when it comes to applications in concurrency theory. 

Two dipaths $f$ and $g$ are \emph{reparametrisation equivalent} if there exist %\added{(not necessarily strictly monotonic)} 
reparametrisations $\rho$, $\lambda$ such that $f\circ \rho = g \circ \lambda$. 
The \emph{trace}~\cite{Faj06} of a dipath $f$ in $\Xr$, denoted by $\tr{f}$, or $f$ if no confusion is possible, is the equivalence class of $f$ modulo reparametrisation equivalence. The concatenation of dipaths of $\Xr$ is compatible with this quotient, inducing a concatenation of traces defined by $\tr{f}\star\tr{g} := \tr{f\star g}$, for all dipaths $f$ and~$g$ of $\Xr$, \cite{Uli2}. \deleted{The sub-topological space of $dX$ of dipaths in $X$ from $\alpha \in X$ to $\beta \in X$ is denoted by $dX(\alpha,\beta)$.} For $\varphi: \ X \rightarrow Y$ a morphism of directed spaces, abusing notation, we write $\varphi_*(\overline{f})=\overline{\varphi_*(f)}$ for the image of the trace $\overline{f}$. 

In the case of the geometric realization of finite geometric precubical sets without loops, it is known \cite{goubault2016} that we have a complete metric space and that we can define the $l_1$-arc length $l_1(p)$ of a directed path $p$. In that context, traces $\tr{f}$ always have a representative $f$ such that the length of the sub-path from $f(0)$ to $f(t)$ along $f$ is $t$. 
%\out{Note that the forgetful functor $U : \dtop \fl \Top$ with values in the category of topological spaces admits left and right adjoint functors. The left adjoint functor sends a topological space $X$ to the dispace $(X,X_d)$, where $X_d$ is the set of constant directed paths. The right adjoint sends $X$ to the dispace $(X,X^{[0,1]})$, where $X^{[0,1]}$ is the set of all paths in $X$.}

For a dispace $\Xr=(X, dX)$ and $x,y\in X$, we denote by $\overrightarrow{Di}(\Xr)(x,y)$ the space of dipaths $f$ in $X$ with source $x=f(0)$ and target $y=f(1)$, equipped with the compact-open topology.
For $x,y\in X$, the \emph{trace space} of the dispace $\Xr$ from $x$ to $y$, denoted by~$\trace{\Xr}{x}{y}$, is the quotient of $\overrightarrow{Di}(\Xr)(x,y)$ by reparametrisation equivalence, equipped with the quotient topology.

\subsection{Trace category and trace diagrams}
\label{sec:tracecat}

Recall that the \emph{category of factorizations} of a category $\Cr$, denoted by $\Fact{\Cr}$, is the category whose objects are the morphisms of $\Cr$, and a morphism from $f$ to $g$ is a pair $(u,v)$ of morphisms of $\Cr$ such that $v\circ f \circ u =g$ holds in $\Cr$.

\begin{center}
\begin{minipage}{0.4\textwidth}
$(u,v): f \fl g$
\end{minipage}
\begin{minipage}{0.4\textwidth}
\xymatrix@C=1em@R=1em{
&
\cdot
    \ar[rr]^-{f}
&
{}\ar@{}[d]|-{=}
&
\cdot
    \ar[dr]^-{v}
&
\\
\cdot
\ar[ur]^-{u}
\ar[rrrr]_-{g}
&&&&
\cdot
}
\end{minipage}
\end{center}

Composition is given by 
\begin{center}
\begin{minipage}{0.4\textwidth}
$(u,v)\circ (u',v') = (u\circ u',v'\circ v)$
\end{minipage}
\begin{minipage}{0.4\textwidth}
\xymatrix@C=1em@R=1em{
&&
\cdot \ar[rr]^-{f}
&
{}\ar@{}[d]|-{=}
&
\cdot \ar[dr] ^-{v}
&&
\\
&
\cdot \ar[ur]^-{u} \ar[rrrr]_-{g}
&&
{}\ar@{}[d]|-{=}
&&
\cdot \ar[dr] ^-{v'}
\\
\cdot \ar[ur]^-{u'} \ar[rrrrrr]_-{h}
&&&&&&
\cdot
}
\end{minipage}
\end{center}
whenever the pairs $u',u$ and $v,v'$ are composable in $\Cr$, and the identity on $\map{f}{x}{y}$ is the pair~$(1_x , 1_y)$. 
A \emph{natural system on $\Cr$ with values in a category $\V$} is a functor
$\map{D}{\Fact{\Cr}}{\V}$.
%\out
{We will denote by $D_{f}$ (resp.~$D(u,v)$) the image of an object $f$ (resp. morphism $(u,v)$) of $\Fact\Cr$.}

A directed space may be seen as a category using the notion of traces, and this association is functorial.
Indeed, consider the functor
\[
\overrightarrow{\mathbf{P}} : \dtop \fl \cat
\] 
with values in the category of small categories which associates to a dispace $\Xr$ the \emph{trace category of $\Xr$}, whose objects are points of~$X$, morphisms are traces of $\Xr$, and composition is given by concatenation of traces.
The \emph{trace diagram} of a dispace $\Xr$ in the category $\Top$ of topological spaces is the natural system
\[
\map{T(\Xr)}{{\Fact\diP{\Xr}}}{\Top},
\]
sending a trace $\map{\tr{f}}{x}{y}$ of $\Xr$ to the topological space $\trace{\Xr}{x}{y}$, and a morphism $(\tr{u},\tr{v})$ of $\Fact{\diP{\Xr}}$ to the continuous map 
\[
\tr{u} \star \_ \star \tr{v} : \trace{\Xr}{x}{y} \fl \trace{\Xr}{x'}{y'},
\]
which sends a trace $\tr{f}$ to $\tr{u} \star \tr{f} \star \tr{v}$.

Recall from \cite{CameronGoubaultMalbosHHA22}, that the \emph{category of natural systems with values in $\V$}, denoted by $\opnat{\V}$, is defined as follows:
\begin{enumerate}[{\bf i)}]
\setlength{\itemsep}{0pt}
  \setlength{\parskip}{0pt}
  \setlength{\parsep}{0pt}
\item its objects are pairs $(\Cr,D)$ made of a category $\Cr$ and a natural system $D:F\Cr \fl \V$,
\item its morphisms are pairs
\[
(\Phi,\tau): (\Cr,D)\rightarrow (\Cr',D')
\]
\noindent consisting of a functor $\Phi : \Cr \rightarrow \Cr'$ and a natural transformation $\tau: D \rightarrow \Phi^*D'$, where the natural system $\Phi^*D': \Fact{\Cr} \rightarrow \V$ is defined by
\[
(\Phi^*D')(f)=D'(\Phi f),
\]
\noindent for every morphism $f$ in $\Cr$, and $\Phi^*D'(u,v)=D'(\Phi(u),\Phi(v))$, for all morphisms $u,v$ of $\Cr$, 
\item composition of morphisms $(\Psi,\sigma):(\Cr', D') \fl (\Cr'', D'')$ and $(\Phi,\tau):(\Cr, D) \fl (\Cr', D')$ is defined by
\[
(\Psi,\sigma) \circ (\Phi,\tau):=(\Psi \circ \Phi,(\Phi^*\sigma) \circ \tau),
\]
where $\Psi \circ \Phi$ denotes composition of functors and where the component of the natural transformation $(\Phi^*\sigma) \circ \tau$ at a morphism $f$ of $\Cr$ is $\sigma_{\Phi(f)}\circ \tau_f$.
\end{enumerate}

\medskip

The family of functors ${T}(\Xr)$ indexed by dispaces $\Xr$ extends to a functor
\[
T : \dtop \fl \opnat{\Top}
\]
sending a dispace $\Xr$ to the pair $(\diP{\Xr}, {T}(\Xr))$.
A morphism of dispaces $\varphi : \Xr \fl \Yr$ induces continuous maps 
\[
\varphi_{x,y} : \trace{\Xr}{x}{y} \fl \trace{\Yr}{\varphi(x)}{\varphi(y)},
\]
for all points $x,y$ of $X$, and thus a natural transformation between the corresponding trace diagrams:
\[
\overrightarrow{\varphi} : T(\Xr) \dfl \diP{\varphi}^* T(\Yr).
\] 

\subsection{Natural homology}
\label{SS:NaturalHomology}

This subsection is \replaced{inspired}{taken} from~\cite{Eilenberg}. For $n\geq 1$, the~\emph{$n^{th}$ natural homology functor of $\Xr$} 
\[
\map{\sysh{n}{\Xr}}{{\Fact\diP{\Xr}}}{\ab}
\]
is the functor defined as the composite 
\[
{\Fact\diP{\Xr}} 
\overset{{T}(\Xr)}{\longrightarrow} 
\Top
\overset{H_{n-1}}{\longrightarrow}
\ab,
\] 
where ${T}(\Xr)$ is the trace diagram and $H_{n-1}$ is the $(n-1)^{th}$ singular homology functor.
This extends to a functor
\[
\overrightarrow{H}_n : \dtop \longrightarrow \opnat{\ab},
\]
sending a dispace $\Xr$ to $(\diP{\Xr},\sysh{n}{\Xr})$. %\todo{comparer dans la suite $\opnat{\ab}$ avec $\Pers\ab$}
%\todo{Mention bisimulation and proves that dihomeomorphism implies bisimilar natural homologies}

The natural homology $\sysh{n}{\Xr}$ of a directed topological space $\Xr$ is
very fine-grained: it not only records local homology groups
$\sysh{n}{\Xr}_f$ for a trace $f$ of $\Xr$, but also for which traces
they occur.  
%If we wish to compare the natural homology of two
%pospaces, the latter should be unimportant.  
As with persistent homology, see Section~\ref{S:Sketch}, the aim is to describe the transformation between the groups $\sysh{n}{\Xr}_f$ and $\sysh{n}{\Xr}_g$ for any concatenation $g = u\star f \star v$ of $f$ by an extension $(u,v)$, rather than just the groups $\sysh{n}{\Xr}_f$ themselves.
%Just as with persistent homology, see Section~\ref{S:Sketch}, what counts is the patterns of change between groups $\sysh{n}{\Xr}_f$ and $\sysh{n}{\Xr}_g$ when $f$ is changed into the concatenation of traces $g = u\star f \star u$ by some extension $(u,v)$, not the values of the trace $f$.
%
This is done by looking at natural homology modulo a notion of bisimulation of natural systems, and more
generally of $\ab$-valued functors, defined in~\cite{naturalhomology} and recalled below. 

\subsection{Bisimulation}
\label{sec:bisimulation}

%  ($\Ab$ is
% the category of abelian groups.)
Given two small categories $\Cr$, $\Dr$ and two functors $\map{F}{\Cr}{\ab}$, $\map{G}{\Dr}{\ab}$, we call \emph{bisimulation} between $F$ and $G$
any set $R$ of triples $(x,\eta,y)$, with $x$ an object of $\Cr$, $y$ an object of $\Dr$ and $\eta : Fx \fl Gy$ an isomorphism such that:
\begin{enumerate}[{\bf i)}]
\item for every object $x$ of $\Cr$, $R$ contains some triple of the
  form $(x, \eta, y)$, and similarly for every object $y$ of $\Dr$;
\item for every triple $(x,\eta,y) \in R$ and every morphism
  $\map{i}{x}{x'}$ in $\Cr$, there is 
  a triple $(x',\eta',y') \in R$ and
  a morphism $\map{j}{y}{y'}$ in $\Dr$ such that $\eta' \circ Fi = Gj
  \circ \eta$, and symmetrically, for every $(x,\eta,y) \in R$ and
  every morphism $\map{j}{y}{y'}$ of $\Dr$ there is a triple
  $(x',\eta',y') \in R$ and a morphism $\map{i}{x}{x'}$ such that
  $\eta' \circ Fi = Gj \circ \eta$:
\end{enumerate}

\begin{center}
  \begin{tikzpicture}[scale=1]
    \node (x') at (0,0) {$x'$};
    \node (x) at (0,1.5) {$x$};
    \node (Fx') at (1,0) {$Fx'$};
    \node (Fx) at (1,1.5) {$Fx$};
    \node (Gy') at (3,0) {$Gy'$};
    \node (Gy) at (3,1.5) {$Gy$};
    \node (y') at (4,0) {$y'$};
    \node (y) at (4,1.5) {$y$};
    \draw[->] (x) -- (x');
    \draw[->] (y) -- (y');
    \draw[->] (Fx) -- (Fx');
    \draw[->] (Fx) -- (Gy);
    \draw[->] (Fx') -- (Gy');
    \draw[->] (Gy) -- (Gy');
    \node (i) at (-0.2,0.75) {$i$};
    \node (j) at (4.2,0.75) {$j$};
    \node (Fi) at (0.7,0.75) {$Fi$};
    \node (Gj) at (3.3,0.75) {$Gj$};
    \node (eta) at (2,1.7) {$\eta$};
    \node (eta') at (2,-0.3) {$\eta'$};
  \end{tikzpicture}
\end{center}

\noindent
We say that $F$ and $G$ are \emph{bisimulation equivalent} provided that there is a
bisimulation $R$ between them.  This is an equivalence relation.
We now prove that dihomeomorphic directed spaces have bisimilar natural homologies. This is somewhat implicit in \cite{dubutthesis} but has not been explicited as of, yet. 

\begin{lemma}
\label{lem:dihomeobisim}
Suppose $f: \ \Xr \rightarrow \Yr$ is a \replaced{directed homeomorphism}{morphism of directed spaces with inverse 
$g: \ \Yr \rightarrow \Xr$}. Then the functors $\sysh{n}{\Xr}$ and $\sysh{n}{\Yr}$ are bisimulation equivalent, for every $n \in \N$. 
\end{lemma}
\begin{proof}
\added{By hypothesis, $f$ is a morphism of directed spaces with inverse 
$g: \ \Yr \rightarrow \Xr$}. 
For \replaced{a}{some} given $n$ in $\N$, we need to construct $R$, a set of triples $(x,\eta,y)$, with $x$ an object of $\Fact{\diP{\Xr}}$, $y$ an object of $\Fact{\diP{\Yr}}$, that is $x$ and $y$ are traces of $\Xr$ and $\Yr$ respectively, and $\eta : \sysh{n}{\Xr}_x \fl \sysh{n}{\Yr}_y$ is an isomorphism making this set of triples a bisimulation. 

We define \deleted{the} $R$ from $f$ and $g$ as follows. It is going to contain all
triples $(x,\eta_{x,y},y)$ with $x$ any trace of $\Xr$, $y=f_*(x)$ which is a trace of $\Yr$, and $\eta_{x,y}$ is the map from $\sysh{n}{\Xr}_x$ to $\sysh{n}{\Yr}_y$ induced by $f$ as follows.

We note first that when $f$ is an homeomorphism, its induced map $f_*$ from ${\mathfrak{T}}({\Xr})$ to ${\mathfrak{T}}({\Yr})$ (the trace spaces defined end of Section \ref{SS:dspaces}) is an homeomorphism as well, and similarly for its restriction to $\trace{\Xr}{\alpha}{\beta}$, into $\trace{\Yr}{f(\alpha)}{f(\beta)}$, where $\alpha=x(0)$ and $\beta=y(0)$. Hence $f_*$ induces an isomorphism from the (standard) homology $H_{n-1}(\trace{\Xr}{\alpha}{\beta})$ to the (standard) homology $H_{n-1}(\trace{\Yr}{f(\alpha)}{f(\beta)})$. The former is $\sysh{n}{\Xr}_x$ and the latter is $\sysh{n}{\Yr}_y$ since $y(0)=f(\alpha)$ and $y(1)=f(\beta)$. This  isomorphism $\eta_{x,y}$ can be written explicitely as follows.

%which sends the homology class %\todo{EG: make explicit the cycle of directed paths (and not just the directed path)} 
Consider a singular homology class 
$[p]$, where $p$ is a cycle of dimension $n-1$ in the trace space of $\Xr$ from $\alpha$ to $\beta$.
More precisely, this is the homology class of a sum \replaced{$p=\sum \lambda_i p_i$}{$p=\sum\limits_{i=1}^k \lambda_i p_i$} for some singular simplexes $p_i$ of dimension $n-1$, which are continuous maps $p_i: \ \Delta_{n-1} \rightarrow \trace{\Xr}{\alpha}{\beta}$ taking \replaced{their images}{its image} within the traces of ${\mathfrak{T}}{(\Xr)}$ that start at $\alpha$ and end at $\beta$.  
%Let us define $\eta(u)$ to be the following map from $\Delta_n$ to $\T{\Yr}$: 
Then $\eta_{x,y}([p])$ is 
% the homology class of the chain of $(n-1)$-simplexes:  
\begin{equation}
\eta_{x,y}([p]) = \sum\limits_{i=1}^k \lambda_i [p'_i],
\label{eq:eta}
\end{equation}
\noindent where $p'_i$ is the $(n-1)$-simplex defined by 
$p'_i(t_0,\ldots,t_{n-1})=f_*(p_i(t_0,\ldots,t_{n-1}))$, for $(t_0,\ldots,t_{n-1})\in \Delta_{n-1}$ the standard $(n-1)$-simplex in $\Re^n$. By an abuse of notation, we write $p'_i=f_*(p_i)$. 
%The definition above is legal, indeed, if $v_i$ is reparameterization equivalent to $u_i$, $\eta_{x,y}(v_i)$ is reparameterization equivalent to $\eta_{x,y}(u_i)$. 

%to $p$ starting from $x(0)$ and ending in $x(1)$ within $\Xr$, to 
%class $[f_*(p)]$. This is well \replaced{defined, since}{defined since, first} $f_*(p)$ is a directed path from $f_*(p)(0)=f(p(0))=f(x(0))=y(0)$ to $y(1)$ and as $f$ is a directed homeomorphism, $f_*$ induces an homeomorphism between trace spaces $\Fact{\diP{\Xr}}$ and $\Fact{\diP{\Yr}}$, hence induces an isomorphism between 
%their homology groups. 

Now we check that the set of triples $R$ is indeed a bisimulation. Consider a morphism 
$i: \ x \rightarrow x'$ in $\Fact{\diP{\Xr}}$, i.e. $i=(u,v)$ with $u$ and $v$, traces in $\Xr$. We consider now \added{the} morphism $j: \ y \rightarrow y'$ in $\Fact{\diP{\Yr}}$ defined as $j=(f_*(u),f_*(v))$ and $y'=f_*(u)\star y \star f_{*}(v)=f_*(x')$. 
As before, $f_*$ induces an isomorphism $\eta_{x',y'}$ between $\sysh{n}{\Xr}_{x'}$ and
$\sysh{n}{\Yr}_{y'}$ and $(x',\eta_{x',y'},y')$ is in~$R$. Furthermore, 
$\eta_{x',y'} \circ \sysh{n}{\Xr}_i=\sysh{n}{\Yr}_j \circ \eta_{x,y}$ as we are going to see. 

Let $[p]$ be again a {(standard)} homology class in dimension $n-1$, of $\trace{\Xr}{\alpha}{\beta}$ \replaced{from}{with} $\alpha=x(0)$ to $\beta=x(1)$, i.e. $[p]$ is the homology class of the cycle
\replaced{$p=\sum \lambda_i p_i$}{$p=\sum\limits_{i=1}^k \lambda_i p_i$} as we have seen already, for some singular simplexes $p_i$ of dimension $n-1$.
%, which are continuous maps $p_i: \Delta_{n-1}\rightarrow\T{\Xr}$ taking its image within the traces of $\T{\Xr}$ that start at $\alpha$ and end at $\beta$.  
The effect of $\sysh{n}{\Xr}_i$ on $[p]$ is the homology class in $\sysh{n}{\Xr}_{x'}$: 
$$[q]=\sum\limits_{i=1}^k \lambda_i [q_i]$$
\noindent where $q_i$ is the $(n-1)$-simplex in $\mathfrak{T}(\Yr)$ defined as
\[
q_i(t_0,\ldots,t_{n-1})=u\star p_i(t_0,\ldots,t_{n-1}) \star v,
\]
for $(t_0,\ldots,t_{n-1})\in \Delta_{n-1}$. 
We will write $q_i=u \star p_i \star v$ for short.
Now, 

$$\begin{array}{rcl}
\eta_{x',y'}\circ \sysh{n}{\Xr}_i ([p]) & = & \sum\limits_{i=1}^k \eta_{x',y'}([q_i]),
\end{array}$$
\noindent which is equal, by Equation (\ref{eq:eta}), to 
$$\begin{array}{rcl}
\eta_{x',y'}\circ \sysh{n}{\Xr}_i ([p]) & = & 
\sum\limits_{i=1}^k \lambda_i [f_*(q_i)], \\
& = & \sum\limits_{i=1}^k \lambda_i [f_*(u)\star f_*(p_i) \star f_*(v)],
%& = & [f_*(u) \star r \star f_*(v)]
\end{array}
$$
\noindent as easily checked by instantiating the left and right terms of the equation on tuples $(t_0,\ldots,t_{n-1})\in \Delta_{n-1}$. 

%where $r=\sum\limits_{i=1}^k \lambda_i r_i$, with $r_i$ being the $(n-1)$-simplex of $\mathfrak{T}(\Yr)$ defined by $r_i(t_0,\ldots,t_{n-1})=f_*(q_i(t_0,\ldots,t_{n-1}))$. 
%Therefore, 
%$$\begin{array}{rcl}
%\eta_{x',y'}\circ \sysh{n}{\Xr}_i ([p])  & = & 
%\sum\limits_{i=1}^k \lambda_i [r'_i]
%\end{array}
%$$
%\noindent where $r'_i$ is the $(n-1)$-simplex of $\trace{\Yr}{y'(0)}{y'(1)}$ defined as $r'_i(t_0,\ldots,t_{n-1})=f_*(u)\star r_i(t_0,\ldots,t_{n-1})\star f_*(v)$.

But as we saw in \eqref{eq:eta}, 
$\eta_{x,y}([p])=\sum\limits_{i=1}^k \lambda_i [f_*(p_i)]$. 
%where $p'_i(t_0,\ldots,t_{n-1})=f_*(p_i(t_0,\ldots,t_{n-1})$ for $i=1,\ldots,k$ and $(t_0,\ldots,t_{n-1})\in \Delta_{n-1}$. 
Therefore: 
$$\begin{array}{rcl}
%\sysh{n}{\Yr}_j ([f_*(p)]) \\
\sysh{n}{\Yr}_j \circ \eta_{x,y}([p]) & = & \sum\limits_{i=1}^k \lambda_i [f_*(u) \star f_*(p_i)\star f_*(v)] \\ %[f_*(u) \star p'_i\star f_*(v)], 
& = & \eta_{x',y'}\circ \sysh{n}{\Xr}_i ([p]).
\end{array}
$$
%\noindent with,
%$$\begin{array}{rcl}
%s_i(t_0,\ldots,t_{n-1})=f_*(u)\star p'_i(t_0,\ldots,t_{n-1})\star f_*(v) \\
%& = & f_*(u)\star f_*(p_i(t_0,\ldots,t_{n-1}))\star f_*(v) \\
%& = & f_*(u\star p_i(t_0,\ldots,t_{n-1})\star v)
%\end{array}$$

%\todo[inline]{The other way round now}
Conversely, 
consider a morphism $j: \ y \rightarrow y'$ in $\Fact{\diP{\Yr}}$ with $j=(u',v')$. We have $x'=g_*(u')\star y \star g_{*}(v')=g_*(y')$. 
Let us consider now $i: \ x \rightarrow x'$ in $\Fact{\diP{\Xr}}$, with 
$i=(g_*(u'),g_*(v'))$. 
%so that, with a similar argument as above, 
As before, $g_*$ induces an isomorphism $\eta'_{x',y'}$  (resp. $\eta'_{x,y}$) between $\sysh{n}{\Xr}_{x'}$ (resp. $\sysh{n}{\Xr}_{x}$) and
$\sysh{n}{\Yr}_{y'}$ (resp. $\sysh{n}{\Yr}_{y}$), which is easily seen to be equal to the inverse $\eta^{-1}_{x',y'}$ (resp. $\eta^{-1}_{x,y}$) of $\eta_{x',y'}$ (resp. $\eta_{x,y}$), since $g_*$ is the inverse of~$f_*$.  %and $(x',\eta_{x',y'},y')$ is in~$R$. Furthermore, 
%$\eta_{x',y'} \circ \sysh{n}{\Xr}_i=\sysh{n}{\Yr}_j \circ \eta_{x,y}$ as we are going to see. 

The same calculation as before, exchanging the roles of $\Xr$ with $\Yr$, and $\eta$ with $\eta'$, yields
$\eta'_{x',y'} \circ \sysh{n}{\Yr}_j=\sysh{n}{\Xr}_i \circ \eta'_{x,y}$,
\ie $\eta^{-1}_{x',y'}\circ \sysh{n}{\Yr}_j=\sysh{n}{\Xr}_i \circ \eta^{-1}_{x,y}$, and, by pre-composing this equation with $\eta_{x',y'}$ and post-composing it with $\eta_{x,y}$, we get 
$\sysh{n}{\Yr}_j \circ \eta_{x,y}=\eta_{x',y'}\circ \sysh{n}{\Xr}_i$, and $(x,\eta_{x,y},y) \in R$. 

Hence $R$ is a bisimulation. 
%\todo[inline]{EG: je suis en train de donner plus explicitement l'isomorphisme en homologie $\eta$, car le reviewer 2 n'a pas trop compris.}
\end{proof}

%\todo{Mention the extra structure we have in the time reversal paper - so the classical persistence modules we are generating here will be isomorphic under time reversal ; is the concatenation operation or some similar operation definable and interesting in more general contexts for persistence? EG: yes}

\section{Persistent homology of directed spaces}

\label{sec:persdir}

In this section we explore the use of persistent homology as an invariant of directed spaces, as well as its relationship to natural homology in the context of partially ordered spaces. First, we show that for these directed spaces, the associated natural homology modules are in fact persistence objects as defined in Section~\ref{S:Sketch}. We then define persistent homology along traces in a directed space and show the compatibility of this definition with the interpretation of natural homology as a persistence object. To finish the section, we illustrate our constructions on the matchbox example.

\subsection{Natural homology as a persistence object}

\label{sec:nathomaspers}

Consider a directed space $\Xr = (X, dX)$ and its trace category $\diP{\Xr}$. We define a relation on traces in $\Xr$ by setting
\begin{equation}\label{Eqn:ExtPreorder}
f \leq g \qquad \iff \qquad \exists\; u,v, \quad g = ufv,
\end{equation}
\added{that is, there exists an extension $(u,v)$ from $f$ to $g$. }

We recall that a \emph{pospace} $\Xr = (X,\leq_X)$ consists of a Hausdorff topological space $X$ and a partial order $\leq_X$ which is closed in the product topology $X \times X$. {In what follows, we will exclusively consider compact pospaces}. Pospaces are naturally interpreted as directed spaces by equipping them with the set of increasing paths $dX$ from the unit interval, with its usual ordering, to $X$.

\begin{lemma}\label{L:DspacePO}
In any directed space, the relation $\leq$ defined in (\ref{Eqn:ExtPreorder}) is a pre-order. In a pospace, it is a partial order relation.
\end{lemma}
\begin{proof}
%\todo{write short proof}
Since constant paths are directed, the relation is reflexive, and we have transitivity by associativity of concatenation of traces. Indeed, if $f\leq g$ and $g \leq h$, there exist extensions $(u,v)$ and $(u',v')$ such that
\[
f = ugv \qquad\text{and}\qquad g = u'hv'.
\]
Thus, $f = u(u'hv')v = (uu')h(v'v)$, and $f \leq h$. In the case of a \replaced{pospace}{loop-free directed space}, we also prove anti-symmetry of $\leq$. Consider $f,g \in dX$ such that $f \leq g$ and $g \leq f$. By definition there exist extensions such that $f = ugv$ and $g=u'fv'$. Thus $f = uu' f v'v$, so $uu'$ and $vv'$ \replaced{each have the same beginning and end-points, which implies they are both constant paths. Indeed, given a path $h:x \fl x$ in a pospace, any point $x'$ in the image of $h$ satisfies $x\leq x' \leq x$.}{are loops, and must therefore be constant paths.} \replaced{T}{By Theorem~\ref{T:HaucourtPospace} below, t}his means that the image of both $uu'$ and $v'v$ are \replaced{singletons}{the singleton space}, meaning that the image of each of the dipaths $u,u',v,v'$ are \replaced{singletons}{the singleton space}, \ie these are all constant paths, concluding the proof.
\end{proof}
This poset, denoted  by~$\tpos{\Xr}$, will be called the \emph{trace poset} of $\Xr = (X,\leq_X)$. 
%\added[id=CC, comment={Removed unnecessary reference}]{}
%{Consider a pospace $\Xr = (X, \leq)$. We know from \cite{haucourtdstream} that dipaths in $\Xr$ are characterized by their image:
% Thm 3.14 Corollary 3.16 Two dipaths on the same pospace sharing the same image are di- homotopic.
%\begin{theorem}[{\cite[Thm. 3.15]{haucourtdstream}}]
%\label{T:HaucourtPospace}
%The image of a dipath in a pospace is isomorphic to either the directed %unit interval or {the singleton space}.
%\end{theorem}
%The above results essentially state that two dipaths are equal modulo reparametrisation if and only if they have the same image, and that {this induces a partial ordering on traces}.
In particular,\added{ by construction,} we obtain the following result:
\begin{proposition}\label{P:FactisPos}
For a pospace $\Xr$, $\tpos \Xr$ is isomorphic to $\Fact{\diP{\Xr}}$.
\end{proposition}
\begin{proof}
By the same arguments as in the preceding proof, if $ufv = u'fv'$ then $u=u'$ and $v=v'$ since these must be constant maps equal to $f(0)$ (resp.~$f(1)$). Thus there is at most one extension between any two traces in a pospace, which concludes the proof.
\end{proof}
This allows us to interpret natural homology as a functor on a poset, \ie as a persistence object. Indeed, in a pospace $\Xr$, the $i^{th}$ natural homology diagram associated to $\Xr$ is a $\tpos \Xr$-persistence group. Taking coefficients in a field $\k$, we obtain $\tpos \Xr$-persistence $\k$-vector spaces.

\subsection{Persistent homology along a trace}

\label{sec:persistenttrace}

Here, we show how unidimensional persistence modules may be obtained from parametrisations of traces.
For the following, we fix a trace $\tr{f}$ in a pospace $\Xr$, and a parametrisation $f$ thereof. 
Moreover, we fix a point $\alpha_f$ in the image of $f$, and denote by $t_0$ \replaced{a}{the} value such that $f(t_0) = \alpha_f$. 

We abuse notation in the following, denoting by $\alpha_f$ the constant trace equal to the point $\alpha_f$, thereby considering it as an element of the trace poset $\tpos \Xr$. 
In this optic, we denote by $[\alpha_f , f]$ the interval between $\alpha_f$ and $f$ in the trace poset, \ie traces $p$ such that $\alpha_f \leq p \leq f$ in $\tpos \Xr$.

%(\emph{Intermediate point filtration})
%Consider some $t_0 \in [0,1]$, and 
Given parametrisations of $[0,t_0]$ and $[t_0,1]$ within $[0,1]$ via maps $\gamma_-: \ [0,1] \rightarrow [0,t_0]$ and  $\gamma_+: \ [0,1] \rightarrow [t_0,1]$ respectively, we denote by ${}_s f_s$ the trace of $f$ restricted to $[t_0 - \gamma_-(s), \deleted{t_0 +} \gamma_+(s)]$. \replaced{Applying the trace space functor, w}{W}e thus obtain a $[0,1]$-persistence \replaced{topological space}{simplicial complex} $\{ T({}_s f_s ) \}_s$. %\new
{This is a filtration of the trace space associated to the trace $\tr f$.} Notice that this is derived from a chain \added{$({}_s f_s)_{s\in[0,1]}$ in the poset of traces (see \added{next} Section~\ref{sec:maxtracesmaxchains})} from the constant trace \replaced{$\alpha_f$}{$f(0)$, $f(1)$, or $f(t_0)$} to $f$.
We call such filtrations initial (resp.~terminal) point filtrations when $\alpha_f=f(0)$ (resp.~$\alpha_f=f(1)$)\added{.}
Taking some %\todo{(carefully chosen)} 
order preserving map $\N \fl [0,1]$, we obtain $\N$-persistence simplicial complexes from the above constructions. 

All of these can be considered as $\Real$-persistence objects as follows: $[0,1]$-persistence objects are $\Real$-persistence objects which are constant on negative parameters and on parameters greater or equal than one. \added{The }$\N$-persistence objects \added{constructed above }will be considered as \deleted{piecewise constant} $\Real$-persistence objects \replaced{by considering them first as piecewise constant $[0,1]$-persistence objects}{in what follows}. However, for questions of computability, it is useful to consider $\N$-persistence.

In all the above cases, we obtain a chain $c=(f_i)_{i\in\Real}$ in the interval $[\alpha_f , f]$
and define the \emph{persistent homology along a trace $f$ with respect to $c$} as a functor from $\Real$, seen as the poset category (with the usual ordering) to the category of abelian groups, associating $i \in \Real$ to: % $\N \fl $
\[
\sysh{n}{f,c}_i := \sysh{n}{\Xr}_{f_i}
\]
\replaced{In the following section, we will see that this object is independent of the chosen parametrisations $f,\gamma_-$, and $\gamma_+$.}{\noindent where we recall that $\sysh{n}{\Xr}_{f_i}$ is the natural homology in dimension $n$ of $\Xr$ at trace $f_i$. This $\Real$-persistence module is constructed as described above, that is by considering the $[0,1]$-persistence module given by a trace as a $\N$-persistence module and completing by constants.}

\subsection{Maximal traces and maximal chains}
\label{sec:maxtracesmaxchains}

In order to make the link between parametrisations and chains more concrete, we present the \replaced{relationship}{link} between maximal chains in the trace poset and parametrisations of extensions along maximal traces. In fact, the chain we use does not depend on (re)parametrisations, but rather on the pair of end-points visited along a trace.

Recall from the previous subsection that given a trace $\overline f$, we construct a persistence module using the following data:
\begin{itemize}
\item A parametrisation $f$ of $\overline f$.
\item A point $\alpha = f(t_0)$ in the image of $f$.
\item Reparametrisations $\gamma_-: \ [0,1] \rightarrow [0,t_0]$ and  $\gamma_+: \ [0,1] \rightarrow [t_0,1]$.
\end{itemize}

A \emph{maximal trace} is a trace which cannot be extended on either side. Hence, these are the maximal elements in $\tpos \Xr$. Maximal chains in $\tpos \Xr$ link minimal elements, \ie constant dipaths, to maximal traces along a sequence of extensions. Clearly, for a trace $\overline f$, the above data encodes a maximal chain in the trace poset. Indeed, we obtain a chain ${}_s f_s$ from $\alpha$, seen as a constant trace, to $\overline f$, which is maximal since it is derived from the continuous maps $\gamma_+$ and $\gamma_-$.

Consider $t\in \tpos\Xr$ a maximal trace. Let $\alpha_0:= t(0) \in \Xr$ be the starting point of $t$, and consider the interval $[\alpha_0, t]$ in $\tpos \Xr$. For any $g,h \in [\alpha_0, t]$, we have either $g\leq h$ or $h\leq g$, since all elements of this interval are traces starting at $\alpha_0$ and which are sub-traces of $t$. Thus the interval $[\alpha_0, t]$ is itself a chain. 

From this we deduce that there is a unique maximal chain from $\alpha_0$ to~$t$, which we denote by $C^0_t$, namely the interval between them. A symmetric argument shows that this is also the case when we consider chains from $\alpha_1:=t(1)$, the terminal point of $t$, in which case the unique chain, denoted by $C^1_t$\added{,} is the interval $[\alpha_1, t]$.

Now let us consider a point $\alpha$ in the trace $t$ which is neither the starting nor terminal point of $t$. Let $t_1 \leq t$ the trace from $\alpha_0$ to $\alpha$, and $t_2 \leq t$ the trace from $\alpha$ to $\alpha_1$. 
We claim that the interval $[\alpha, t]$ 
%\todo{We may have to be a little "heavier" on this. I agree this $[\alpha,t]$ is an interval in the trace poset, but below and elsewhere, we use mostly the interval notation for denoting chains. Shouldn't we have 2 notations, or a slightly heavier discussion on this in the paper?}
is the cartesian product, as posets, of the chains $C^1_{t_1}$ and $C^0_{t_2}$. Indeed, given any trace $f \in [\alpha, t]$, we know that $f$ is a subtrace of $t$ which passes through $\alpha$. Thus there exist a unique pair of traces $(f_1, f_2) \in [\alpha, t_1] \times [\alpha, t_2]$ such that $f = f_1 \star f_2$. Conversely, any pair $(f_1, f_2) \in [\alpha, t_1] \times [\alpha, t_2]$ of traces can be concatenated to obtain a trace in the interval $[\alpha, t]$. So any chain from $\alpha$ to $t$ is a chain in the product $[\alpha, t_1] \times [\alpha, t_2]$, ordered coordinate-wise.

More generally, given \emph{any} \added{complete, \ie closed under infima and suprema, }maximal chain $C$ in $\tpos \Xr$, we know that it has a maximal element, namely \replaced{its supremum}{the union of all traces in $C$}, which we denote by $t$. Furthermore, since every trace is a compact subspace of $X$, it has a minimal element, $\alpha$, which must be a constant trace by maximality of $C$. Thus every maximal chain in $\tpos \Xr$ is of the type described in the previous paragraphs.

Summing this up, we obtain the following result:

\begin{theorem}
\label{t:maxchains}
Let $\Xr$ be a compact pospace, consider its trace poset $\tpos \Xr$, and let $C$ be a \added{complete }maximal chain therein. Then \replaced{the maximal element of $C$}{$C$ has a maximal element which} is a maximal trace $t_C$, and \replaced{its minimal element}{has a minimal element which} is a constant trace $\alpha_C$. Furthermore,
\begin{itemize}
\item If $\alpha_C$ is the beginning-point $\alpha_0$ or end-point $\alpha_1$ of \replaced{$t_C$}{$t$}, then $C$ is equal to the interval $[\alpha_C, t_C]$.
\item If not, then $C$ is a chain in the product of chains $C^1_{t_1}$ and $C^0_{t_2}$, where $t_1$ (resp.~$t_2$) is the unique subtrace of \replaced{$t_C$}{$t$} going from $\alpha_0$ to $\alpha_C$ (resp.~from $\alpha_C$ to $\alpha_1$).
\end{itemize}
\end{theorem}

This shows that \added{complete }maximal chains in the trace poset are independent of the parametrisation we choose for our maximal traces. Indeed, given a maximal trace $\dual f$ the first item above states that there is only one chain from the initial or terminal point, which can be obtained from \emph{any} parametrisation $f$ of $\dual f$. The second item states that if we are considering extensions from an interior point $\alpha$, the maximal chains from $\alpha$ to $\dual f$ correspond to directed paths in the cartesian square $[0,1] \times [0,1]$. These directed paths correspond to equivalence classes of pairs $\gamma_+, \gamma_-$ under the equivalence relation which identifies pairs which visit the same pairs of extremal points.

\replaced{As a result, 
the $\Real$-persistence vector space defined in Section~\ref{sec:persistenttrace} from parametrisations $f,\gamma_-, \gamma_+$ coincides with the restriction of the $n$th natural homology of $\Xr$ %$\sysh{n}{$\Xr$}$ 
to a given maximal chain:
}{To make the construction above more clear, we will show explicitly how to find the homology along a trace by restriction of the natural homology diagram. Recall that in a pospace $\Xr$, the trace poset $\tpos{\Xr}$ and the factorisation category of the trace category $\Fact{\diP{X}}$ are isomorphic, so we may view natural homology $\sysh{n}{\Xr}$ as a functor on $\tpos{\Xr}$ with values in the category $\ab$.
In practice, persistent homology is defined with coefficients in $\Real$ or in a field $\k$, whereas natural homology generally considers more general abelian group coefficients. We will restrict in all practical cases to homology with coefficients in $\k$ so as to make comparisons possible.}

%\deleted

%\end{document}

\begin{proposition}
\label{prop:1}
%\replaced{
Let $\Xr = (X,dX)$ be a compact pospace and $\dual f$ a trace in $\Xr$. Let $C$ be a complete maximal chain in the interval 
$[\alpha_f , \dual f]$ in the poset of traces. Let $(f_i)_{i\in \Real}$ be obtained from $\dual f$ as described in Section~\ref{sec:persistenttrace}.
Restricting the natural homology functor $\sysh{n}{\Xr}$ to $C$, we obtain the persistence $\Real$-vector space 
$\sysh{n}{\Xr}_{f_i}$. %}
%{
%Let $\Xr = (X,dX)$ be a pospace and $f$ a trace in $\Xr$. Let $C$ be a chain in the interval 
%$[\alpha_f , f]$ in the poset of traces. Let $(f_i)_{i\in \k}$ be a parametrisation of $C$ obtained as described in Section~\ref{sec:persistenttrace}.
%Restricting the natural homology functor $\sysh{n}{\Xr}$ to $C$, we obtain a persistence $\k$-vector space 
%$\sysh{n}{\Xr}_{f_i}$.
%}
\end{proposition}

\subsection{Example: persistent homology of the matchbox}
\label{ex:persmatchbox}

Consider the case of the matchbox example presented in Section~\ref{SS:MatchboxExample}. 
Let us calculate the persistent homology along each of its maximal traces, starting at the initial point and extending into the future. We fix some field $\k$ and will proceed via the method developed in~\cite{Zomorodian:2005aa}. To each of the maximal traces $\zeta,\epsilon,\alpha,\beta,\delta,\gamma$ pictured in \eqref{E:SimplicialSet2}, we associate a sequence of subtraces corresponding to a decomposition of the trace along each $1$-simplex in the matchbox. For example, the trace $\alpha$ is decomposed into the chain

\medskip

\begin{center}
$\alpha_0 \leq \alpha_1 \leq \alpha_2 \leq \alpha_3 = \alpha$,
\qquad
\raisebox{-0.5cm}{
\begin{tikzpicture}[scale=1.2,tdplot_main_coords]
    \coordinate (O) at (0,0,0);
    \tdplotsetcoord{P}{1.414213}{54.68636}{45}

    \draw[->,thick,fill=gray!50,fill opacity=0.3] (O) -- (Py);
    \draw[->,thick,fill=gray!50,fill opacity=0.3] (Py) -- (Pxy);
    \draw[->,thick,fill=gray!50,fill opacity=0.3] (Pxy) -- (P);

    \draw[dashed,fill=gray!50,fill opacity=0.1] (O) -- (Py) -- (Pyz) -- (Pz) -- cycle;
    \draw[dashed,fill=yellow,fill opacity=0.1] (O) -- (Px) -- (Pxz) -- (Pz) -- cycle;
    \draw[dashed,fill=green,fill opacity=0.1] (Pz) -- (Pyz) -- (P) -- (Pxz) -- cycle;
    \draw[dashed,fill=red,fill opacity=0.1] (Px) -- (Pxy) -- (P) -- (Pxz) -- cycle;
    \draw[dashed,fill=magenta,fill opacity=0.1] (Py) -- (Pxy) -- (P) -- (Pyz) -- cycle;
\end{tikzpicture}
}
\end{center}

\medskip

\noindent where $\alpha_0$ is the constant path equal {to $(0,0,0)$, which $\alpha_1$ extends to $(0,1,0)$.} This trace is further extended to $(0,1,1)$ obtaining $\alpha_2$, and then finally extended to $(1,1,1)$ giving the total trace $\alpha$. A similar decomposition, which we will denote with the same indices, can be found for each of the maximal traces.
This sequence of traces gives a filtration of the simplicial complex pictured above. In the case of $\alpha$, we obtain the following filtration:
\[
\begin{tabular}{|c|c|c|c|}
\hline
\xymatrix@C=3mm@R=2mm{
& 
{}%*+<10pt>[o][F]{\scriptstyle \gamma}
	%\ar@{-}[r]^{A}
& 
{}%*+<10pt>[o][F]{\scriptstyle \delta}
	%\ar@{-}[dr]^{C}
& 
\\
{}%*+<10pt>[o][F]{\scriptstyle \beta}
	%\ar@{-}[ur]^{B}
	%\ar@{-}[dr]_{E}
& & & 
{}%*+<10pt>[o][F]{\scriptstyle \epsilon}
	%\ar@{-}[ul]_{C}
	%\ar@{-}[dl]^{D}
\\
&
*+<10pt>[o][F]{\scriptstyle \alpha_0}
&
{}%*+<10pt>[o][F]{\scriptstyle \zeta} 
&
}
&
\xymatrix@C=3mm@R=2mm{
& 
{}%*+<10pt>[o][F]{\scriptstyle \gamma}
	%\ar@{-}[r]^{A}
& 
{}%*+<10pt>[o][F]{\scriptstyle \delta}
	%\ar@{-}[dr]^{C}
& 
\\
{}%*+<10pt>[o][F]{\scriptstyle \beta}
	%\ar@{-}[ur]^{B}
	%\ar@{-}[dr]_{E}
& & & 
{}%*+<10pt>[o][F]{\scriptstyle \epsilon}
	%\ar@{-}[ul]_{C}
	%\ar@{-}[dl]^{D}
\\
&
*+<10pt>[o][F]{\scriptstyle \alpha_1}
&
{}%*+<10pt>[o][F]{\scriptstyle \zeta} 
&
}
&
\xymatrix@C=3mm@R=2mm{
& 
{}%*+<10pt>[o][F]{\scriptstyle \gamma}
	%\ar@{-}[r]^{A}
& 
{}%*+<10pt>[o][F]{\scriptstyle \delta}
	%\ar@{-}[dr]^{C}
& 
\\
{}%*+<10pt>[o][F]{\scriptstyle \beta}
	%\ar@{-}[ur]^{B}
	%\ar@{-}[dr]_{E}
& & & 
{}%*+<10pt>[o][F]{\scriptstyle \epsilon}
	%\ar@{-}[ul]_{C}
	%\ar@{-}[dl]^{D}
\\
&
*+<10pt>[o][F]{\scriptstyle \alpha_2}
&
*+<10pt>[o][F]{\scriptstyle \zeta_2} 
&
}
&
\xymatrix@C=3mm@R=2mm{
& 
*+<10pt>[o][F]{\scriptstyle \gamma}
	\ar@{-}[r]^{A}
& 
*+<10pt>[o][F]{\scriptstyle \delta}
	\ar@{-}[dr]^{C}
& 
\\
*+<10pt>[o][F]{\scriptstyle \beta}
	\ar@{-}[ur]^{B}
	\ar@{-}[dr]_{E}
& & & 
*+<10pt>[o][F]{\scriptstyle \epsilon}
	\ar@{-}[ul]_{C}
	\ar@{-}[dl]^{D}
\\
&
*+<10pt>[o][F]{\scriptstyle \alpha}
&
*+<10pt>[o][F]{\scriptstyle \zeta} 
&
}
\\
\hline
$K(\alpha_0)$
&
$K(\alpha_1)$
&
$K(\alpha_2)$
&
$K(\alpha_3)$
\\
\hline
\end{tabular}
\]

We extend this filtration into a $\N$-persistence simplicial complex by considering copies of $K(\alpha_3)$ for all $i \geq 4$.
We denote by $C_*(K(\alpha_i))$ the chain complex over $\k$ obtained from the simplicial complex $K(\alpha_i)$, obtaining a $\N$-persistence chain complex
\[
\xymatrix{
C_*(K(\alpha_0))
	\ar[r]^-{f_{0,1}}
&
C_*(K(\alpha_1))
	\ar[r]^-{f_{1,2}}
&
C_*(K(\alpha_2))
	\ar[r]^-{f_{2,3}}
&
C_*(K(\alpha_3))
	\ar[r]^-{f_{3,4}}
&
\cdots
}
\]
where $f_{n,n+1}$ are induced by the inclusions of simplicial sets given by the filtration.
For each natural number $p$, we denote by $H_p^i$ the $p^{th}$ homology group of $C_*(K(\alpha_i))$, thus obtaining a sequence of homology groups
\[
\xymatrix{
H_p^0
	\ar[r]^-{\phi_{0,1}}
&
H_p^1
	\ar[r]^-{\phi_{1,2}}
&
H_p^2
	\ar[r]^-{\phi_{2,3}}
&
H_p^3
	\ar[r]^-{\phi_{3,4}}
&
H_p^3
	\ar[r]^-{\phi_{4,5}}
&
\cdots
	\ar[r]^-{\phi_{n-1,n}}
&
H_p^3
	\ar[r]^-{\phi_{n,n+1}}
&\cdots,
}
\]
where the $\phi_{n,n+1}$ are identities for $n \geq 3$. This is in fact an $\N$-persistence $\k$-vector space. We define a non-negatively graded module over $\k[t]$ by setting
\[
H_p := \left( \bigoplus_{i=0}^2 H_p^i \right) \oplus \left( \bigoplus_{i=3}^\infty H_p^3 \right),
\]
and defining the action of $t$ by $t \cdot (h_i)_{i} = (\phi_{i,i+1}(h_i))_i$, where the $h_i$ belongs to $H_p^i$. 

We will calculate the graded module of persistent homology via matrix representations of the boundary maps $\partial_k$ associated to the persistence chain complex $C_*(K(\alpha_i))$ as described in~\cite{Zomorodian:2005aa}.
We calculate $H_0(\alpha)$, the $0^{th}$ persistent homology along $\alpha$. For this, we fix homogeneous bases for $Z_0$ and $\mathcal C_1$. Since $Z_0 = \mathcal C_0$, we may take the standard basis in both cases. Thus, for $\mathcal C_0$ we obtain the basis $\{\alpha_0, \zeta_2, \beta, \gamma, \delta, \epsilon \}$, and for $\mathcal C_1$ we obtain the basis $\{A,B,C,D,E\}$. We now calculate the matrix of $\partial_1$ with respect to these bases, taking care to order the basis of $\mathcal C_0$ in reverse degree order:
\[
\begin{tabular}{c|ccccc}
{}
&
$B$
&
$A$
&
$C$
&
$D$
&
$E$
\\
\hline
$\beta$ & -1 & 0 & 0 & 0 & 1
\\
$\gamma$ & 1 & -1 & 0 & 0 & 0
\\
$\delta$ & 0 & 1 & -1 & 0 & 0
\\
$\epsilon$ & 0 & 0 & 1 & -1 & 0
\\
$\zeta_2$ & 0 & 0 & 0 & $t$ & 0
\\
$\alpha_0$ & 0 &0 &0 &0 & $-t^3$
\end{tabular}
\]
We now calculate the column-echelon form of the above matrix, obtaining
\[
\begin{tabular}{c|ccccc}
{}
&
$B$
&
$A$
&
$C$
&
$D$
&
$E'$
\\
\hline
$\beta$ & \fbox{-1} & 0 & 0 & 0 & 0
\\
$\gamma$ & 1 & \fbox{-1} & 0 & 0 & 0
\\
$\delta$ & 0 & 1 &  \fbox{-1} & 0 & 0
\\
$\epsilon$ & 0 & 0 & 1 &  \fbox{-1} & 0
\\
$\zeta_2$ & 0 & 0 & 0 & $t$ &  \fbox{$t$}
\\
$\alpha_0$ & 0 &0 &0 &0 & $-t^3$
\end{tabular}
\]
where $E' = A+B+C+D+E$. 
In the case of the persistent homology along $\alpha$, we see that the first four rows contribute nothing to the description of $H_0(\alpha)$, and that the last two contribute $\Sigma^2 \k[t]/ t$ and $\K[t]$ respectively, \ie
\[
H_0(\alpha) \cong \k \oplus \k \oplus \k^2 \oplus \k \oplus \cdots \oplus \k \oplus \cdots
\]

We obtain a similar result for the persistent homology along $\zeta$. Below are the standard and column-echelon forms of $\partial_1$ in this case:
\[
\begin{minipage}{0.5\textwidth}
\begin{tabular}{c|ccccc}
{}
&
$B$
&
$A$
&
$C$
&
$D$
&
$E$
\\
\hline
$\beta$ & -1 & 0 & 0 & 0 & 1
\\
$\gamma$ & 1 & -1 & 0 & 0 & 0
\\
$\delta$ & 0 & 1 & -1 & 0 & 0
\\
$\epsilon$ & 0 & 0 & 1 & -1 & 0
\\
$\zeta_2$ & 0 & 0 & 0 & 0 & $-t$
\\
$\alpha_0$ & 0 &0 &0 & $t^3$ & 0
\end{tabular}
\end{minipage}
\begin{minipage}{0.5\textwidth}
\begin{tabular}{c|ccccc}
{}
&
$B$
&
$A$
&
$C$
&
$D$
&
$E'$
\\
\hline
$\beta$ & \fbox{-1} & 0 & 0 & 0 & 0
\\
$\gamma$ & 1 & \fbox{-1} & 0 & 0 & 0
\\
$\delta$ & 0 & 1 & \fbox{-1} & 0 & 0
\\
$\epsilon$ & 0 & 0 & 1 & \fbox{-1} & 0
\\
$\zeta_2$ & 0 & 0 & 0 & 0 & \fbox{$-t$}
\\
$\alpha_0$ & 0 &0 &0 & $t^3$ & $t^3$
\end{tabular}
\end{minipage}
\]
We therefore obtain the same isomorphism class for $H_0(\zeta)$ as we did in the case of $\alpha$:
\[
H_0(\zeta) \cong \k \oplus \k \oplus \k^2 \oplus \k \oplus \cdots \oplus \k \oplus \cdots
\]

The other four dipaths yield a simple persistent homology since there is only one class which persists throughout the sequence. For example, along $\beta$, a homogeneous basis of $C_0$ is $\{\beta_0, \gamma_2, \epsilon,\zeta, \delta, \alpha \}$. The basis for $C_1$ is the same, but now $B$ has degree two rather than three. We obtain the following matrix representation of $\partial_1$:
\[
\begin{minipage}{0.5\textwidth}
\begin{tabular}{c|ccccc}
{}
&
$E$
&
$A$
&
$C$
&
$D$
&
$B$
\\
\hline
$\alpha$ & $-1$ & 0 & 0 & 0 & 0
\\
$\delta$ & 0 & 1 & $-1$ & 0 & 0
\\
$\zeta$ & 0 & 0 & 0 & 1 & 0
\\
$\epsilon$ & 0 & 0 & 1 & $-1$ & 0
\\
$\gamma_2$ & 0 & $-t$ & 0 & 0 & 1
\\
$\beta_0$ & $t^3$ & 0 &0 & 0 & $-t^2$
\end{tabular}
\end{minipage}
\begin{minipage}{0.5\textwidth}
\begin{tabular}{c|ccccc}
{}
&
$E$
&
$A$
&
$D$
&
$A+C$
&
$B$
\\
\hline
$\alpha$ & \fbox{$-1$} & 0 & 0 & 0 & 0
\\
$\delta$ & 0 & \fbox{1} & $0$ & 0 & 0
\\
$\zeta$ & 0 & 0 & \fbox{1} & 0 & 0
\\
$\epsilon$ & 0 & 0 & $-1$ & \fbox{$1$} & 0
\\
$\gamma_2$ & 0 & $-t$ & 0 & $-t$ & \fbox{1}
\\
$\beta_0$ & $t^3$ & 0 &0 & 0 & $-t^2$
\end{tabular}
\end{minipage}
\]
We see that the graded module $H_0(\beta)$ is simply $\k[t]$, since all pivots are of degree zero and there is one non-pivot row.

\subsection{Natural homology of the matchbox}

\label{sec:nathommatchbox}

Now we describe the natural homology diagram of the matchbox. First we will describe its restriction to the principal upset in $\tpos{\Xr}$ given by the constant path at the initial point $(0,0,0)$. The following diagram depicts the Hasse diagram of this upset:
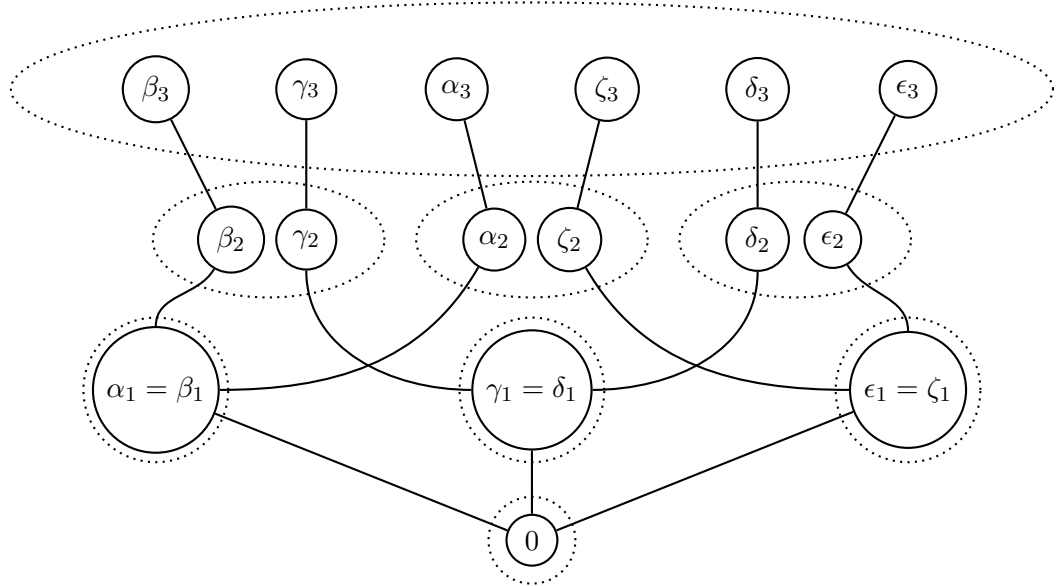
\[
\begin{tikzpicture}[nodes={draw, thick}]
%Nodes
\node[circle] (initial) at (0,-1) {$0$};
\node[circle] (c1) at (0,1) {$\gamma_1 = \delta_1$};
\node[circle] (a1) at (-5,1) {$\alpha_1 = \beta_1$};
\node[circle] (e1) at (5,1) {$\epsilon_1 = \zeta_1$};
\node[circle] (b2) at (-4,3) {$\beta_2$};
\node[circle] (c2) at (-3,3) {$\gamma_2$};
\node[circle] (d2) at (3,3) {$\delta_2$};
\node[circle] (e2) at (4,3) {$\epsilon_2$};
\node[circle] (a2) at (-.5,3) {$\alpha_2$};
\node[circle] (z2) at (.5,3) {$\zeta_2$};
\node[circle] (b3) at (-5,5) {$\beta_3$};
\node[circle] (c3) at (-3,5) {$\gamma_3$};
\node[circle] (d3) at (3,5) {$\delta_3$};
\node[circle] (e3) at (5,5) {$\epsilon_3$};
\node[circle] (a3) at (-1,5) {$\alpha_3$};
\node[circle] (z3) at (1,5) {$\zeta_3$};
%Edges
\draw [thick] (initial) to (c1);
\draw [thick] (initial) to (a1);
\draw [thick] (initial) to (e1);
\draw [thick] (a1) to [out=0, in=-120] (a2);
\draw [thick] (e1) to [out=-180, in=-60] (z2);
\draw [thick] (a1) to [out=90, in=-120] (b2);
\draw [thick] (e1) to [out=90, in=-60] (e2);
\draw [thick] (c1) to [out=-180, in=-90] (c2);
\draw [thick] (c1) to [out=0, in=-90] (d2);
\draw [thick] (b2) to (b3);
\draw [thick] (a2) to (a3);
\draw [thick] (c2) to (c3);
\draw [thick] (d2) to (d3);
\draw [thick] (e2) to (e3);
\draw [thick] (z2) to (z3);
%Groups
\draw[thick, dotted] (0,-1) ellipse (1.5em and 1.5em);
\draw[thick, dotted] (c1) ellipse (2.5em and 2.5em);F
\draw[thick, dotted] (-5,1) ellipse (2.5em and 2.5em);
\draw[thick, dotted] (5,1) ellipse (2.5em and 2.5em);
\draw[thick, dotted] (-3.5,3) ellipse (4em and 2em);
\draw[thick, dotted] (3.5,3) ellipse (4em and 2em);
\draw[thick, dotted] (0,3) ellipse (4em and 2em);
\draw[thick, dotted] (0,5) ellipse (18em and 3em);
\end{tikzpicture}
\]
Traces in the same \replaced{dotted}{red} circle yield the same trace space, \ie have the same beginning and end points. Each line corresponds to an extension. The natural homology diagram\added{, in which copies of $\k$ corresponding to the same trace spaces have been identified,} is depicted below, the arrows being induced by extensions:
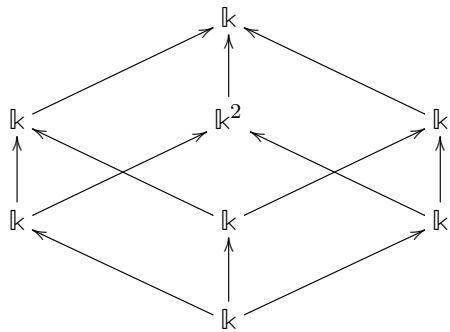
\[
\xymatrix@C=6em{
&
\k
&
\\
\k
	\ar[ur]
&
\k^2
	\ar[u]
&
\k
	\ar[ul]
\\
\k
	\ar[u]
	\ar[ur]
&
\k
	\ar[ul]
	\ar[ur]
&
\k
	\ar[u]
	\ar[ul]
\\
&
\k
	\ar[ur]
	\ar[u]
	\ar[ul]
&
}
\]
Each path of length $3$ in the above diagram corresponds to the persistence vector space obtained by taking the persistent homology along one of the maximal traces. All maps from $\k$ to $\k$ depicted in the diagram are identities, the two maps from $\k$ to $\k^2$ are canonical injection into the first (resp. the second) component of $\k^2$. The map from $\k^2$ to $\k$ is the sum of the two components of $\k^2$.

Comparing the above with the Hasse diagram of the trace poset of the matchbox, it is clear that \replaced{``}{''}glueing\replaced{''}{``} the persistence modules along chains while taking end-points into account relates the uni-dimensional persistence complexes to natural homology. This will be made explicit in the next section.

\section{Simultaneous persistence of directed spaces}

\label{sec:simulpers}

In this section, we give methods for amalgamating information from persistent homology along traces.
We recall constructions in the category of posets which will be useful when using persistent homology along traces to recover natural homology.
We will denote by $\Pos$ the category of posets and order preserving maps, and by $\Posin$ the wide subcategory with order preserving inclusions.

\subsection{Colimits of chains in posets}
\label{ss:colimitschains}

In order to reconstruct natural homology from persistent homology along traces, we must amalgamate information from each chain in the trace poset. For this reason, we discuss here how to \added{re}construct a poset from its (maximal) chains.

Firstly, notice that in general, using exclusively maximal chains is not enough to reconstruct a poset. Indeed, consider the poset whose Hasse diagram is depicted below:
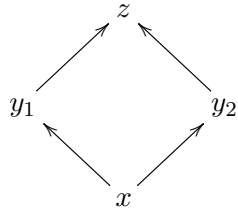
\[
\xymatrix{
&z&\\
y_1
	\ar[ur]
&
&
y_2
	\ar[ul]
\\
&x
	\ar[ur]
	\ar[ul]
&
}
\]
Its maximal chains are $(x,y_1,z)$ and $(x,y_2,z)$. In order to obtain the whole poset as a colimit, we additionally need the inclusions of $x$ and $z$ into these chains. %The minimal way of achieving this \todo{seems to be} by taking intersections, \ie pullbacks in the category $\Pos$ of posets and inclusions.

While a poset is in general not the colimit of its maximal chains, it is the colimit of all of its chains.
Given a poset $P$, we consider the thin subcategory $\Ch_P$ of $\Posin$ consisting of the chains of $P$ and their inclusions, called the \emph{poset of chains}. The \emph{diagram of chains} associated to $P$ is the inclusion functor $\dCh_P: \Ch_P \fl \Pos$. \replaced{The following result is well-known, we include a short proof for completeness:}{We have:}{}
\begin{proposition}%[\todo{ref or proof??}]
\label{prop:colimchain}
For any poset $P$,
\[
\colim {} \, \dCh_P = P.
\]
\end{proposition}
%\todo{I did not find a reference!}

%\todo{EG: Below is a sketch of proof ; in fact this is fully spelled out in the beginning of the proof of Proposition 5 ; should we put that part there, or give a sketch here, and use the heavier notations just in the proof of Proposition 5? }\new{CC: I agree, let's leave it like this - the only thing is that later we mention that we described the situation when we start the proof of Prop 5}

\begin{proof}
%The colimit of $\dCh_P$ is computed as follows, in $\Pos$: first we take the coproduct of all chains of $P$, which is the union of all sub-linear orders $(x^j_i)_{i\in I}$ of $P$ where $j$ ranges over all chains of $P$, for some indexing family $I$ and $x^j_i \in P$ for all $i \in I$. Then we identify all common subchains within the $(x^j_i)_{i\in I}$, in particular the elements themselves. Therefore we identify all $x^j_i$ within these linear orders that have to be identified, and take the transitive closure of the corresponding orders. This is indeed the poset $P$. 
$P$ is clearly a co-cone for $\dCh_P$. Let $(f_C : C \fl Q)_{C \in \Ch_P}$ be another. Define a map $f: P \fl Q$ by $f(x) = f_{\{x\}}(x)$.
Then $f_{\mid C}=f_C$ for all $C \in \Ch_{P}$. 
Indeed, since, for all $x \in C$, $\{x\}$ is a sub-chain of $C$, we have $f_{x}(x) = f_C(x)$, so $f$ commutes with all maps of the co-cone. 
Consider $x,y \in P$ such that $x \leq y$. 
Then $f(x) = f_{\{x,y\}}(x) \leq f_{\{x,y\}}(y) = f(y)$, meaning that $f$ is a morphism of posets. 
Given another such morphism of posets $g: \ P \rightarrow Q$ which commutes with maps of the co-cone, we have $g(x)=f_{x}(x)=f(x)$ for all $x\in P$ and hence $g=f$.
\end{proof}

When considering exclusively maximal chains, we may obtain $P$ by adding their intersections to the diagram. 
Consider a poset $P$ and two maximal chains $C_1, C_2$ of $P$. The pullback of $C_1, C_2$ in $\Pos$ corresponds to the full sub-poset of~$P$ given by the intersection $C_1 \cap C_2$, which is also a chain. Denote by~$\pmCh_P$ the subcategory of $\Posin$ consisting of maximal chains in $P$ and their pullbacks. By maximality of the considered chains, this category looks like a zig-zag.

\begin{proposition}
\label{prop:colimmaxchain}
A poset $P$ is the colimit of the diagram of its maximal chains $\pmCh_P \fl \Pos$.
\end{proposition}
\begin{proof}
This follow the same schema as the proof of Proposition~\ref{prop:colimchain}, but with a different map $f$. Consider again another co-cone $(f_C : C \fl Q)_{C \in \Ch_P}$. Define the map $f$ by sending $x\in P$ to $f_C(x)$ if $x$ is an element of only one maximal chain $C$, and to $f_{C_1 \cap C_2}(x)$ if $x \in C_1 \cap C_2$. The map $f$ commutes with the co-cone maps and is unique by construction. It is a morphism of posets since if $x \leq y$, then $x$ and $y$ are elements of the same maximal chain(s) and hence $f(x) = f_C(x) \leq f_C(y) = f(y)$ where $C$ is either a maximal chain or the intersection of two maximal chains.
\end{proof}

%In the case of trace posets for (nice) directed spaces, it suffices to consider the maximal chains and their prefix chains. \todo{Indeed, the up-set of a minimal element in the trace poset is a tree}. This is useful, since if we have persistent homology along maximal traces, then we of course will also have the persistent homology along all prefix traces (up to the filtration).

\subsection{Application to diagrams}

\label{sec:applidiag}

Now that we have assembled information on colimits of chains, we will see how these constructions carry over to diagrams over subposets. We fix a category~$\V$.

First, we introduce a category representing persistence objects of a certain type $\V$, without fixing the indexing poset. Specifically, this category, denoted by $\Pers{\V}$, has for objects pairs $(P, F: P \fl \V)$ where $P$ is a poset, \ie $F$ is a $P$-persistence object in $\V$. A morphism from $(P, F: P \fl \V)$ to $(Q, G: Q \fl \V)$ is a pair $(f, \sigma)$ where $f$ is a morphism $P \fl Q$ of $\Pos$ and $\sigma$ is a natural transformation \replaced{$F \Rightarrow G \circ f$}{$F \Rightarrow G \circ \phi$:}
\[
\xymatrix@R=1.5em{
P
	\ar[rr] ^-{f}
	\ar[ddr] _-{F}
&
&
Q
	\ar[ddl] ^-{G}
\\
&
\ar@<+0.95ex>@{}[ur]^(-.25){}="a"^(.45){}="b" 
\ar@{=>}^-{\sigma} "a";"b"
&
\\
&
\V
&
}
\]
Composition of morphisms is given by $(g, \tau) \circ (f, \sigma)$ = $(g \circ f, \tau_{f} \circ \sigma)$, and the identity on $(P, F)$ is the pair $(1_P , id_F)$.
%We denote by $\iPers\V$ the subcategory of $\Pers\V$ in which we only take morphisms $(\phi, \sigma)$ where $\sigma$ is a natural isomorphism.

Consider the full sub-category of $\opnat{\V}$, as defined in Section \ref{sec:tracecat}, consisting of natural systems $(\Cr, D)$ with values in $\V$ such that $\Fact\Cr$ is a poset. This category naturally embeds in the category of diagrams $\Pers{\V}$. %\new{$\lfl$ are they not exactly the same category?}

Let $P$ be a poset and an inclusion $G : \Phi \fl \Posin$, where $\Phi$ is a category of sub-posets of $P$ and their inclusions such that $\mathrm{colim}_\Phi G = P$. Suppose that we are given a functor $D : P \fl \V$, \ie an object of $\Pers{\V}$. Since the colimit of~$G$ is $P$, for each $\phi$ in $\Phi$, we have an inclusion $i_\phi: \phi \fl P$. Using these inclusions, we define a functor $F : \Phi \fl \Pers{\V}$, sending each $\phi$ to the restriction of $D$ to $\phi$, denoted by $F_\phi := (i_\phi)^*D$. 

Given an inclusion $i_{\phi, \phi'} : \phi \fl \phi'$ in  $\Phi$, the functor $F$ induces the morphism of persistence objects given by $F_{\phi, \phi'} = (i_{\phi,\phi'}, id_{F_\phi})$, where the latter is the identity natural transformation on $F_{\phi}$ since $F_{\phi'} \circ i_{\phi,\phi'} = F_\phi$.

Before proving the main technical result of this section, we prove useful properties of the colimit construction in $\Posin$. Recall that given posets $P_1, P_2$ and $p_i \in P_i$, a pointed inclusion of posets $i : (P_1,p_1) \fl (P_2, p_2)$ is an inclusion $i : P_1 \fl P_2$ such that $i(p_1) = p_2$. Additionally, we recall that given $p,p' \in P$, we say that $p$ is a lower cover of $p'$ when $p\leq 
 p'$, and any other element $q\in P$ such that $p\leq q \leq p'$ is equal to either $p$ or $p'$. We have the following:
\begin{lemma}
\label{lemma:Colimits}
Let $P$, $\Phi$ and $G$ be as defined in the above paragraph. We have
\begin{enumerate}
\item For any $p \in P$, there exists a sub-poset $\phi$ in $\Phi$ and an element $\pi \in \phi$ such that $i_\phi(\pi) = p$.
\item For any pair $(\phi,\pi)$, $(\phi', \pi')$ such that $i_\phi(\pi) = i_{\phi'}(\pi')$, there exists a zig-zag of pointed inclusions in $\Phi$ as pictured below:
\[
\xymatrix{
{(\phi, \pi)}
&
(\phi_1, \pi_1)
    \ar[r]
    \ar[l]
&
\cdots
&
(\phi_n, \pi_n)
    \ar[r]
        \ar[l]
&
(\phi', \pi').
}
\]
\item For any pair $p,p' \in P$ such that $p'$ is a cover of $p$, there exists $\phi$ in $\Phi$ and $\pi, \pi' \in \phi$ such that $pi\leq \pi'$, $i_\phi(\pi)=p$ and $i_\phi(\pi')=p'$.
\end{enumerate}
\end{lemma}
\begin{proof}
For the first point, if no such pair $(\phi, \pi)$ exists, the poset $P'$ obtained by removing $p$ from $P$ would be a co-cone under $G$, contradicting the assumption that~$P$ is the colimit.

For the second, if no such zig-zag exists, then the poset $P'$ obtained by creating two distinct copies of $p$ would be a co-cone for $G$, but there are two inclusions $P \fl P'$, contradicting the assumption that $P$ is the colimit.

Finally, the third item holds since if no such $\phi$ exists, then the the poset $P'$ defined by $P' = (P, \leq_P \setminus \{(p,p')\})$ would be a co-cone for $G$, again providing a contradiction.
\end{proof}

\begin{proposition}
\label{Proposition:Colimits}
Let $P$, $G$, $D$ and $F$ be as defined in the above paragraph. We have
\[
\colim{\Pers\V} F = D.
\]
\end{proposition}
\begin{proof}
We show that $D$ satisfies the colimit property. Firstly, it is a co-cone for $F$ since by construction, for any morphism $i_{\phi, \phi'} : \phi \fl \phi'$ of $\Phi$, the following diagram commutes in $\Pers\V$:
\[
\xymatrix{
F_\phi
    \ar[rr] ^-{(i_{\phi,\phi'}, id_{F_\phi})}
    \ar[dr] _-{(i_{\phi}, id_{F_\phi})}
&
&
F_{\phi'}
    \ar[dl] ^-{(i_{\phi'}, id_{F_{\phi'}})}
\\
&
D.
&
}
\]

Now to show that it is universal. Let $\Delta: Q \fl \V$ be another co-cone for $F$, \ie for any $i_{\phi, \phi'} : \phi \fl \phi'$, we have morphisms $F_\phi \fl \Delta$ and $F_{\phi'} \fl \Delta$  in $\Pers\V$ such that the following diagram commutes
\begin{align}
\label{E:DeltaCocone}
\raisebox{0.75cm}{\xymatrix{
F_\phi
    \ar[rr] ^-{(i_{\phi,\phi'}, id_{F_\phi})}
    \ar[dr] _-{(\psi_\phi, \tau_{\phi})}
&
&
F_{\phi'}
    \ar[dl] ^-{(\psi_{\phi'}, \tau_{\phi'})}
\\
&
\Delta.
&
}}
\end{align}
In particular, this means that $Q$ is a co-cone for $G$, since the outer triangle in the following diagram commutes:
\[
\xymatrix{
\phi
    \ar[rr] ^-{i_{\phi,\phi'}}
    \ar@/_/[ddr] _-{\psi_\phi}
    \ar[dr] ^-{i_\phi}
&
&
{\phi'}
    \ar@/^/[ddl] ^-{\psi_{\phi'}}
    \ar[dl] _-{i_{\phi'}}
\\
&
P \ar[d] ^-{f}
&
\\
&
Q.
&
}
\]
In particular, we obtain a unique inclusion of posets $f:P\fl Q$ such that all triangles in the above diagram commute.

By $(1)$ of Lemma~\ref{lemma:Colimits}, we know that there exists $(\phi, \pi)$ such that $i_\phi(\pi) = p$. Given two such distinct sub-posets $(\phi, \pi)$, $(\phi', \pi')$ such that a morphism $i_{\phi, \phi'}$ exists between them, the hypothesis that $\Delta$ is a co-cone for $F$ implies that the following diagram commutes:
\[
\xymatrix{
F_{\phi}(\pi) 
    \ar@{=}[r]
    \ar[d] _-{\tau_\phi(\pi)}
&
D_p
&
F_{\phi'}(\pi')
    \ar@{=}[l]
    \ar[d] ^-{\tau_{\phi'}(\pi')}
\\
\Delta_{\psi_{\phi}(\pi)}
    \ar@{=}[r]
&
\Delta_{(f(p))}
&
\Delta_{\psi_{\phi'}(\pi')}
    \ar@{=}[l]
}
\]
Therefore we have
\[
\tau_\phi(\pi) = \tau_{\phi'}(i_{\phi,\phi'}(\pi))\circ id_{F_\phi}(\pi)
= \tau_{\phi'}(\pi')\circ 1_{F_\phi(\pi)}
= \tau_{\phi'}(\pi'),
\]
in which the first equality holds because diagram (\ref{E:DeltaCocone}) commutes, the second by definition, and the third because $F_\phi(\pi) = F_{\phi'}(\pi')$. 

If no such map $i_{\phi, \phi'}$ exists, then by $(2)$ of Lemma~\ref{lemma:Colimits}, $\phi$ and $\phi'$ must be connected by a zig-zag of such inclusions. By induction on the length of such a zig-zag and using the above argument, we see that the equality $\tau_\phi(\pi) = \tau_{\phi'}(\pi')$ holds for any pair $(\phi,\pi)$, $(\phi', \pi'$) such that $i_\phi(\pi) = p = i_{\phi'}(\pi')$. We denote this morphism of $\V$ by $\sigma_p : D_p \fl \Delta_{f(p)}$.
Now, we must show that $\sigma$ defines a natural transformation between $D$ and $\Delta\circ f$.

Consider $p,p' \in P$ such that $p$ is a lower cover of $p'$. By $(3)$ of Lemma~\ref{lemma:Colimits}, there exists $\phi$ in $\Phi$ and $\pi, \pi' \in \phi$ such that $pi\leq \pi'$, $i_\phi(\pi)=p$ and $i_\phi(\pi')=p'$. Since $\tau_\phi$ is a natural transformation between $F_\phi$ and $\Delta\circ\psi_\phi$, the diagram below on the left commutes, in which the vertical arrows are those induced by the inequality $\pi \leq\pi'$.
 \[
 \xymatrix{
 F_\phi(\pi)
    \ar[r] ^-{\tau_\phi(\pi)}
    \ar[d] 
&
\Delta_{\psi_\phi(\pi)} \ar[d]
\\
 F_\phi(\pi')
    \ar[r] _-{\tau_\phi(\pi')}
&
\Delta_{\psi_\phi(\pi')}
 }
\qquad\qquad
  \xymatrix{
 D_p
    \ar[r] ^-{\sigma_p}
    \ar[d] 
&
\Delta_{f(p)}
\ar[d]
\\
 D_{p'}
    \ar[r] _-{\sigma_{p'}}
&
\Delta_{f(p')}
 }
 \]
 This means the right-most diagram must commute, since the two are equal. Now, given any comparison $p\leq p'$ in $P$, there exists a chain $p \leq p_1 \leq \cdots \leq p_n \leq p'$ in which each of the comparisons is a cover. Using induction on the length of this chain and the above argument, we see that $\sigma$ is indeed a natural transformation.

To conclude, we have found a morphism $(f,\sigma): D \fl \Delta$ of $\Pers\V$ which commutes with the co-cone maps. This morphism is unique since $f$ was obtained by the colimit property of $P$, and, given any other candidate $\sigma'$ for the natural transformation, its components would have to coincide with the $\tau_\phi$ on appropriate sub-posets, meaning that it is equal to $\sigma$. In other words, $D$ is the universal co-cone for $F$, thereby concluding the proof.
\end{proof}

\subsection{Natural homology as a colimit}

\label{sec:nathomupset}

Using Proposition~\ref{Proposition:Colimits} and constructions similar to those explained above, we obtain the following results:

\begin{theorem}
\label{thm:1}
Let $\Xr = (X,dX)$ be a compact pospace\added{,} $\alpha$ a point in $X$\added{,} and $\Vect$ the category of $\k$-vector spaces.
\begin{enumerate}
\item The natural homology of $\Xr$ is the colimit in $\Pers{\Vect}$ of the persistent homology along each of its traces.
\item The natural homology of $\Xr$ is the colimit in $\Pers{\Vect}$ of the persistent homology of its maximal traces, seen as maximal chains in $\tpos \Xr$, completed with pullbacks.
\item The natural homology of the up-set of $\alpha$, seen as a constant trace, in $\tpos \Xr$ is the colimit in $\Pers{\Vect}$ of the persistent homologies along the maximal chains passing through $\alpha$, completed with pullbacks.
\end{enumerate}
\end{theorem}
\begin{proof}
These are direct consequences of Propositions~\ref{Proposition:Colimits} and~\ref{prop:colimchain}.
\end{proof}

\begin{example}
Here we illustrate the above construction in the case of a particular directed graph using the initial point filtration. Consider the following directed acyclic graph $G$:
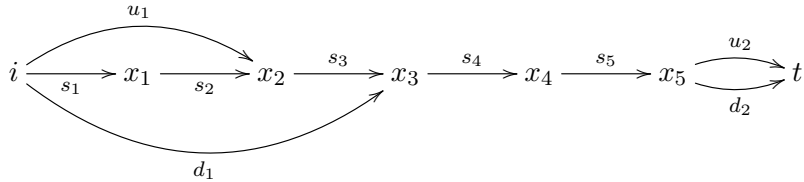
\[
\xymatrix@C=3em{
i
\ar[r] _-{s_1}
\ar@/^1.5pc/[rr] ^-{u_1}
\ar@/_2.5pc/[rrr] _-{d_1}
&
x_1
\ar[r]_-{s_2}
&
x_2
\ar[r]^-{s_3}
&
x_3
\ar[r]^-{s_4}
&
x_4
\ar[r]^-{s_5}
&
x_5
\ar@/_/[r] _-{d_2}
\ar@/^/[r] ^-{u_2}
&
t
}
\]
This graph contains six maximal discrete traces, considered as words over the labels on the edges. Since we will consider initial point filtrations, we depict below the Hasse diagram of the up-set of the constant trace on $i$ as a subposet of the trace poset:
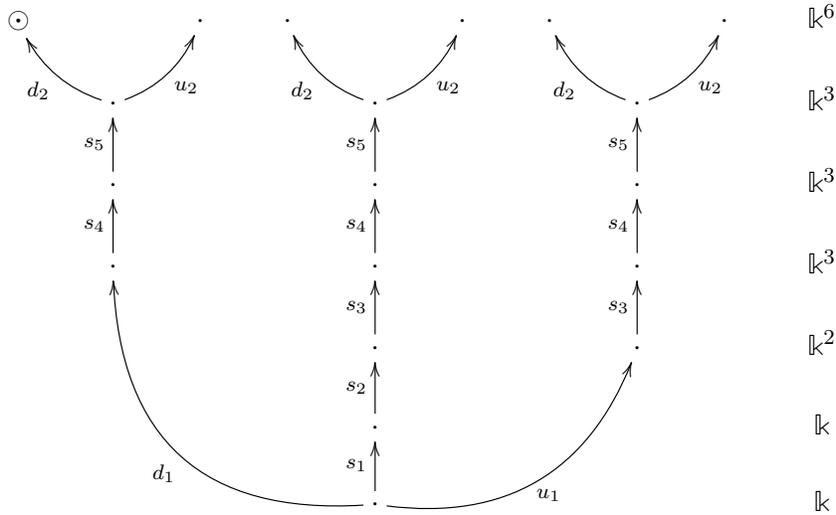
\[
\xymatrix@R=1.4em{
\odot
&&
\cdot
&
\cdot
&&
\cdot
&
\cdot
&&
\cdot
&
\k^6
\\
&
\cdot \ar@/_/[ur]_{u_2} \ar@/^/[ul]^{d_2}
&&&
\cdot \ar@/_/[ur]_{u_2} \ar@/^/[ul]^{d_2}
&&&
\cdot \ar@/_/[ur]_{u_2} \ar@/^/[ul]^{d_2}
&
&
\k^3
\\
&
\cdot \ar[u] ^{s_5}
&&&
\cdot\ar[u] ^{s_5}
&&&
\cdot\ar[u] ^{s_5}
&
&
\k^3
\\
&
\cdot\ar[u] ^{s_4}
&&&
\cdot\ar[u] ^{s_4}
&&&
\cdot\ar[u] ^{s_4}
&
&
\k^3
\\
&&&&
\cdot\ar[u] ^{s_3}
&&&
\cdot\ar[u] ^{s_3}
&
&
\k^2
\\
&&&&
\cdot\ar[u] ^{s_2}
&&&&
&
\k
\\
&&&&
\cdot
\ar[u] ^{s_1} 
\ar@/^3pc/[uuulll] ^-{d_1}
\ar@/_2pc/[uurrr] _-{u_1}
&&&&
&
\k
}
\]
The bottom-most node in the Hasse diagram corresponds to the constant trace on $i$, and each of the others corresponds to the trace built from the labels on the arrows in the above diagram, starting at the bottom-most node and ending at the node in question. For example, the circled node corresponds to the (maximal) path $d_1 \cdot s_4 \cdot s_5 \cdot d_2$. All nodes at the same height correspond to paths with the same beginning- and end-points, the degree of which is noted on the right.

We first consider the initial filtrations associated to the maximal traces
\begin{align*}
\alpha := d_1 \cdot s_4 \cdot s_5 \cdot u_2 
\qquad\text{and}\qquad
\alpha' := d_1 \cdot s_4 \cdot s_5 \cdot d_2.
\end{align*}
The persistence vector space $\sysh{0}{G}_\alpha$ obtained along $\alpha$ is
\[
\xymatrix{
\k
\ar[r]
&
\k^3
\ar[r]
&
\k^3
\ar[r]
&
\k^3
\ar[r]
&
\k^6,
}
\]
where each of the linear maps are given by the one-step extensions along $\alpha$ and all are injective. The persistence vector space associated to $\alpha'$ is identical, except for the final map $\k^3 \fl \k^6$. Indeed, the map induced along $\alpha$ coincides with that induced along $\alpha'$ on two of the basis elements of $\k^3$, but have distinct images for the third. 

This means that in the colimit, the first three maps will be identified, giving a subdiagram of the form:
\[
\xymatrix@R=.5em{
\k
\ar[r]
&
\k^3
\ar[r]
&
\k^3
\ar[r]
&
\k^3
\ar@/^1pc/[r]
\ar@/_1pc/[r]
&
\k^6,
}
\]
where we have identified the two copies of $\k^6$.

Applying similar reasoning to the other maximal paths, the colimit we obtain is given below:
\[
\xymatrix@C=3em{
\k
\ar[r] _-{}
\ar@/^1.5pc/[rr] ^-{}
\ar@/_2.5pc/[rrr] _-{}
&
\k
\ar[r]_-{}
&
\k^2
\ar[r]^-{}
&
\k^3
\ar[r]^-{}
&
\k^3
\ar[r]^-{}
&
\k^3
\ar@/_/[r] _-{}
\ar@/^/[r] ^-{}
&
\k^6.
}
\]
We have identified copies of $\k^n$ corresponding to the same end-points in order to visually relate the natural homology diagram to the graph $G$.
\end{example}
    
The persistent homology along a path in a directed acyclic graph tells us, at each step, how many parallel paths reach the same point. In this sense, the birth of generators in the persistent homology along a given trace $\alpha$ indicate the existence of a trace which has ``merged'' with $\alpha$ at that point. Long \replaced{``stagnant'' periods}{periods of ``sterility''}, \ie intervals $[k,k']$ in which no births occur, indicate a portion of the trace which has no parallel paths for that portion. In this sense, such ``stagnant'' \deleted{or ``sterile''} intervals indicate parts of the graph which have only one path connecting them. Furthermore, in the above example, there is only one \emph{source}, \ie node with only outgoing edges, and one \emph{sink}, \ie node with only incoming edges. In a finite directed acyclic graph, there are only finitely many pairs of sources and sinks, and applying the above to each pair yields the natural homology diagram restricted to up-sets, in the trace poset, of sinks.

\begin{example}
We are now considering the case of a cube of dimension $3$, minus its interior, where we \replaced{label}{labelled} two points, the initial point $0$ and the final point $1$, as well as the twelve arrows $x$, $y$, $z$, $xy$, $xz$, $yx$, $yz$, $zx$, $zy$, $xyz$, $yxz$ and $zxy$:
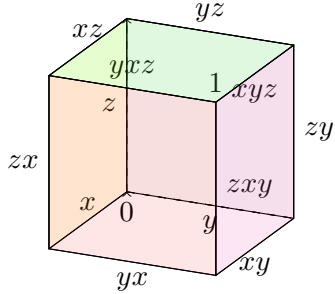
\begin{center}
\begin{tikzpicture}[scale=3,tdplot_main_coords]
    \coordinate (O) at (0,0,0);
    \tdplotsetcoord{P}{1.414213}{54.68636}{45}
    % A l'origine que des -> dans les arrow plus bas, jusqu'au STOP ->
    \draw[dashed,fill=gray!50,fill opacity=1] (O) edge node[below] {$y$} (Py);
%\draw (P) node {$1$};
    \draw[dashed,fill=gray!50,fill opacity=1] (O) node[below] {$0$} -- (Py);    
    \draw[fill=gray!50,fill opacity=1] (Py) edge node[below] {$xy$} (Pxy);
    \draw[fill=gray!50,fill opacity=1] (Pxy) edge node[right]{$zxy$} (P);
\draw[fill=gray!50,fill opacity=1] (O) edge node[above]{$x$} (Px);
\draw[dashed,->,fill=yellow,fill opacity=1] (O) edge node[left]{$z$} (Pz);
    \draw[fill=green,fill opacity=1] (Py) edge node[right] {$zy$} (Pyz);
    \draw[fill=red,fill opacity=1] (Px) edge node[below] {$yx$} (Pxy);
    \draw[fill=green,fill opacity=1] (Pz) edge node[above] {$yz$} (Pyz);
    \draw[fill=red,fill opacity=1] (Pz) edge node[above] {$xz$} (Pxz);
    \draw[fill=magenta,fill opacity=1] (Px) edge node[left] {$zx$} (Pxz); 
\draw[fill=magenta,fill opacity=1] (Pxz) edge node[above] {$yxz$} (P); 
\draw[fill=magenta,fill opacity=1] (Pyz) edge node[below] {$xyz$} (P); % STOP ->
\draw (P) node[above] {$1$};
\draw[fill=gray!50,fill opacity=0.1] (O) -- (Py) -- (Pyz) -- (Pz) -- cycle;
    \draw[fill=yellow,fill opacity=0.1] (O) -- (Px) -- (Pxz) -- (Pz) -- cycle;
    \draw[fill=green,fill opacity=0.1] (Pz) -- (Pyz) -- (P) -- (Pxz) -- cycle;
    \draw[fill=red,fill opacity=0.1] (Px) -- (Pxy) -- (P) -- (Pxz) -- cycle;
    \draw[fill=magenta,fill opacity=0.1] (Py) -- (Pxy) -- (P) -- (Pyz) -- cycle;
\end{tikzpicture}
\end{center}
We now consider, as in the previous example, the sub-diagram of the natural homology, for traces beginning at the initial state 0. 

We start with the first natural homology group: the unidimensional persistent homology value for any filtration of any maximal trace starting at 0 is just~$\k$, and its colimit just identifies all these copies of $\k$ for each prefix of maximal traces that are the same. This colimit is indeed isomorphic to the natural homology in dimension 1 of the cube, which is $\k$ for each trace starting at 0. 

We examine now the case of the second natural homology group: the unidimensional persistence diagram for any filtration of any maximal trace starting at 0 is $0 \rightarrow \k$ ($\k$ corresponds to the case of a maximal trace from 0 to 1), and its colimit just identifies all these copies of $\k$ for \added{a} maximal trace from 0 to 1. This colimit is indeed isomorphic to the natural homology in dimension 2 of the cube, which is $0$ for each trace starting at 0 and ending at a point which is not 1, and which is equal to $\k$ for maximal traces from 0 to 1. 
%and indeed, all unidimensional homologies, 
%Consider the following types of traces: 
%\begin{itemize}
%\item (a) traces from $0$ to $1$ going through edges among the twelve edges $x$, $y$ etc.
%\item (b) traces from a point of the interior of $x$ to any point in the interior of $xyz$, or from a point of the interior of $y$ to any point in the interior of $yxz$, or from a point of the interior of $z$ to any point in the interior of $zxy$
%\end{itemize}

%We see that that, for each trace of type: 
%\begin{itemize}
%\item (a): the unidimensional zeroth homologies of any filtration is $k$ (all traces spaces between two points along any trace from $0$ to $1$ are dihomotopic 
%\item (b): the unidimensional zeroth homologies of any filtration is $k \rightarrow k^2$ since, except for the trace space between points in the interior of $x$ to the interior of $xyz$ (resp. from the interior of $y$ to the interior of $yxz$ and from $z$ to the interior of $zxy$) where there are two non-dihomotopic traces \cite{Faj06}, the other trace spaces between points along traces of type $b$ is $k$ 
%\end{itemize}
%(...)
%For the second directed homology group, things are even simpler. For any filtration of maximal traces (i.e. traces from $0$ to $1$), the corresponding persistent diagram is $0 \rightarrow k$ since the 1-dimensional hole in the trace space appears only when considering the traces from $0$ to $1$. And any filtration of non-maximal trace generates $0$ as a persistence diagram. 
%
%Colimit? 
\end{example}

\begin{example}
Consider Fahrenberg's matchbox, which is the sub-space of the cube $[0,1]^3$, where we leave out the interior of the cube, and the bottom face (which is the convex hull of points $(0,0,0)$, $(0,1,0)$, $(1,0,0)$ and $(1,1,0)$). 
We recall from Section \ref{sec:nathommatchbox}, that the sub-diagram of the natural homology in dimension 1 of Fahrenberg's matchbox example, for traces starting at the initial point is: 
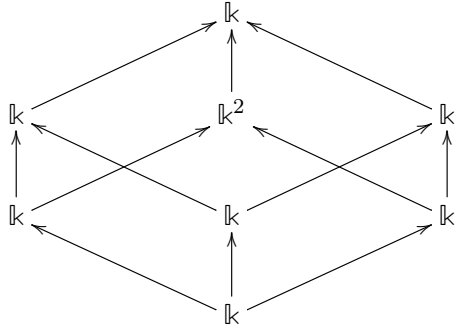
\[
\xymatrix@C=6em{
&
\k
&
\\
\k
	\ar[ur]
&
\k^2
	\ar[u]
&
\k
	\ar[ul]
\\
\k
	\ar[u]
	\ar[ur]
&
\k
	\ar[ul]
	\ar[ur]
&
\k
	\ar[u]
	\ar[ul]
\\
&
\k
	\ar[ur]
	\ar[u]
	\ar[ul]
&
}
\]    
This can also be obtained by applying Theorem \ref{thm:1} as follows. The unidimensional persistence diagrams for any filtration of any maximal trace starting at the initial state $(0,0,0)$ and finishing at the final trace $(1,1,1)$ are:
\begin{itemize}
\item Either $\k$, if the trace never intersects the half-open segment going from $(1,1,0)$ to $(1,1,1)$ (without the last point),
\item Or $\k \rightarrow \k^2 \rightarrow \k$, otherwise.
\end{itemize}
There again, the colimit identifies the copies of $\k$ (resp. $\k^2$) on equal prefixes of maximal trace, and we recover the diagram above. 
\end{example}

\section{Further remarks} %, and extensions}

\label{sec:remarks}

%\todo{Just very drafty ideas, that I should think about, and we could talk about them next time}

%Having shown the technical similarities between natural homology and multidimensional homology, it is tempting to discuss some of the features that have been introduced in one, but not studied in the other. 

%\subsection{Categories of natural systems and of persistence objects in $\ab$}

%Note that the category $\Pers{\Cr}$ defined in Section \ref{sec:applidiag}, 
%has for objects pairs $(P, F: P \fl \Cr)$ where $P$ is a poset (\ie $F$ is a $P$-persistence object in $\Cr$), hence defines functors from posets to some category $\Cr$ that we take here, as usual in classical persistence theory, to be real vector spaces. Real vector spaces being a subcategory $\Vect$ of $\ab$, the concept of bisimulation of Section \ref{sec:natural} applies. 

In this section we discuss two lines of research concerning the \replaced{persistent}{persistence} homology of trace spaces and its potential application.

\subsection{Metric structure on natural homology}
\label{ss:MetricStructure}

An important feature of persistent homology is that it gives information about a space in terms of its metric and not just its topology. We show that the construction of \replaced{persistent}{persistence} homology bears similarities with the one of natural homology, \deleted{construction} with particular functors from a factorization category to $\Top$. One may ask whether, when the directed space $X$ we consider has a metric space structure, if there is a metric structure that can be exploited in the corresponding persistence modules. 

In the case of the geometric realization of finite geometric precubical sets without loops, as we saw in Section \ref{SS:dspaces}, we can define the $l_1$-arc length $l_1(p)$ of a directed path $p$. From this, we can define a distance between elements of the trace poset, or more precisely, we can see the trace poset as a weighted poset as in %Section 3.3 of 
\cite{relativeinterleavings,bubenik} as follows: $w(p,q)=0$ if $p\not \leq q$, otherwise $q=upv$ and $w(p,q)=max\{l_1(u),l_1(v)\}$. Interleaving distances in such contexts are expressed in terms of translations on posets, that vastly generalize the translations \replaced{$t \mapsto t+\epsilon$}{$t \in \R \rightarrow t+\epsilon$} \replaced{which define the}{that defines that} usual interleaving distance in unidimensional persistence. 
%\todo{From this weighted poset, we should be able to derive an interleaving distance for our natural homology, from \cite{relativeinterleavings}?}
%\todo{Can be done also directly on the finite precubical set if we want}
Here, we will only briefly discuss the classical interleaving distance in unidimensional persistence homology, extracted along a trace of a directed \replaced{space}{spaces}, that will be induced by translations on the trace poset. The construction of interleaving distances for the full persistence object derived from natural homology is left for future work. 

Take two directed spaces $\Xr$ and $\Yr$, geometric realizations of finite geometric precubical sets without loops, $p$ a trace of $\Xr$ and $q$ a trace of $\Yr$. %Suppose we have a directed map $f$ from $\Xr$ to $\Yr$. 
We also suppose we are given a chain $(p_i)_{i\in\Real}$ in the interval $[p(0),p]$, and a chain $(q_i)_{i \in \Real}$ in the interval $[q(0),q]$. 
\replaced
{
In what follows, we write $p_i$ (resp. $q_i$) for the trace going from $x^p_i$ to $y^p_i$ (resp. from $x^q_i$ to $y^q_i$) and $(\alpha^p_{i,j},\beta^p_{i,j})$ (resp. $(\alpha^q_{i,j},\beta^q_{i,j})$) for the morphism from $p_i$ to $p_j$ (resp. from $q_i$ to $q_j$).
}
{
In what follows, we write $p_i$ to be the trace going from $x^p_i$ to $y^p_i$ (resp. $q_i$ the trace going from $x^q_i$ to $y^q_i$) and $(\alpha^p_{i,j},\beta^p_{i,j})$ the morphism from $p_i$ to $p_j$ (resp. 
$(\alpha^q_{i,j},\beta^q_{i,j})$ the morphism from $q_i$ to $q_j$).
}

%\todo{More interesting to index things with $\R$ here, to be discussed - this encodes also diagrams over $\N$, finite diagrams etc.}

%Note that in order to derive interesting information about the spaces $\Xr$, $\Yr$ and the map $f$, as metric spaces, we need to make assumptions on the filtrations we are using. We are going to suppose that the indexing $i\in \R$ in the partition corresponds
%to the arc length of the corresponding directed paths, i.e. , for all $i \in \R$, $i=l_1(p_i)=l_1(q_i)$. 

We construct by Proposition \ref{prop:1} a persistence $\k$-vector space $\sysh{n}{\Xr}_{p_i}$ indexed by the chain $(p_i)_{i\in\Real}$ and a persistence $\k$-vector space $\sysh{n}{\Yr}_{q_i}$ indexed by the chain $(q_i)_{i\in\Real}$. 
In particular we have linear maps
%\replaced
{
\begin{eqnarray*}
\sysh{n}{\Xr}_{(\alpha^p_{i,j},\beta^p_{i,j})} :
\sysh{n}{\Xr}_{p_i} \rightarrow \sysh{n}{\Xr}_{p_j} \\
\text{(resp.\;\;} 
\sysh{n}{\Yr}_{(\alpha^q_{i,j},\beta^q_{i,j})} :
\sysh{n}{\Yr}_{q_i} \rightarrow \sysh{n}{\Yr}_{q_j} \;).
\end{eqnarray*}
}
%{
%$\sysh{n}{\Xr}_{(\alpha^p_{i,j},\beta^p_{i,j})}$ 
%from $\sysh{n}{\Xr}_{p_i}$ to $\sysh{n}{\Xr}_{p_j}$
%(resp.  $\sysh{n}{\Yr}_{(\alpha^q_{i,j},\beta^q_{i,j})}$ 
%from $\sysh{n}{\Yr}_{q_i}$ to $\sysh{n}{\Xr}_{q_j}$). 
%}

The \emph{interleaving distance} between $\sysh{n}{\Xr}_{p_i}$ and $\sysh{n}{\Yr}_{q_i}$ is defined as follows.
We say that $\sysh{n}{\Xr}_{p_i}$ and $\sysh{n}{\Yr}_{q_i}$ are \emph{$\epsilon$-interleaved}, for some $\epsilon \geq 0$\added{,} if we have families of linear maps
$(\phi_ i: \ \sysh{n}{\Xr}_{p_i} \rightarrow \sysh{n}{\Yr}_{q_{i+\epsilon}})_{i \in \Real}$ 
and $(\psi_i: \ \sysh{n}{\Yr}_{q_{i}} \rightarrow \sysh{n}{\Xr}_{p_{i+\epsilon}})_{i \in \Real}$ such that for all $i \leq j \in \Real$:
\begin{eqnarray}
\phi_{j} \circ 
\sysh{n}{\Xr}_{(\alpha^p_{i,j},\beta^p_{i,j})} & = & \sysh{n}{\Yr}_{(\alpha^q_{i+\epsilon,j+\epsilon},\beta^q_{i+\epsilon,j+\epsilon})} \circ \phi_i \label{eq:interleave1}\\
\psi_{j} \circ 
\sysh{n}{\Yr}_{(\alpha^q_{i,j},\beta^q_{i,j})} & = & \sysh{n}{\Xr}_{(\alpha^p_{i+\epsilon,j+\epsilon},\beta^p_{i+\epsilon,j+\epsilon})} \circ \psi_i \label{eq:interleave2} \\
\psi_{i+\epsilon} \circ \phi_{i} & = & \sysh{n}{\Xr}_{(\alpha^p_{i,i+2\epsilon},\beta^p_{i,i+2\epsilon})} \label{eq:interleave3} \\
\phi_{i+\epsilon} \circ \psi_{i} & = & \sysh{n}{\Yr}_{(\alpha^q_{i,i+2\epsilon},\beta^q_{i,i+2\epsilon})} \label{eq:interleave4}
\end{eqnarray}

Then the interleaving distance between $\sysh{n}{\Xr}_{p_i}$ and $\sysh{n}{\Yr}_{q_i}$ is the infimum of all $\epsilon \geq 0$ such that $\sysh{n}{\Xr}_{p_i}$ and $\sysh{n}{\Yr}_{q_i}$ are $\epsilon$-interleaved, $\infty$ if there is no such $\epsilon$.

There are similarities between bisimulation equivalent natural homologies\deleted{,} and having non-infinite interleaving distance, as we show below. 

When there exists a dihomeomorphism $f: \ \Xr \rightarrow \Yr$, we know from Lemma \ref{lem:dihomeobisim} that their corresponding natural homologies are bisimulation equivalent, and we now show that there are particular relations between
the persistent objects $\sysh{n}{\Xr}_{p_i}$ and $\sysh{n}{\Yr}_{q_i}$. 
Then, for any chain $(p_i)_{i\in\Real}$ in the interval $[p(0),p]$, for any dipath $p$ in $\Xr$, $\sysh{n}{\Xr}_{p_i}$ and $\sysh{n}{\Yr}_{f^*(p_i)}$ have interleaving distance 0.

Indeed, we take $\phi_t: \ \sysh{n}{\Xr}_{p_t} \rightarrow \sysh{n}{\Yr}_{f^*(p_t)}$ to be $[f^*(x)]$, for each $[x] \in \sysh{n}{\Xr}_{p_t}$, the homology class of $f^*(x)$, for $[x]$ the homology class of any $x$ cycle in ${\mathfrak{T}}(\Xr)$ from $p_t(0)$ to $p_t(1)$.
%\replaced
{
Equations \eqref{eq:interleave3} and \eqref{eq:interleave4} are trivially verified since
$(f^{-1})^*\circ f^*=f^* \circ (f^{-1})^*$ and the maps
$\sysh{n}{\Xr}_{(\alpha^p_{i,i},\beta^p_{i,i})}$ and
$\sysh{n}{\Yr}_{(\alpha^{f^*(p)}_{i,i},\beta^{f^*(p)}_{i,i})}$ are identities (since $\alpha^p_{i,i}$, $\beta^p_{i,i}$, $\alpha^{f^*(p)}_{i,i}$ and $\beta^{f^*(p)}_{i,i}$ are constant dipaths). 
Similarly, Equations (\ref{eq:interleave1}) and (\ref{eq:interleave2}) are trivially satisfied since the maps $\sysh{n}{\Xr}_{(\alpha^p_{i,i},\beta^p_{i,i})}$,
$\sysh{n}{\Yr}_{(\alpha^{f^*(p)}_{i,i},\beta^{f^*(p)}_{i,i})}$ are identities.
}
%{
%Similarly, 
%The equations (\ref{eq:interleave3}) and (\ref{eq:interleave4}) are trivially verified since
%$(f^{-1})^*\circ f^*=f^* \circ (f^{-1})^*$ and $\sysh{n}{\Xr}_{(\alpha^p_{i,i},\beta^p_{i,i})}$,
%$\sysh{n}{\Yr}_{(\alpha^{f^*(p)}_{i,i},\beta^{f^*(p)}_{i,i})}$ are identities (since $\alpha^p_{i,i}$, $\beta^p_{i,i})$,
%$\alpha^{f^*(p)}_{i,i}$ and $\beta^{f^*(p)}_{i,i})$ are constant dipaths. 
%Simarly, Equations (\ref{eq:interleave1}) and (\ref{eq:interleave2}) are trivially satisfied since $\sysh{n}{\Xr}_{(\alpha^p_{i,i},\beta^p_{i,i})}$,
%$\sysh{n}{\Yr}_{(\alpha^{f^*(p)}_{i,i},\beta^{f^*(p)}_{i,i})}$ are identities
%}

In a \replaced{similar}{smilar} manner, we may prove that for any chain $(q_i)_{i\in\Real}$ in the interval $[q(0),q]$, for any dipath $q$ in $\Yr$, $\sysh{n}{\Xr}_{(f^{-1})^*(q_i)}$ and $\sysh{n}{\Yr}_{q_i}$ have interleaving distance 0.

\deleted{Of course, we miss an important point about persistence if we do not insist on considering only representative\added{s} $p$ of $\tr{p}$ that are parameterized by the arc-length. In that case, the interleaving distance is zero only when $f$ is an isometry.}

\deleted{
For geometric realizations $\Xr$ of finite geometric precubical sets $S$, there are natural $\N$-persistence objects, given by, for the geometric realization $p$ in $\Xr$ of any combinatorial directed edge-path $p_S$ in of $S$ the filtration of $[p(0),p]$ in the trace poset of $\Xr$ given by $C_i$ being the geometric realization of the prefix with $i$ steps of $p_S$. Considering this $\N$-persistence object as a $\Real$-persistence object
$\sysinth{n}{\Xr}_{p}$, which is constant between two integers, we get that 
%
%\begin{lemma}
the interleaving distance between $\sysh{n}{\Xr}_{p}$ and $\sysinth{n}{\Xr}_{p}$ is at most 1.
%\end{lemma}
%\begin{proof}
%(...)
%\end{proof}
Whereas for such geometric realizations it is known (see \cite{naturalhomology}) that the \replaced{``}{"}combinatorial" natural homology of $S$ is bisimulation equivalent to the natural homology of $\Xr$ only in some particular cases, e.g. when $S$ is a cubical complex. 
}

\subsection{Algorithmical considerations}

\label{SS:algo}

In this section, we sketch practical calculations of the persistence module that corresponds, by Theorem \ref{thm:1}, to natural homology, in the case of directed spaces arising in concurrency theory, which was the original motivation of this work. 
It has been shown in e.g. \cite{fajstrup16} that concurrent languages, such as the shared memory PV language, can be given semantics in pre-cubical sets, hence a directed space. This has been been implemented in such software as ALCOOL \cite{alcool} and oplate \cite{oplate}. It has been shown that under some hypotheses, \replaced{valid}{validated} for such concurrent languages as the PV language, the natural homology computed on pre-cubical sets $K$, which is the homology of the diagram of \replaced{trace}{traces} spaces $\trac{K}{f}$, for all $a$, $b$ vertex in $K_0$, and $f$ dipath from $a$ to $b$ in $K_1$, is bisimulation equivalent to the natural homology of the dispace, geometric realization of $K$, \cite{naturalhomology}. 

For finite pre-cubical sets, Raussen \cite{rausimo2} and Ziemianski \cite{ziemianski} showed that singular homology
groups of trace spaces such as $\trace{K}{a}{b}$ (which is isomorphic to all $\trac{K}{f}$, $f$ dipath from $a$ to $b$) 
are computable, by calculating a finite presentation of the trace spaces
(prod-simplicial, simplicial or CW-complex) from which we can compute
homology using \added{the} Smith normal form of matrices. This has also been implemented in e.g. 
\cite{calctracespace1}. 
%The problem is that this
%construction is not nicely behaved in general with respect to changes of base points $a$,
%$b$, hence making difficult the construction of a (poset) filtration of spaces. 

%\pmin{rajouter la remarque sur une homologie unidimensionnelle d'une trace}

For ``nice'' precubical sets $X$ such as the ones given by the PV language (without loops), and for any vertices $a$ and
$b$ in $X$, there is a way to get a finite combinatorial model $T(X)(a,b)$ 
(a finite CW-complex, or a finite simplicial set) 
that is homotopy equivalent to the trace space of 
$X$ from $a$ to $b$, $\trace X a b$, which is 
both functorial in $X$, $a$, $b$ and also minimal among such functors~\cite{ziemianski}. Algorithmically, this relies on the cells of the CW-complex being identified with certain combinatorial paths in $X$, as was exemplified in Section~\ref{SS:MatchboxExample}. 
This algorithm allows to construct the (poset) filtration of CW-complexes (or homotopy equivalent simplicial sets) $T(X)(a,b)$ and therefore compute their homology, for each $a$ and $b$ in $X$. Thus this would algorithmically build the persistence module, or natural homology, of the precubical set~$X$. 
%\todo{I can easily show what the bar code of the iterated snapshot model, used in distributed computing, with respect to e.g. its diagonal path. I could show for instance that the homology types are not going to be the same as the ones we would get with iterated atomic read/write on one register for instance?}
Of course this is a rather naive algorithm, in that we compute homology separately for each pair of points. We postpone the discussion of non naive algorithms similar to rank invariant computations for a future article. 

%\cite{calctracespace1} going from cubical semantics (book) to the calculation of each trace space as a prod-simplicial set. But better with Kris version? (for the filtration?)
%Mention oplate? 

\section{\replaced{F}{Conclusion and f}uture work}

%\out{
In \cite{CameronGoubaultMalbosHHA22}, we related natural homotopy, a natural system of homotopy groups associated to a dispace, to group objects in a certain category. This was achieved via the notion of composition pairing, introduced by Porter and largely inspired by lax functors. This extra structure of composition pairing is also present in natural homology, and we believe such \replaced{``}{"}homological" operations may have some interesting counterparts in persistent homology, which we are planning on studying in future work. 

In hybrid systems theory, there is a natural sheaf theoretic formulation of dynamical systems, starting from \cite{lawvere1}, put in action (in particular, concerning temporal logic formulas and verification) and extended in 
\replaced{\cite{schultz2019dynamical,schultz2017temporal}}{\cite{schultz2019dynamical} \cite{schultz2017temporal}}. The formalization bears similarities with natural homology for directed spaces, except that it takes a sheaf-theoretic view, leading to contravariant, instead of covariant, functors. 

Consider the monoid $(\Real_+,+)$, seen as a one-object category $ \Real$. Its factorization category ${\cal F}\Real$ has as objects, intervals $[0,l]$ for $l \in \Real$, and as morphisms, inclusions of such intervals into larger intervals $[0,l']$, translated by some value $k\leq l'-l$. Presheaves on ${\cal F}\Real$ can be endowed with a Grothendieck topology, which is the Johnstone topology \cite{johnstone1}, so that sheaves correspond to dynamical systems (in some ways, \replaced{``}{"}continuous graphs" described locally by gluings of paths of various lengths, that agree on subpaths). 

Therefore, dynamical systems on some metric space $X$ are considered as being particular presheaves, \emph{i.e.} functors \replaced{$D: {\cal F}\Real \rightarrow Set$}{$D: \ {\cal F}\R \rightarrow Set$}, describing, for each length $l$, the set of paths of length $l$, and for each inclusion $Tr_k$ of $[0,l]$ into $[0,l']$ as $[k,k+l]$ (with $k\leq l'-l$), a restriction map from the length $l'$ paths to length $l\leq l'$ paths. 

Of course, for real continuous dynamical systems, we have an extra structure, which is that the set of paths on $X$ is a topological space, with the compact-open topology (equivalently, uniform convergence). Hence what we really have is a presheaf \replaced{$D:{\cal F}\Real \rightarrow \Top$}{$D: \ {\cal F}\R \rightarrow \Top$}, where $\Top$ is a convenient category of topological spaces. 

This means that $D$ belongs to a certain $\infty$-topos \cite{Lurie}, but the point for now is to notice that $D$ gives in particular a filtration of topological spaces from $\Real_-$ to $Top$, by setting $X_{-l}=D([0,l])$ and maps $X_{-l'}\rightarrow X_{-l}$ when $-l'\leq -l$,
that we can choose to be $D(Tr_0: \ X_{-l'} \rightarrow X_{-l})$. Hence this filtration looks at a dynamical system on $X$ through,  successively, its set of smaller initial paths. 

%\todo{Maybe we have to consider the translation invariant setup here?}

Similarly, taking the additive monoid $\N$, the same sheaf construction gives the ordinary category of graphs, and the filtration we have been giving starts with the topological space of maximal length in $X$, and carries on with paths of smaller length. 

The persistent homology for these filtrations is an interesting invariant of dynamical systems, and has a metric information induced by the metric on $X$ that we are planning on examining in future work. 
%}
%\cite{Carlsson}

%Some sort of bisimulation equivalence?

%Define "bar code" in such branching cases 

%\bibliographystyle{plain}
%\bibliography{BigBib}

%\section{Declarations}

%\paragraph{Ethical Approval}

%Not applicable.  

%\paragraph{Competing interests}

%The authors have no competing interests to declare that are relevant to the content of this article.

%\paragraph{Authors' contributions}

%The authors have contributed equally to the work reported here. 
%
%Part of the work was carried \replaced{out}{over} when Cameron Calk was affiliated with LIX, Ecole Polytechnique, and Universit\'e Claude Bernard Lyon 1.

%\paragraph{Funding}
 
%No funding. 
 
%\paragraph{Availability of data and materials}

%Not applicable. 

%\def\cprime{$'$}
%\bibliography{BigBib}

\def\cprime{$'$}

\end{document}

\clearpage

\quad

\vfill

\begin{footnotesize}

\bigskip
\auteur{Cameron Calk}{cameron.calk@inria.fr}
{\'Ecole Polytechnique \\
CNRS, Institut Polytechnique de Paris\\
1 rue Honor\'e d'Estienne d'Orves \\
91128 Palaiseau, France}
\bigskip

\auteur{Eric Goubault}{eric.goubault@polytechnique.edu}
{\'Ecole Polytechnique \\
CNRS, Institut Polytechnique de Paris\\
1 rue Honor\'e d'Estienne d'Orves \\
91128 Palaiseau, France}
\bigskip

\auteur{Philippe Malbos}{malbos@math.univ-lyon1.fr}
{Universit\'e Claude Bernard Lyon 1\\
CNRS UMR 5208, Institut Camille Jordan\\
43 Blvd. du 11 novembre 1918\\
F-69622 Villeurbanne cedex, France}
\bigskip

\end{footnotesize}

\vspace{1.5cm}

\begin{small}---\;\;\today\;\;-\;\;\hhmm\;\;---\end{small}

\end{document}

\section{Generation of a CW-complex filtration of the trace space from cubical sets}\label{S:CWFilt}

Cf. Ziema\'nski. We are going to describe algorithmically how to generate a filtration of CW-complexes that represent the evolution of the traces spaces, from an initial point of a cubical set (in a suitable class of cubical sets) to a varying end point, in order to implement efficiently the computation of the corresponding persistent homology. 

Let us take the following cubical set example, as an illustration of the algorithm. 

\begin{center}
\begin{tikzpicture}[scale=1]
\fill[fill=gray] (2,4) -- (4,4) -- (4,2) -- (2,2) -- cycle;
\draw[->] (0,0) -- (7,0);
\draw[->] (0,0) -- (0,7);
\draw (0,6) -- (6,6) -- (6,0);
\draw (0,2) -- (2,2) -- (2,0);
\draw (6,4) -- (4,4) -- (4,6);
\draw (0,4) -- (2,4) -- (2,6);
\draw (4,0) -- (4,2) -- (6,2);
\draw (1,1) node[right=-4] {$\scriptstyle A$};
\draw (3,1) node[right=-4] {$\scriptstyle B$};
\draw (5,1) node[right=-4] {$\scriptstyle C$};
\draw (1,3) node[right=-4] {$\scriptstyle D$};
\draw (5,3) node[right=-4] {$\scriptstyle E$};
\draw (1,5) node[right=-4] {$\scriptstyle F$};
\draw (3,5) node[right=-4] {$\scriptstyle G$};
\draw (5,5) node[right=-4] {$\scriptstyle H$};

\draw (1,-0.2) node[right=-4] {$\scriptstyle a$};
\draw (3,-0.2) node[right=-4] {$\scriptstyle b$};
\draw (5,-0.2) node[right=-4] {$\scriptstyle c$};

\draw (-0.2,1) node[right=-4] {$\scriptstyle d$};
\draw (1.8,1) node[right=-4] {$\scriptstyle e$};
\draw (3.8,1) node[right=-4] {$\scriptstyle f$};
\draw (5.8,1) node[right=-4] {$\scriptstyle g$};

\draw (1,1.8) node[right=-4] {$\scriptstyle h$};
\draw (3,1.8) node[right=-4] {$\scriptstyle i$};
\draw (5,1.8) node[right=-4] {$\scriptstyle j$};

\draw (-0.2,3) node[right=-4] {$\scriptstyle k$};
\draw (1.8,3) node[right=-4] {$\scriptstyle l$};
\draw (4.2,3) node[right=-4] {$\scriptstyle m$};
\draw (5.8,3) node[right=-4] {$\scriptstyle n$};

\draw (1,3.8) node[right=-4] {$\scriptstyle o$};
\draw (3,4.2) node[right=-4] {$\scriptstyle p$};
\draw (5,3.8) node[right=-4] {$\scriptstyle q$};

\draw (-0.2,5) node[right=-4] {$\scriptstyle r$};
\draw (1.8,5) node[right=-4] {$\scriptstyle s$};
\draw (3.8,5) node[right=-4] {$\scriptstyle t$};
\draw (5.8,5) node[right=-4] {$\scriptstyle u$};

\draw (1,5.8) node[right=-4] {$\scriptstyle v$};
\draw (3,5.8) node[right=-4] {$\scriptstyle w$};
\draw (5,5.8) node[right=-4] {$\scriptstyle x$};

\draw (7.5,0) node[right=-4] {$\scriptstyle t_0$};
\draw (0,7.5) node[above=-4] {$\scriptstyle t_1$};
    \end{tikzpicture}
\end{center}

The initial state is the lower corner of $A$. We then proceed to enumerate the tame path from the initial state, using a breadth first search (hence each step is parameterized by the length $l$ of such paths:
\begin{center}
\begin{tabular}{|c|c|c|c|c|c|c|c|c|}
\hline
l $\backslash$ $\sharp$ & 1 & 2 & 3 & 4 & 5 & 6 & 7 & 8 \\
\hline
1 & \begin{tikzpicture}[scale=0.1]
%      \draw[line width=0.3mm,color=red] (0,0) .. controls (.5,5) .. (5,5);
\fill[fill=gray] (2,4) -- (4,4) -- (4,2) -- (2,2) -- cycle;
%\fill[fill=gray] (3,1) -- (4,1) -- (4,2) -- (3,2) -- cycle;
% \draw (1.5,3.5) node {$0$};
% \draw (3.5,1.5) node {$1$};
%\draw (1.5,-.5) node[below=-5] {$\scriptscriptstyle u_0$};
%\draw (3.5,-.5) node[below=-5] {$\scriptscriptstyle s_0$};
%\draw (-.5,1.5) node[left=-6] {$\scriptscriptstyle u_1$};
%\draw (-.5,3.5) node[left=-6] {$\scriptscriptstyle s_1$};
\draw[->] (0,0) -- (7,0);
\draw[->] (0,0) -- (0,7);
\draw (0,6) -- (6,6) -- (6,0);
\draw (0,2) -- (2,2) -- (2,0);
\draw (6,4) -- (4,4) -- (4,6);
\draw (0,4) -- (2,4) -- (2,6);
\draw (4,0) -- (4,2) -- (6,2);

\draw (1,-0.2) node[right=-4] {$\scriptscriptstyle a$};
\draw[thick,color=blue] (0,0) -- (2,0);

\draw (7.5,0) node[right=-4] {$\scriptscriptstyle t_0$};
\draw (0,7.5) node[above=-4] {$\scriptscriptstyle t_1$};
    \end{tikzpicture}
& 
\begin{tikzpicture}[scale=0.1]
%      \draw[line width=0.3mm,color=red] (0,0) .. controls (.5,5) .. (5,5);
\fill[fill=gray] (2,4) -- (4,4) -- (4,2) -- (2,2) -- cycle;
%\fill[fill=gray] (3,1) -- (4,1) -- (4,2) -- (3,2) -- cycle;
% \draw (1.5,3.5) node {$0$};
% \draw (3.5,1.5) node {$1$};
%\draw (1.5,-.5) node[below=-5] {$\scriptscriptstyle u_0$};
%\draw (3.5,-.5) node[below=-5] {$\scriptscriptstyle s_0$};
%\draw (-.5,1.5) node[left=-6] {$\scriptscriptstyle u_1$};
%\draw (-.5,3.5) node[left=-6] {$\scriptscriptstyle s_1$};
\draw[->] (0,0) -- (7,0);
\draw[->] (0,0) -- (0,7);
\draw (0,6) -- (6,6) -- (6,0);
\draw (0,2) -- (2,2) -- (2,0);
\draw (6,4) -- (4,4) -- (4,6);
\draw (0,4) -- (2,4) -- (2,6);
\draw (4,0) -- (4,2) -- (6,2);

\draw (-0.2,1) node[right=-4] {$\scriptscriptstyle d$};
\draw[thick,color=blue] (0,0) -- (0,2);

\draw (7.5,0) node[right=-4] {$\scriptscriptstyle t_0$};
\draw (0,7.5) node[above=-4] {$\scriptscriptstyle t_1$};
    \end{tikzpicture}
& & & & & & \\
\hline 
2 & 
\begin{tikzpicture}[scale=0.1]
%      \draw[line width=0.3mm,color=red] (0,0) .. controls (.5,5) .. (5,5);
\fill[fill=gray] (2,4) -- (4,4) -- (4,2) -- (2,2) -- cycle;
%\fill[fill=gray] (3,1) -- (4,1) -- (4,2) -- (3,2) -- cycle;
% \draw (1.5,3.5) node {$0$};
% \draw (3.5,1.5) node {$1$};
%\draw (1.5,-.5) node[below=-5] {$\scriptscriptstyle u_0$};
%\draw (3.5,-.5) node[below=-5] {$\scriptscriptstyle s_0$};
%\draw (-.5,1.5) node[left=-6] {$\scriptscriptstyle u_1$};
%\draw (-.5,3.5) node[left=-6] {$\scriptscriptstyle s_1$};
\draw[->] (0,0) -- (7,0);
\draw[->] (0,0) -- (0,7);
\draw (0,6) -- (6,6) -- (6,0);
\draw (0,2) -- (2,2) -- (2,0);
\draw (6,4) -- (4,4) -- (4,6);
\draw (0,4) -- (2,4) -- (2,6);
\draw (4,0) -- (4,2) -- (6,2);

\draw (1,-0.2) node[right=-4] {$\scriptscriptstyle a$};
\draw[thick,color=blue] (0,0) -- (2,0);
\draw (3,-0.2) node[right=-4] {$\scriptscriptstyle b$};
\draw[thick,color=blue] (2,0) -- (4,0);

\draw (7.5,0) node[right=-4] {$\scriptscriptstyle t_0$};
\draw (0,7.5) node[above=-4] {$\scriptscriptstyle t_1$};
    \end{tikzpicture}
&
\begin{tikzpicture}[scale=0.1]
%      \draw[line width=0.3mm,color=red] (0,0) .. controls (.5,5) .. (5,5);
\fill[fill=gray] (2,4) -- (4,4) -- (4,2) -- (2,2) -- cycle;
%\fill[fill=gray] (3,1) -- (4,1) -- (4,2) -- (3,2) -- cycle;
% \draw (1.5,3.5) node {$0$};
% \draw (3.5,1.5) node {$1$};
%\draw (1.5,-.5) node[below=-5] {$\scriptscriptstyle u_0$};
%\draw (3.5,-.5) node[below=-5] {$\scriptscriptstyle s_0$};
%\draw (-.5,1.5) node[left=-6] {$\scriptscriptstyle u_1$};
%\draw (-.5,3.5) node[left=-6] {$\scriptscriptstyle s_1$};
\draw[->] (0,0) -- (7,0);
\draw[->] (0,0) -- (0,7);
\draw (0,6) -- (6,6) -- (6,0);
\draw (0,2) -- (2,2) -- (2,0);
\draw (6,4) -- (4,4) -- (4,6);
\draw (0,4) -- (2,4) -- (2,6);
\draw (4,0) -- (4,2) -- (6,2);

\draw (1,-0.2) node[right=-4] {$\scriptscriptstyle a$};
\draw[thick,color=blue] (0,0) -- (2,0);
\draw (1.8,1) node[right=-4] {$\scriptscriptstyle e$};
\draw[thick,color=blue] (2,0) -- (2,2);

\draw (7.5,0) node[right=-4] {$\scriptscriptstyle t_0$};
\draw (0,7.5) node[above=-4] {$\scriptscriptstyle t_1$};
    \end{tikzpicture}
& 
\begin{tikzpicture}[scale=0.1]
%      \draw[line width=0.3mm,color=red] (0,0) .. controls (.5,5) .. (5,5);
\fill[fill=gray] (2,4) -- (4,4) -- (4,2) -- (2,2) -- cycle;
%\fill[fill=gray] (3,1) -- (4,1) -- (4,2) -- (3,2) -- cycle;
% \draw (1.5,3.5) node {$0$};
% \draw (3.5,1.5) node {$1$};
%\draw (1.5,-.5) node[below=-5] {$\scriptscriptstyle u_0$};
%\draw (3.5,-.5) node[below=-5] {$\scriptscriptstyle s_0$};
%\draw (-.5,1.5) node[left=-6] {$\scriptscriptstyle u_1$};
%\draw (-.5,3.5) node[left=-6] {$\scriptscriptstyle s_1$};
\draw[->] (0,0) -- (7,0);
\draw[->] (0,0) -- (0,7);
\draw (0,6) -- (6,6) -- (6,0);
\draw (0,2) -- (2,2) -- (2,0);
\draw (6,4) -- (4,4) -- (4,6);
\draw (0,4) -- (2,4) -- (2,6);
\draw (4,0) -- (4,2) -- (6,2);
\draw[color=black] (1,1) node[right=-4] {$\scriptscriptstyle A$};
\draw[thick,color=blue,fill=blue] (0,0) -- (0,2) -- (2,2) -- (2,0) -- cycle;
\draw[color=black] (1,1) node[right=-4] {$\scriptscriptstyle A$};

\draw (7.5,0) node[right=-4] {$\scriptscriptstyle t_0$};
\draw (0,7.5) node[above=-4] {$\scriptscriptstyle t_1$};
    \end{tikzpicture}
& 
\begin{tikzpicture}[scale=0.1]
%      \draw[line width=0.3mm,color=red] (0,0) .. controls (.5,5) .. (5,5);
\fill[fill=gray] (2,4) -- (4,4) -- (4,2) -- (2,2) -- cycle;
%\fill[fill=gray] (3,1) -- (4,1) -- (4,2) -- (3,2) -- cycle;
% \draw (1.5,3.5) node {$0$};
% \draw (3.5,1.5) node {$1$};
%\draw (1.5,-.5) node[below=-5] {$\scriptscriptstyle u_0$};
%\draw (3.5,-.5) node[below=-5] {$\scriptscriptstyle s_0$};
%\draw (-.5,1.5) node[left=-6] {$\scriptscriptstyle u_1$};
%\draw (-.5,3.5) node[left=-6] {$\scriptscriptstyle s_1$};
\draw[->] (0,0) -- (7,0);
\draw[->] (0,0) -- (0,7);
\draw (0,6) -- (6,6) -- (6,0);
\draw (0,2) -- (2,2) -- (2,0);
\draw (6,4) -- (4,4) -- (4,6);
\draw (0,4) -- (2,4) -- (2,6);
\draw (4,0) -- (4,2) -- (6,2);

\draw (-0.2,1) node[right=-4] {$\scriptscriptstyle d$};
\draw[thick,color=blue] (0,0) -- (0,2);
\draw (1,1.8) node[right=-4] {$\scriptscriptstyle h$};
\draw[thick,color=blue] (0,2) -- (2,2);

\draw (7.5,0) node[right=-4] {$\scriptscriptstyle t_0$};
\draw (0,7.5) node[above=-4] {$\scriptscriptstyle t_1$};
    \end{tikzpicture}
& 
\begin{tikzpicture}[scale=0.1]
%      \draw[line width=0.3mm,color=red] (0,0) .. controls (.5,5) .. (5,5);
\fill[fill=gray] (2,4) -- (4,4) -- (4,2) -- (2,2) -- cycle;
%\fill[fill=gray] (3,1) -- (4,1) -- (4,2) -- (3,2) -- cycle;
% \draw (1.5,3.5) node {$0$};
% \draw (3.5,1.5) node {$1$};
%\draw (1.5,-.5) node[below=-5] {$\scriptscriptstyle u_0$};
%\draw (3.5,-.5) node[below=-5] {$\scriptscriptstyle s_0$};
%\draw (-.5,1.5) node[left=-6] {$\scriptscriptstyle u_1$};
%\draw (-.5,3.5) node[left=-6] {$\scriptscriptstyle s_1$};
\draw[->] (0,0) -- (7,0);
\draw[->] (0,0) -- (0,7);
\draw (0,6) -- (6,6) -- (6,0);
\draw (0,2) -- (2,2) -- (2,0);
\draw (6,4) -- (4,4) -- (4,6);
\draw (0,4) -- (2,4) -- (2,6);
\draw (4,0) -- (4,2) -- (6,2);

\draw (-0.2,1) node[right=-4] {$\scriptscriptstyle d$};
\draw[thick,color=blue] (0,0) -- (0,2);
\draw (-0.2,3) node[right=-4] {$\scriptscriptstyle k$};
\draw[thick,color=blue] (0,2) -- (0,4);

\draw (7.5,0) node[right=-4] {$\scriptscriptstyle t_0$};
\draw (0,7.5) node[above=-4] {$\scriptscriptstyle t_1$};
    \end{tikzpicture}
& & & \\
\hline 
3 & 
\begin{tikzpicture}[scale=0.1]
%      \draw[line width=0.3mm,color=red] (0,0) .. controls (.5,5) .. (5,5);
\fill[fill=gray] (2,4) -- (4,4) -- (4,2) -- (2,2) -- cycle;
%\fill[fill=gray] (3,1) -- (4,1) -- (4,2) -- (3,2) -- cycle;
% \draw (1.5,3.5) node {$0$};
% \draw (3.5,1.5) node {$1$};
%\draw (1.5,-.5) node[below=-5] {$\scriptscriptstyle u_0$};
%\draw (3.5,-.5) node[below=-5] {$\scriptscriptstyle s_0$};
%\draw (-.5,1.5) node[left=-6] {$\scriptscriptstyle u_1$};
%\draw (-.5,3.5) node[left=-6] {$\scriptscriptstyle s_1$};
\draw[->] (0,0) -- (7,0);
\draw[->] (0,0) -- (0,7);
\draw (0,6) -- (6,6) -- (6,0);
\draw (0,2) -- (2,2) -- (2,0);
\draw (6,4) -- (4,4) -- (4,6);
\draw (0,4) -- (2,4) -- (2,6);
\draw (4,0) -- (4,2) -- (6,2);

\draw (1,-0.2) node[right=-4] {$\scriptscriptstyle a$};
\draw[thick,color=blue] (0,0) -- (2,0);
\draw (3,-0.2) node[right=-4] {$\scriptscriptstyle b$};
\draw[thick,color=blue] (2,0) -- (4,0);
\draw (3.8,1) node[right=-4] {$\scriptscriptstyle f$};
\draw[thick,color=blue] (4,0) -- (4,2);

\draw (7.5,0) node[right=-4] {$\scriptscriptstyle t_0$};
\draw (0,7.5) node[above=-4] {$\scriptscriptstyle t_1$};
    \end{tikzpicture}

&
\begin{tikzpicture}[scale=0.1]
%      \draw[line width=0.3mm,color=red] (0,0) .. controls (.5,5) .. (5,5);
\fill[fill=gray] (2,4) -- (4,4) -- (4,2) -- (2,2) -- cycle;
%\fill[fill=gray] (3,1) -- (4,1) -- (4,2) -- (3,2) -- cycle;
% \draw (1.5,3.5) node {$0$};
% \draw (3.5,1.5) node {$1$};
%\draw (1.5,-.5) node[below=-5] {$\scriptscriptstyle u_0$};
%\draw (3.5,-.5) node[below=-5] {$\scriptscriptstyle s_0$};
%\draw (-.5,1.5) node[left=-6] {$\scriptscriptstyle u_1$};
%\draw (-.5,3.5) node[left=-6] {$\scriptscriptstyle s_1$};
\draw[->] (0,0) -- (7,0);
\draw[->] (0,0) -- (0,7);
\draw (0,6) -- (6,6) -- (6,0);
\draw (0,2) -- (2,2) -- (2,0);
\draw (6,4) -- (4,4) -- (4,6);
\draw (0,4) -- (2,4) -- (2,6);
\draw (4,0) -- (4,2) -- (6,2);

\draw (1,-0.2) node[right=-4] {$\scriptscriptstyle a$};
\draw[thick,color=blue] (0,0) -- (2,0);
\draw (1.8,1) node[right=-4] {$\scriptscriptstyle e$};
\draw[thick,color=blue] (2,0) -- (2,2);
\draw (1.8,3) node[right=-4] {$\scriptscriptstyle l$};
\draw[thick,color=blue] (2,2) -- (2,4);

\draw (7.5,0) node[right=-4] {$\scriptscriptstyle t_0$};
\draw (0,7.5) node[above=-4] {$\scriptscriptstyle t_1$};
    \end{tikzpicture}
& 
\begin{tikzpicture}[scale=0.1]
%      \draw[line width=0.3mm,color=red] (0,0) .. controls (.5,5) .. (5,5);
\fill[fill=gray] (2,4) -- (4,4) -- (4,2) -- (2,2) -- cycle;
%\fill[fill=gray] (3,1) -- (4,1) -- (4,2) -- (3,2) -- cycle;
% \draw (1.5,3.5) node {$0$};
% \draw (3.5,1.5) node {$1$};
%\draw (1.5,-.5) node[below=-5] {$\scriptscriptstyle u_0$};
%\draw (3.5,-.5) node[below=-5] {$\scriptscriptstyle s_0$};
%\draw (-.5,1.5) node[left=-6] {$\scriptscriptstyle u_1$};
%\draw (-.5,3.5) node[left=-6] {$\scriptscriptstyle s_1$};
\draw[->] (0,0) -- (7,0);
\draw[->] (0,0) -- (0,7);
\draw (0,6) -- (6,6) -- (6,0);
\draw (0,2) -- (2,2) -- (2,0);
\draw (6,4) -- (4,4) -- (4,6);
\draw (0,4) -- (2,4) -- (2,6);
\draw (4,0) -- (4,2) -- (6,2);
\draw[color=black] (1,1) node[right=-4] {$\scriptscriptstyle A$};
\draw[thick,color=blue,fill=blue] (0,0) -- (0,2) -- (2,2) -- (2,0) -- cycle;
\draw[color=black] (1,1) node[right=-4] {$\scriptscriptstyle A$};
\draw (1.8,3) node[right=-4] {$\scriptscriptstyle l$};
\draw[thick,color=blue] (2,2) -- (2,4);

\draw (7.5,0) node[right=-4] {$\scriptscriptstyle t_0$};
\draw (0,7.5) node[above=-4] {$\scriptscriptstyle t_1$};
    \end{tikzpicture}
& 
\begin{tikzpicture}[scale=0.1]
%      \draw[line width=0.3mm,color=red] (0,0) .. controls (.5,5) .. (5,5);
\fill[fill=gray] (2,4) -- (4,4) -- (4,2) -- (2,2) -- cycle;
%\fill[fill=gray] (3,1) -- (4,1) -- (4,2) -- (3,2) -- cycle;
% \draw (1.5,3.5) node {$0$};
% \draw (3.5,1.5) node {$1$};
%\draw (1.5,-.5) node[below=-5] {$\scriptscriptstyle u_0$};
%\draw (3.5,-.5) node[below=-5] {$\scriptscriptstyle s_0$};
%\draw (-.5,1.5) node[left=-6] {$\scriptscriptstyle u_1$};
%\draw (-.5,3.5) node[left=-6] {$\scriptscriptstyle s_1$};
\draw[->] (0,0) -- (7,0);
\draw[->] (0,0) -- (0,7);
\draw (0,6) -- (6,6) -- (6,0);
\draw (0,2) -- (2,2) -- (2,0);
\draw (6,4) -- (4,4) -- (4,6);
\draw (0,4) -- (2,4) -- (2,6);
\draw (4,0) -- (4,2) -- (6,2);

\draw (-0.2,1) node[right=-4] {$\scriptscriptstyle d$};
\draw[thick,color=blue] (0,0) -- (0,2);
\draw (1,1.8) node[right=-4] {$\scriptscriptstyle h$};
\draw[thick,color=blue] (0,2) -- (2,2);
\draw (1.8,3) node[right=-4] {$\scriptscriptstyle l$};
\draw[thick,color=blue] (2,2) -- (2,4);

\draw (7.5,0) node[right=-4] {$\scriptscriptstyle t_0$};
\draw (0,7.5) node[above=-4] {$\scriptscriptstyle t_1$};
    \end{tikzpicture}
&    

\begin{tikzpicture}[scale=0.1]
%      \draw[line width=0.3mm,color=red] (0,0) .. controls (.5,5) .. (5,5);
\fill[fill=gray] (2,4) -- (4,4) -- (4,2) -- (2,2) -- cycle;
%\fill[fill=gray] (3,1) -- (4,1) -- (4,2) -- (3,2) -- cycle;
% \draw (1.5,3.5) node {$0$};
% \draw (3.5,1.5) node {$1$};
%\draw (1.5,-.5) node[below=-5] {$\scriptscriptstyle u_0$};
%\draw (3.5,-.5) node[below=-5] {$\scriptscriptstyle s_0$};
%\draw (-.5,1.5) node[left=-6] {$\scriptscriptstyle u_1$};
%\draw (-.5,3.5) node[left=-6] {$\scriptscriptstyle s_1$};
\draw[->] (0,0) -- (7,0);
\draw[->] (0,0) -- (0,7);
\draw (0,6) -- (6,6) -- (6,0);
\draw (0,2) -- (2,2) -- (2,0);
\draw (6,4) -- (4,4) -- (4,6);
\draw (0,4) -- (2,4) -- (2,6);
\draw (4,0) -- (4,2) -- (6,2);

\draw (1,-0.2) node[right=-4] {$\scriptscriptstyle a$};
\draw[thick,color=blue] (0,0) -- (2,0);
\draw (1.8,1) node[right=-4] {$\scriptscriptstyle e$};
\draw[thick,color=blue] (2,0) -- (2,2);
\draw (3,1.8) node[right=-4] {$\scriptscriptstyle i$};
\draw[thick,color=blue] (2,2) -- (4,2);

\draw (7.5,0) node[right=-4] {$\scriptscriptstyle t_0$};
\draw (0,7.5) node[above=-4] {$\scriptscriptstyle t_1$};
    \end{tikzpicture}
& 
\begin{tikzpicture}[scale=0.1]
%      \draw[line width=0.3mm,color=red] (0,0) .. controls (.5,5) .. (5,5);
\fill[fill=gray] (2,4) -- (4,4) -- (4,2) -- (2,2) -- cycle;
%\fill[fill=gray] (3,1) -- (4,1) -- (4,2) -- (3,2) -- cycle;
% \draw (1.5,3.5) node {$0$};
% \draw (3.5,1.5) node {$1$};
%\draw (1.5,-.5) node[below=-5] {$\scriptscriptstyle u_0$};
%\draw (3.5,-.5) node[below=-5] {$\scriptscriptstyle s_0$};
%\draw (-.5,1.5) node[left=-6] {$\scriptscriptstyle u_1$};
%\draw (-.5,3.5) node[left=-6] {$\scriptscriptstyle s_1$};
\draw[->] (0,0) -- (7,0);
\draw[->] (0,0) -- (0,7);
\draw (0,6) -- (6,6) -- (6,0);
\draw (0,2) -- (2,2) -- (2,0);
\draw (6,4) -- (4,4) -- (4,6);
\draw (0,4) -- (2,4) -- (2,6);
\draw (4,0) -- (4,2) -- (6,2);
\draw[color=black] (1,1) node[right=-4] {$\scriptscriptstyle A$};
\draw[thick,color=blue,fill=blue] (0,0) -- (0,2) -- (2,2) -- (2,0) -- cycle;
\draw[color=black] (1,1) node[right=-4] {$\scriptscriptstyle A$};
\draw (3,1.8) node[right=-4] {$\scriptscriptstyle i$};
\draw[thick,color=blue] (2,2) -- (4,2);

\draw (7.5,0) node[right=-4] {$\scriptscriptstyle t_0$};
\draw (0,7.5) node[above=-4] {$\scriptscriptstyle t_1$};
    \end{tikzpicture}
& 
\begin{tikzpicture}[scale=0.1]
%      \draw[line width=0.3mm,color=red] (0,0) .. controls (.5,5) .. (5,5);
\fill[fill=gray] (2,4) -- (4,4) -- (4,2) -- (2,2) -- cycle;
%\fill[fill=gray] (3,1) -- (4,1) -- (4,2) -- (3,2) -- cycle;
% \draw (1.5,3.5) node {$0$};
% \draw (3.5,1.5) node {$1$};
%\draw (1.5,-.5) node[below=-5] {$\scriptscriptstyle u_0$};
%\draw (3.5,-.5) node[below=-5] {$\scriptscriptstyle s_0$};
%\draw (-.5,1.5) node[left=-6] {$\scriptscriptstyle u_1$};
%\draw (-.5,3.5) node[left=-6] {$\scriptscriptstyle s_1$};
\draw[->] (0,0) -- (7,0);
\draw[->] (0,0) -- (0,7);
\draw (0,6) -- (6,6) -- (6,0);
\draw (0,2) -- (2,2) -- (2,0);
\draw (6,4) -- (4,4) -- (4,6);
\draw (0,4) -- (2,4) -- (2,6);
\draw (4,0) -- (4,2) -- (6,2);

\draw (-0.2,1) node[right=-4] {$\scriptscriptstyle d$};
\draw[thick,color=blue] (0,0) -- (0,2);
\draw (1,1.8) node[right=-4] {$\scriptscriptstyle h$};
\draw[thick,color=blue] (0,2) -- (2,2);
\draw (3,1.8) node[right=-4] {$\scriptscriptstyle i$};
\draw[thick,color=blue] (2,2) -- (4,2);

\draw (7.5,0) node[right=-4] {$\scriptscriptstyle t_0$};
\draw (0,7.5) node[above=-4] {$\scriptscriptstyle t_1$};
    \end{tikzpicture} 
    & \ldots
%    \begin{tikzpicture}[scale=0.1]
%\fill[fill=gray] (2,4) -- (4,4) -- (4,2) -- (2,2) -- cycle;
%\draw[->] (0,0) -- (7,0);
%\draw[->] (0,0) -- (0,7);
%\draw (0,6) -- (6,6) -- (6,0);
%\draw (0,2) -- (2,2) -- (2,0);
%\draw (6,4) -- (4,4) -- (4,6);
%\draw (0,4) -- (2,4) -- (2,6);
%\draw (4,0) -- (4,2) -- (6,2);
%
%\draw (1,-0.2) node[right=-4] {$\scriptscriptstyle a$};
%\draw[thick,color=blue] (0,0) -- (2,0);
%\draw (3,-0.2) node[right=-4] {$\scriptscriptstyle b$};
%\draw[thick,color=blue] (2,0) -- (4,0);
%\draw (5,-0.2) node[right=-4] {$\scriptscriptstyle c$};
%\draw[thick,color=blue] (4,0) -- (6,0);
%
%\draw (7.5,0) node[right=-4] {$\scriptscriptstyle t_0$};
%\draw (0,7.5) node[above=-4] {$\scriptscriptstyle t_1$};
%    \end{tikzpicture}

    \\
\hline 
4 & \ldots & & & & & & & \\
\hline 
5 & & & & & & & & \\
\hline 
6 & & & & & & & & \\
\hline 
\end{tabular}
\end{center}

Each time we do the next step of the breadth-first traversal, we are able to construct the extension maps from CW-complexes arising for paths of length $l$ into CW-complexes for paths of length $l+1$. 
The length encodes also the time at which some cells appear or disappear in the graded CW-complex we obtain the CW-complex with labels and time of appearing in Figure \ref{fig:gradedCW}, at least for time 1 and 2. 

Note that in the CW-complex representation of the successive trace spaces along time, paths that end in different final states are necessarily in different connected components. We label elements of the CW-complex as $(p,f)$ where $p$ is the path that is represented by the corresponding cell, and where $f$ is the final state reached by this path. 

\begin{figure}
\begin{center}
\begin{tikzpicture}[scale=1]
%\fill[fill=gray] (2,4) -- (4,4) -- (4,2) -- (2,2) -- cycle;
%\draw[->] (0,0) -- (7,0);
%\draw[->] (0,0) -- (0,7);
%\draw (0,6) -- (6,6) -- (6,0);
%\draw (0,2) -- (2,2) -- (2,0);
%\draw (6,4) -- (4,4) -- (4,6);
%\draw (0,4) -- (2,4) -- (2,6);
%\draw (4,0) -- (4,2) -- (6,2);
\tikzstyle{every node}=[draw,shape=circle]

\draw[xshift=-1cm] (0,0) node[circle,fill,inner sep=1pt,label=above:$a$](a){} -- (1,0) 
node[circle,fill,inner sep=1pt,label=below:$u_1$](b){};
\draw[xshift=-1cm] (2,0) node[circle,fill,inner sep=1pt,label=below:$v_1$](c) {} -- (3,0) 
node[circle,fill,inner sep=1pt,label=below:$u_2$](d){};
\draw[xshift=-1cm] (4,0) node[circle,fill,inner sep=1pt,label=below:$v_2$](e){} -- (5,0) 
node[circle,fill,inner sep=1pt,label=above:x](f){};
\draw[xshift=-1cm] (5,0) node[circle,fill,inner sep=1pt] {} -- (6,0) node[circle,fill,inner sep=1pt,label=above:b](B){};
    \end{tikzpicture} 
\caption{Graded, labeled CW-complex from Example \ref{??}}
\label{fig:gradedCW}
\end{center}
\end{figure} 

Given the labelling of the final states given in Figure \ref{fig:labellingfinalstates}

The problem is that simplices are split each time there might be several futures (there is no free lunch, we are trying here to represent a multidimensional persistence problem as a unidimensional one). A possibility is to duplicate each simplices to be rather the set of paths that contain the current path, as follows: 

(...)

\begin{figure}
\begin{center}
\begin{tikzpicture}[scale=1]
%      \draw[line width=0.3mm,color=red] (0,0) .. controls (.5,5) .. (5,5);
\fill[fill=gray] (2,4) -- (4,4) -- (4,2) -- (2,2) -- cycle;
%\fill[fill=gray] (3,1) -- (4,1) -- (4,2) -- (3,2) -- cycle;
% \draw (1.5,3.5) node {$0$};
% \draw (3.5,1.5) node {$1$};
%\draw (1.5,-.5) node[below=-5] {$\scriptstyle u_0$};
%\draw (3.5,-.5) node[below=-5] {$\scriptstyle s_0$};
%\draw (-.5,1.5) node[left=-6] {$\scriptstyle u_1$};
%\draw (-.5,3.5) node[left=-6] {$\scriptstyle s_1$};
\draw[->] (0,0) -- (7,0);
\draw[->] (0,0) -- (0,7);
\draw (0,6) -- (6,6) -- (6,0);
\draw (0,2) -- (2,2) -- (2,0);
\draw (6,4) -- (4,4) -- (4,6);
\draw (0,4) -- (2,4) -- (2,6);
\draw (4,0) -- (4,2) -- (6,2);

%\draw (1,-0.2) node[right=-4] {$\scriptscriptstyle a$};
\draw %[thick,color=blue] 
(0,0) -- (2,0);
\draw (2,-0.2) node[right=-4] {$\scriptscriptstyle 1$};
\draw (4,-0.2) node[right=-4] {$\scriptscriptstyle 3$};
\draw (6,-0.2) node[right=-4] {$\scriptscriptstyle 6$};
\draw (2,-0.2) node[right=-4] {$\scriptscriptstyle 1$};
\draw (-0.2,2) node[right=-4] {$\scriptscriptstyle 2$};
\draw (-0.2,4) node[right=-4] {$\scriptscriptstyle 5$};
\draw (-0.2,6) node[right=-4] {$\scriptscriptstyle 9$};

\draw (1.8,1.8) node[right=-4] {$\scriptscriptstyle 4$};
\draw (3.8,1.8) node[right=-4] {$\scriptscriptstyle 7$};
\draw (6.2,1.8) node[right=-4] {$\scriptscriptstyle 10$};
\draw (4.2,4.2) node[right=-4] {$\scriptscriptstyle 11$};
\draw (6.2,4.2) node[right=-4] {$\scriptscriptstyle 13$};
\draw (1.8,4.2) node[right=-4] {$\scriptscriptstyle 8$};
\draw (1.8,6.2) node[right=-4] {$\scriptscriptstyle 12$};
\draw (4.2,6.2) node[right=-4] {$\scriptscriptstyle 14$};
\draw (6.2,6.2) node[right=-4] {$\scriptscriptstyle 15$};

\draw (7.5,0) node[right=-4] {$\scriptstyle t_0$};
\draw (0,7.5) node[above=-4] {$\scriptstyle t_1$};
    \end{tikzpicture}
\caption{Labelling of the final states}
\end{center}
\label{fig:labellingfinalstates} 
\end{figure}

Particular case of NPC cubical complexes: enough to take \replaced{``}{"}greedy paths", \emph{i.e.} $AlGx$, $AiEu$ and suffixes. 

It may be possible to compute the full \replaced{``}{"}natural homology" with this idea of persistence by starting with all points in the cubical complexes and extend them in both directions. 

Some notes on Gudhi, for a potential implementation: 

\bibliographystyle{plain}
\bibliography{BigBib}

\vfill

\begin{footnotesize}

\bigskip
\auteur{Cameron Calk}{cameron.calk@inria.fr}
{\'Ecole Polytechnique \\
B\^atiment Alan Turing \\
1 rue Honor\'e d'Estienne d'Orves \\
91128 Palaiseau, France}
\bigskip

\auteur{Eric Goubault}{eric.goubault@polytechnique.edu}
{\'Ecole Polytechnique \\
B\^atiment Alan Turing \\
1 rue Honor\'e d'Estienne d'Orves \\
91128 Palaiseau, France}
\bigskip

\auteur{Philippe Malbos}{malbos@math.univ-lyon1.fr}
{Univ Lyon, Universit\'e Claude Bernard Lyon 1\\
CNRS UMR 5208, Institut Camille Jordan\\
43 Blvd. du 11 novembre 1918\\
F-69622 Villeurbanne cedex, France}

\end{footnotesize}

\vspace{1.5cm}

\begin{small}---\;\;\today\;\;-\;\;\hhmm\;\;---\end{small}

\end{document}